\documentclass[12pt]{article}
\usepackage[margin=0.7in]{geometry}
\usepackage{amssymb,amsmath,amsthm}
\usepackage{graphicx,epstopdf,color}
\usepackage{esint}

\usepackage[alphabetic]{amsrefs}

\usepackage[latin1]{inputenc}
\newcommand{\R}{\mathbb{R}}

\newcommand{\N}{\mathbb{N}}

\usepackage{miama}
\usepackage[T1]{fontenc}
\usepackage{mathtools}

\allowdisplaybreaks

\numberwithin{equation}{section}

\newtheorem{thm}{Theorem}[section]
\newtheorem{cor}[thm]{Corollary}
\newtheorem{lem}[thm]{Lemma}
\newtheorem{prop}[thm]{Proposition}
\newtheorem{defn}[thm]{Definition}
\newtheorem{rem}[thm]{Remark}

\renewcommand{\leq}{\leqslant}
\renewcommand{\le}{\leqslant}
\renewcommand{\geq}{\geqslant}
\renewcommand{\ge}{\geqslant}

\usepackage{authblk}

\begin{document}

\title{Density estimates for a nonlocal variational model \\
with a degenerate double-well potential \\
via the Sobolev inequality}


\author{Serena Dipierro, Alberto Farina, Giovanni Giacomin and Enrico Valdinoci
\thanks{Serena Dipierro, Giovanni Giacomin and Enrico Valdinoci:
Department of Mathematics and Statistics,
University of Western Australia, 35 Stirling Highway,
WA6009 Crawley, Australia.\\
Alberto Farina: LAMFA, UMR CNRS 7352, 
	Universit\'e Picardie Jules Verne
	33, rue St Leu, 80039 Amiens, France.\\
{\tt serena.dipierro@uwa.edu.au, alberto.farina@u-picardie.fr
giovanni.giacomin@research.uwa.edu.au,
enrico.valdinoci@uwa.edu.au\\
SD, GG and EV are members of Australian Mathematical Society.
Supported by the Australian Future Fellowship
FT230100333 ``New perspectives on nonlocal equations''
and the Australian Laureate Fellowship FL190100081 ``Minimal
surfaces, free boundaries and partial differential equations''. 
Helpful discussions with Francesco De Pas and Jack Thompson are acknowledged. AF and GG would like to thank the Institute Henri Poincar\'e where part of this research was performed.}}}

\maketitle

\begin{abstract}
We provide density estimates for level sets of minimizers of the energy 
\begin{equation*}
\frac{1}{2}\int_{\Omega}\int_{\Omega}\frac{\left|u(x)-u(y)\right|^p}{\left|x-y\right|^{n+sp}}\,dx\,dy +\int_{\Omega}\int_{\R^n\setminus \Omega}\frac{\left|u(x)-u(y)\right|^p}{\left|x-y\right|^{n+sp}}\,dx\,dy+\int_{\Omega}W(u(x))\,dx
\end{equation*}
where~$p\in (1,+\infty)$, $s \in \left(0,\frac{1}{p}\right)$ and~$W$ is a double-well potential with polynomial growth~$m\in \left[p,+\infty\right)$ from the minima.

These kinds of potentials are ``degenerate'', since they detach ``slowly'' from the minima, therefore they provide additional difficulties if one wishes to determine the relative sizes of the ``layers'' and the ``pure phases''. To overcome these challenges, we introduce new barriers allowing us to rely on the fractional Sobolev inequality and on a suitable iteration method.

The proofs presented here are robust enough to consider the case of quasilinear nonlocal equations driven by the fractional $p$-Laplacian, but our results are new even for the case~$p=2$.
\end{abstract}

\bigskip


	\smallskip
	
\bigskip

\section{Introduction}

Phase coexistence models were introduced in physics to study the interfaces between regions associated with different values of a suitable state parameter, such as the density of a fluid in capillarity phenomena~\cite{MR0523642}, the fraction of volume occupied by two different materials in a non homogeneous system~\cite{cahn1958free}, or the density of the superfluid component of helium~\cite{ginzburg1958theory}. Since these early works, phase coexistence models have been developed to describe a variety of physical systems characterized by the interaction of two or more components. 

All these models are described by some region~$\Omega\subset\R^n$ and a state parameter~$u:\Omega\to [-1,1]$, where the ``pure phases'' correspond in our setting to the values~$1$ and~$-1$, and the phase separation is induced by the minimization of a suitable energy functional. 

The prototype of such an energy consists of a potential and an interaction term. The role of the potential term is to push the state parameter towards the pure phases, while the interaction term discourages the production of unneeded interfaces by suitably ``shaping'' the portions of the domain in which the value of the state parameter is either~$1$ or~$-1$ (or is ``sufficiently close'' to these values).

In the local setting, the interaction energy is proportional to a gradient term of~$L^p$-type with~$p\in(1,+\infty)$, see for instance~\cite{MR0618549,MR0733897,bouchitte1990singular,6,petrosyan2005geometric}. In order to capture long-range interactions, in the nonlocal framework the~$L^p$-norm of the gradient has been replaced by the Gagliardo seminorm, see for instance~\cite{4,5,3}. For further references regarding local and nonlocal phase separation models we refer to the survey~\cite{dipierro2023some}.    

The potential energy is given by the~$L^1$-norm of a ``double-well'' function~$W$ whose absolute minima are the pure phases. More precisely, $W:[-1,1]\to \R$ satisfies 
\begin{equation}\label{9,876}
W(t)>0 \;\;
{\mbox{ for every $t\in (-1,1)$}} \qquad\mbox{and}\qquad W(\pm 1)=0. 
\end{equation}

In the study of phase coexistence models, it is often desirable to quantify the size of the transition layer with respect to that of the pure phases (or, more precisely, of the regions close to pure phases). This is important since, due to the complexity of the problem, it is often difficult, when not impossible, to have general descriptions of the interfaces and the best one can do is typically to describe the phase separation ``in a measure theoretic sense''
and conclude that, at a large scale, most of the space is occupied
by state parameters close to
pure phases, separated by ``thin'' interfaces. In this spirit, the study of density estimates, namely of obtaining sharp quantifications of the measure of the interface, happens to be quite a cross-disciplinary topic, involving, under various perspectives, mathematical analysis, statistical mechanics, and materials sciences.

In the framework of density estimates a common hypothesis on the potential is that its growth from the minima is comparable to a polynomial of degree~$m$, where~$m\in (0,p]$ and~$p\in(1,+\infty)$ is the exponent of the interaction term, see for instance formula~(1.10) in~\cite{6} for the local case and formula~(1.9) in~\cite{3} for the nonlocal one.   

In the current paper, we consider the energy functional 
\begin{equation}\label{diorefdvgyt}
\mathcal{E}_{s}^p(u,\Omega):=\frac{1}{2}\int_{\Omega}\int_{\Omega}\frac{\left|u(x)-u(y)\right|^p}{\left|x-y\right|^{n+sp}}\,dx\,dy +\int_{\Omega}\int_{\R^n\setminus \Omega}\frac{\left|u(x)-u(y)\right|^p}{\left|x-y\right|^{n+sp}}\,dx\,dy+\int_{\Omega}W(u(x))\,dx
\end{equation}
where~$W$ presents a polynomial growth from the minima~$-1$ and~$1$ of degree~$m\in \left[p,+\infty\right)$. We will refer to these potentials as ``degenerate'', since they have a ``slow'' growth rate from their minima (in particular, slower than the homogeneity of the associated interaction energy).

In particular, we prove density estimates for every~$p\in(1,+\infty)$ and~$s\in \left(0,\frac{1}{p}\right)$, see Theorem~\ref{th:fracp>=2} below. The precise hypothesis on~$W$, the statement of Theorem~\ref{th:fracp>=2} and its corollaries are all discussed in Section~\ref{dcvfgbterslop9645r}. In what follows, we introduce the concept of density estimate and we present in more detail the implications of the presence of a degenerate double-well potential.   

\subsection{Density estimates}

To establish a conceptual framework
in a model case, let us consider the prototype energy functional
\begin{equation}\label{local2energy}
\mathcal{E}(u,\Omega):=\frac{1}{2}\int_{\Omega}\left|\nabla u(x)\right|^2\,dx+\int_{\Omega} W(u(x))\,dx,
\end{equation}
with~$W$ satisfying the assumptions in~\eqref{9,876} and ``nondegenerate''
(in the sense that~$W$ detaches at least quadratically from its minima). 
{F}rom one point of view, as expected from the shape of the potential~$W$, a minimizer for~$\mathcal{E}$ will be more likely to attain values close to the pure phases, to reduce the overall contribution of~$W$ in~$\Omega$. Yet, from another point of view, the interaction term penalizes ``big jumps'' of the state parameter. As a consequence of this simultaneous action, the minimizer will be pushed to attain values close to the pure phases but in such a way that at large scales the portion of the domain where the state parameter jumps from one pure phase to the other is negligible. 

A common question in physics and mathematics is related to the geometrical features of these  interfaces as we ``zoom out''. So far, there have a been two approaches to this question: one exploiting~$\Gamma$-convergence and the other relying on
density estimates.

On the one hand,
the idea behind~$\Gamma$-convergence is to analyse the limit functional obtained by suitably rescaling the minimizers of~$\mathcal{E}$ and deduce qualitative information regarding the interfaces at large scales. In particular, if~$u$ minimizes~$\mathcal{E}$ in~$\Omega$, then for~$\epsilon\in (0,1)$ we consider the rescaled state parameter 
\begin{equation*}
u_\epsilon(x):=u\left(\frac{x}{\epsilon}\right).
\end{equation*}
In~\cite{Mor77}, and  later in greater generality in~\cite{bouchitte1990singular}, it is proved that there exists some set~$E\subset\R^n$ such that as~$\epsilon\to 0^+$ the minimizers~$u_\epsilon$ converge, up to a subsequence, in~$L_{\textit{loc}}^1(\Omega)$ to\footnote{As usual, we denote here by~$\chi_A$ the characteristic function (or indicator function) of the set~$A$, namely
$$\chi_A(a):=\begin{cases}1&{\mbox{ if }}a\in A,\\0&{\mbox{ if }}a\not\in A.\end{cases}$$}
a step function~$u_0:=\chi_{E}-\chi_{E^c}$. Interestingly, the interface~$\partial E$ is a minimal surface, namely a minimizer of the perimeter functional.
\medskip

On the other hand, the purpose of density estimates is to provide lower bounds for the measure of the portion of the domain characterized by values of the minimizers of~$\mathcal{E}$ close to the pure phases. The first result in this direction is due to L. Caffarelli and A. C\'ordoba. In~\cite{caffarelli1995uniform} they prove that if~$u:\Omega\to [-1,1]$ minimizes~$\mathcal{E}$, then, for any~$\theta\in (-1,1)$ and~$r\in (0,+\infty)$ sufficiently large, one has that~$\{|u|>\theta\}\cap B_r$ is comparable to~$r^n$
(i.e., to the volume of the full ball~$B_r$), while~$\{|u|\leq \theta\}\cap B_r$ is bounded from above by a constant times~$r^{n-1}$
(corresponding, for example, to the $(n-1)$-dimensional surface measure in~$B_r$ of a hyperplane passing through the origin). In other words, state parameters close to the pure phases occupy a considerable portion of the domain, while the interface is negligible in the sense of measure. 

As a consequence, as the scale parameter~$\epsilon\in (0,1)$ approaches zero, the set~$\{|u_\epsilon|\leq \theta\}\cap B_r$
converges in the Hausdorff distance to the hypersurface~$\partial E$. More specifically, for any~$\delta$, $R>0$ there exists~$\epsilon_0:=\epsilon_0(\delta,R)>0$ such that, for any~$\epsilon\in (0,\epsilon_0)$, 
\begin{equation*}
\left\lbrace \left|u_\epsilon\right|\leq \theta\right\rbrace\cap B_R\subset \bigcup_{x\in\partial E}B_{\delta}(x).
\end{equation*}

Since the pioneering work of L. Caffarelli and A. C\'ordoba, density estimates have been used to study the interface convergence of minimizers and quasiminimizers of a larger spectrum of energy functionals related to phase separation. In~\cite{6} and~\cite{petrosyan2005geometric} gradient terms of~$L^p$-type with~$p\in (1,+\infty)$ are considered, together also with~$\chi$-shaped potentials~$W$, namely~$W(x):=\chi_{(-1,1)}(x)Q(x)$ with~$Q$ uniformly bounded away from zero.
This type of potentials are related to models for fluid jets and capillarity phenomena; see also~\cite{valdinoci2001plane} for the elliptic case. In~\cite{2} density estimates are also proved under the weaker assumption of quasiminimality.  

\medskip
 
As already mentioned, nonlocal counterparts of~\eqref{local2energy} have been dealt with in the literature. A consolidated choice consists in replacing the~$L^2$-norm of the gradient with the Gagliardo seminorm. In~\cite{savin2012gamma} and~\cite{3}, for~$s\in(0,1)$ and~$\epsilon\in (0,1)$, nonlocal phase separation was examined for the energy 
\begin{equation}\label{1nonlo2}\begin{split}
\mathcal{F}_{s,\epsilon}^2(u,\Omega)&:=\epsilon^{2s}\left(\frac{1}{2}\int_{\Omega}\int_{\Omega}\frac{\left|u(x)-u(y)\right|^2}{\left|x-y\right|^{n+2s}}\,dx\,dy +\int_{\Omega}\int_{\R^n\setminus \Omega}\frac{\left|u(x)-u(y)\right|^2}{\left|x-y\right|^{n+2s}}\,dx\,dy\right)\\&\qquad\qquad+\int_{\Omega}W(u(x))\,dx.\end{split}
\end{equation} 
More precisely, in~\cite{savin2012gamma} it is proved that, by suitably scaling the energy in~\eqref{1nonlo2} by some factor depending on~$\epsilon$ and~$s$, as~$\epsilon$ vanishes, the minimizers are compact and the rescaled energy~$\Gamma$-converges (to a functional of geometric
flavor which we discuss below). Such a gauge is proved to be~$\epsilon^{-2s}$, $\epsilon^{-1}|\log(\epsilon)|$ and~$\epsilon^{-1}$ respectively for the cases~$s\in\left(0,\frac{1}{2}\right)$, $s=\frac{1}{2}$ and~$s\in \left(\frac{1}{2},1\right)$.

Interestingly, in Theorem~1.4 of~\cite{savin2012gamma}, it is established that if~$s\in \left(0,\frac{1}{2}\right)$, then~$\epsilon^{-2s}\mathcal{F}_{s,\epsilon}^2$
$\Gamma$-converges as~$\epsilon\to 0^+$ to the nonlocal perimeter~$\mbox{Per}_s(\cdot,\Omega)$ in~$\Omega$, defined for every measurable set~$E\subset\R^n$ as
\begin{equation*}
\mbox{Per}_s(E,\Omega):=\int_{E\cap \Omega}\int_{E^c\cap \Omega}\frac{dx\,dy}{\left|x-y\right|^{n+2s}}+\int_{E\cap \Omega^c}\int_{E^c\cap \Omega}\frac{dx\,dy}{\left|x-y\right|^{n+2s}}+\int_{E\cap \Omega}\int_{E^c\cap \Omega^c}\frac{dx\,dy}{\left|x-y\right|^{n+2s}}.
\end{equation*} 
Additionally, in~\cite{savin2012gamma} it is proved that if~$s\in\left[\frac{1}{2},1\right)$ the~$\Gamma$-limit coincides with the classical perimeter. Surprisingly, this means that in the regime~$s\in \left[\frac{1}{2},1\right)$ and at large scales the effect of the Gagliardo seminorm on the interfaces is similar to the one of the~$L^2$-norm of the gradient and the phase separation resembles the classical one.  

In this sense, the case~$s\in \left(0,\frac{1}{2}\right)$ (when~$p=2$, or, more generally, 
the case~$s\in \left(0,\frac{1}{2}\right)$ when~$p\in(1,+\infty)$) is considered as ``genuinely
nonlocal'', since the problem retains its nonlocal features at every scale.

For what concerns density estimates, the cases~$s\in \left(0,\frac{1}{2}\right)$ and~$s\in \left[\frac{1}{2},1\right)$ can be treated separately. As a matter of fact, in~\cite{5} density estimates are proved for the functional in~\eqref{1nonlo2} for every~$s\in  \left(0,\frac{1}{2}\right)$ using the Sobolev inequality. This approach fails when~$s\in \left[\frac{1}{2},1\right)$ and in~\cite{3} a fine measure theoretic result is established, namely Theorem~1.6 in~\cite{3}, to obtain density estimates for every~$s\in \left(0,1\right)$, see Theorem~1.4 in~\cite{3}. 

The Sobolev inequality will be employed also in the current paper to obtain density estimates for the energy in~\eqref{diorefdvgyt}. Analogously to 
the case~$p=2$ dealt with in~\cite{5} and~\cite{3}, this approach only works if~$s\in\left(0,\frac{1}{p}\right)$. The case~$s\in \left[\frac{1}{p},1\right)$ will be treated in a forthcoming article~\cite{DFVERPP}. 

\subsection{Degenerate double-well potentials}
An interesting problem surrounding the topic of density estimates is how the growth from the minima of the double-well potential~$W$ might affect phase separation. For example, a slow growth might induce, in some cases, the formation of ``intermediate phases''. To picture this, we may consider the extreme situation where, for some~$a\in(0,1)$ and every~$x\in(-1,1)$, the potential is given by 
\begin{equation}
W(x):=\chi_{[-a,a]}(x),
\end{equation}
see Figure~\ref{nijufLO-06557}. In this case, the minimizer will occupy values close to~$-a$ and~$a$ to reduce the interaction, whether this is prompted by a nonlocal or local type of interaction term, while the state parameters with value in~$(-1,-a)$ and~$(a,1)$ become empty, violating density estimates.  In particular, this shows that to obtain density estimates some assumptions on~$W$ are necessary.

As already mentioned, a common hypothesis on the potential is that if~$p\in(1, +\infty)$ is the exponent of the interaction term, then, for some~$m\in (0,p]$, $\lambda\in(0,1)$ and~$\Lambda\in [1,+\infty)$, and every~$x\in (-1,1)$, 
\begin{equation}\label{lscfebvd76554}
\lambda (1-x^2)^m \leq W(x)\leq  \Lambda(1-x^2)^m. 
\end{equation}
This corresponds, in the framework of this paper, to a nondegenerate double-well potential,
since the growth from the potential's minima is not larger than the homogeneity of the interaction term.
In particular, in the classical case of a gradient term of~$L^2$-type this translates into a (sub)quadratic growth, see~\cite{caffarelli1995uniform}. In such a case, a common double-well potential in the literature is 
\begin{equation}\label{bcgdbalksyt}
W(x):=\frac{(1-x^2)^2}{4}.
\end{equation}
In view of~\eqref{lscfebvd76554} and~\eqref{bcgdbalksyt}, for~$m\in \left[p,+\infty\right)$ and~$p\in(1,+\infty)$ it is natural to define   
\begin{equation}\label{kifvdgcty56}
W(x):=\frac{(1-x^2)^m}{2m}
\end{equation} 
as a prototype of a degenerate double-well potential. See Figure~\ref{nijufLO-06557} for a comparison with the potential in~\eqref{bcgdbalksyt}.   
\begin{center}
\begin{figure}[!ht]
$\qquad$ \includegraphics[width=0.40\textwidth]{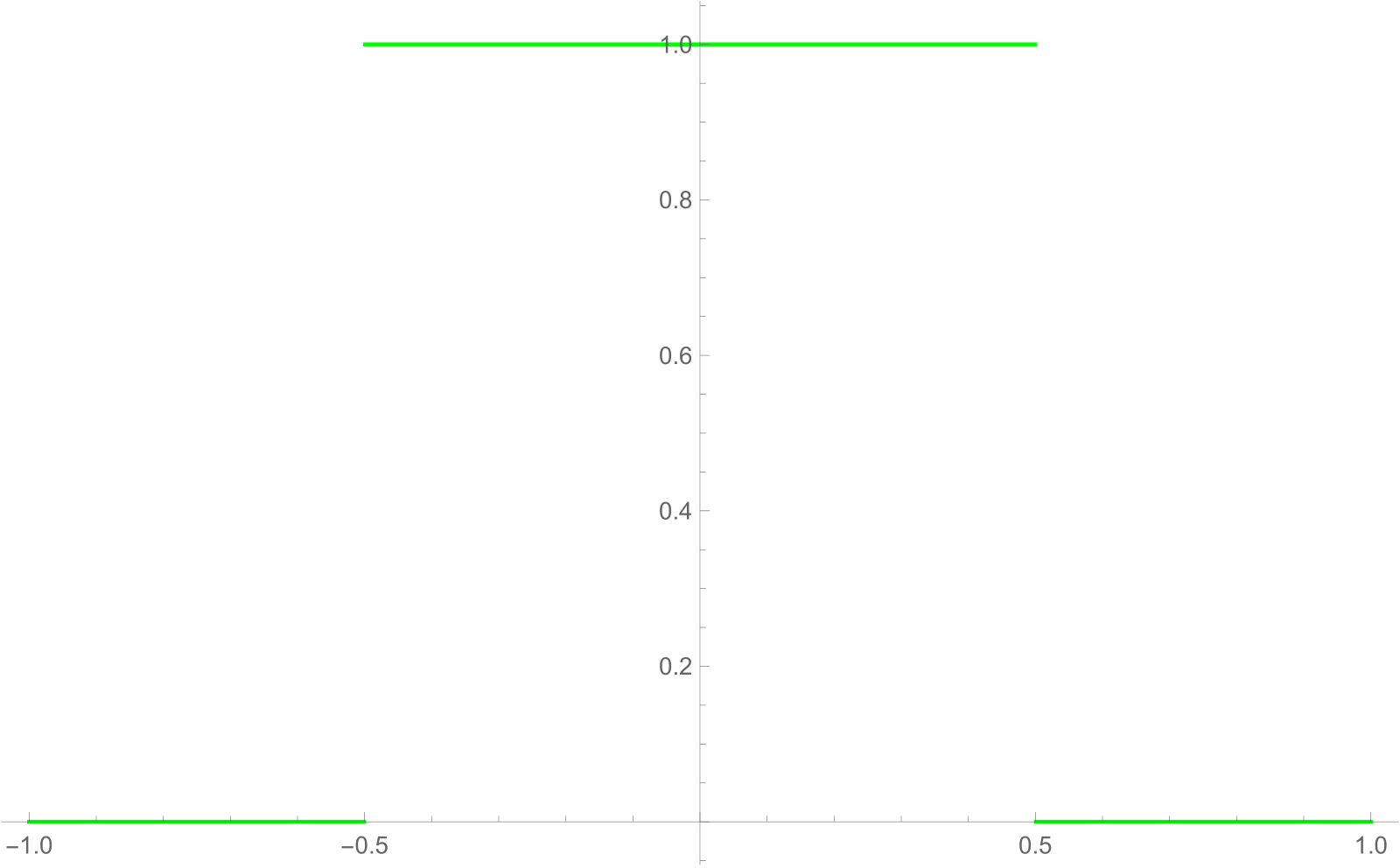}$\qquad$ $\qquad$
\includegraphics[width=0.40\textwidth]{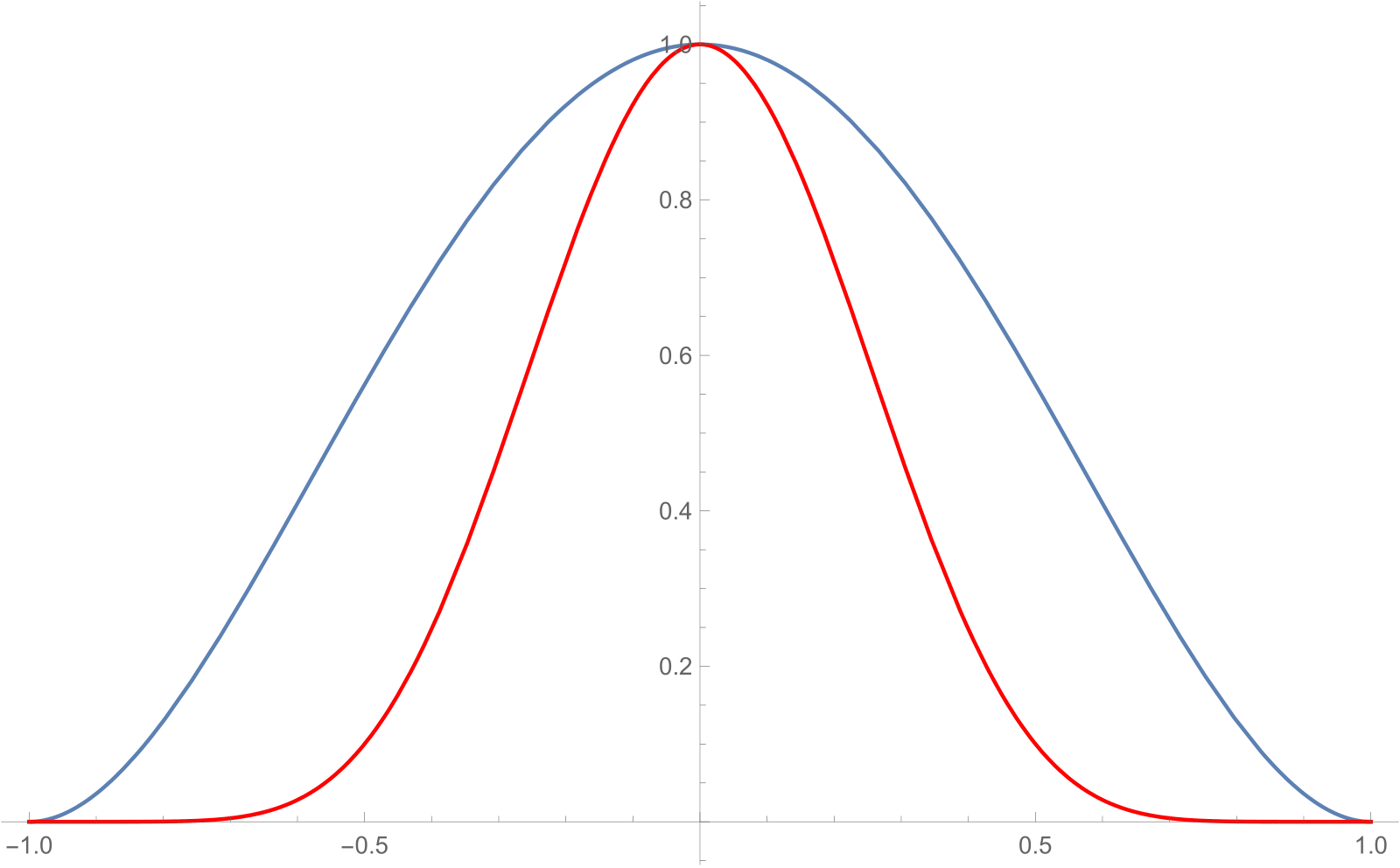}
\caption{\sl\footnotesize Plots of $\chi_{[-1/2,1/2]}(x)$, $(1-x^2)^2$ and $(1-x^2)^8$  respectively in green, blue and red.}
        \label{nijufLO-06557}
\end{figure}
\end{center}

Of course, the conditions in~\eqref{lscfebvd76554} may be too restrictive, and density estimates may still hold for double-well potentials with a flatter profile close to the minima. This problem has been first addressed in the local setting in~\cite{1},  where the authors prove density estimates for quasiminimizers of a phase transition energy driven by the~$L^p$-norm of the gradient and with a degenerate double-well potential of the type~\eqref{kifvdgcty56}, as far as some condition on~$n$, $m$ and~$p$ are met, see equation~(1.7) in~\cite{1}.  

The aim of this paper is to obtain density estimates for minimizers of the energy in~\eqref{diorefdvgyt}, when the potential~$W$ presents a polynomial growth from the minima comparable to the one of the degenerate potential in~\eqref{kifvdgcty56}.

In the following section we introduce all the necessary mathematical notation, state precisely the main results and provide the structure of the paper.

\subsection{Mathematical framework and main results}\label{dcvfgbterslop9645r}
In what follows~$n\in\N$, $\Omega\subset\R^n$ is open, $p\in (1,+\infty)$ and~$s\in \left(0,\frac{1}{p}\right)$. We consider~$W\in C^1([-1,1],\R^+)$ such that for some~$\theta\in (-1,1)$, $\lambda\in (0,1]$, $\Lambda\in [1,+\infty)$
and~$m\in [p,+\infty)$, for every~$x\in [-1,1]$,
\begin{equation}\label{potential}
\lambda\chi_{(-\infty,\theta]}(x)\left|1+x\right|^{m}   \leq W(x)\leq \Lambda\left|1+x\right|^m.
\end{equation}
Moreover, we assume that, for some~$q\in (0,\theta+1]\cap (0,1)$ and~$c_1\in (0,+\infty)$, for every~$x\in (-1,-1+q)$,
\begin{equation}\label{potential1}
c_1\left|1+x\right|^{m-1} \leq  W'(x).
\end{equation}
Note that the prototype of a degenerate double-well potential in~\eqref{kifvdgcty56} satisfies~\eqref{potential} and~\eqref{potential1}. 

Also, for every measurable function~$u:\R^n\to \R$ we define
\begin{equation}\label{kinetic}
\mathcal{K}_s^p(u,\Omega):=\frac{1}{2}\int_{\Omega}\int_{\Omega}\frac{\left|u(x)-u(y)\right|^p}{\left|x-y\right|^{n+sp}}\,dx\,dy +\int_{\Omega}\int_{\R^n\setminus \Omega}\frac{\left|u(x)-u(y)\right|^p}{\left|x-y\right|^{n+sp}}\,dx\,dy
\end{equation} 
and rewrite the energy in~\eqref{diorefdvgyt} as
\begin{equation}\label{energy}
\mathcal{E}_s^p(u,\Omega):= \mathcal{K}_s^p(u,\Omega)+\int_{\Omega}W(u(x))\,dx.
\end{equation}
Furthermore, in analogy with~\cite{savin2012gamma}, for every~$\epsilon \in(0,1)$, we consider the rescaled energy 
\begin{equation}\label{mathcalF}
\mathcal{F}_{s,\epsilon}^p(u,\Omega):=\mathcal{K}_{s}^p(u,\Omega)+\epsilon^{-sp}\int_{\Omega}W(u(x))\,dx.
\end{equation} 
A convenient function space to study~$\mathcal{K}_s^p$ is 
\begin{equation*}
\widehat{H}^{s,p}(\Omega):=\big\lbrace u:\R^n\to \R\mbox{  measurable s.t.  } \mathcal{K}_{s}^p(u,\Omega)<+\infty  \big\rbrace.
\end{equation*}
Note that~$W^{s,p}(\R^n)\subset \widehat{H}^{s,p}(\Omega)$. Since~$W$ is defined in the interval~$[-1,1]$, we consider the subspace of all those functions~$u\in \widehat{H}^{s,p}(\Omega)$ such that~$\left|u\right|\leq 1$ a.e., namely 
\begin{equation}\label{AdSet}
X^{s,p}(\Omega):=\left\lbrace u\in \widehat{H}^{s,p}(\Omega)\mbox{  s.t.  }\left\| u\right\|_{L^\infty(\R^n)}\leq 1\right\rbrace.
\end{equation}

With these choices we have the following definition:

\begin{defn}[Minimizer for~$\mathcal{E}_s^p$ and~$\mathcal{F}_{s,\epsilon}^p$]
If~$\Omega\subset\R^n$ is bounded, we say that~$u\in X^{s,p}(\Omega)$ is a minimizer for~$\mathcal{E}_s^p(\cdot,\Omega)$ in~$X^{s,p}(\Omega)$ if
\begin{equation*}
\mathcal{E}_s^p(u,\Omega) \leq  \mathcal{E}_s^p(v,\Omega)
\end{equation*}
for any~$v\in X^{s,p}(\Omega)$ such that~$v\equiv u$ in~$\R^n\setminus \Omega$.

If~$\Omega\subset\R^n$ is unbounded, we say that~$u\in X^{s,p}(\Omega)$ is a minimizer for~$\mathcal{E}_s^p(\cdot,\Omega)$ in~$X^{s,p}(\Omega)$ if it minimizes~$\mathcal{E}_s^p(\cdot,\Omega')$ in~$X^{s,p}(\Omega')$ for every bounded and open set~$\Omega'\Subset\Omega$. 

Similarly, one also defines a minimizer for~$\mathcal{F}_{s,\epsilon}^p$.
\end{defn}

\begin{rem}\label{lfvc}{\rm
Note that if~$u$ minimizes~$\mathcal{E}_s^p$ (or~$\mathcal{F}_{s,\epsilon}^p$) in an open and bounded set~$\Omega\subset\R^n$, then it minimizes~$\mathcal{E}_s^p$ (or~$\mathcal{F}_{s,\epsilon}^p$) in every measurable subset~$\Omega'\subset\Omega$. }
\end{rem}
\medskip

The main result of this paper is the following density estimate: 

\begin{thm}[Density estimates]\label{th:fracp>=2}
Let~$p\in (1,+\infty)$, $s\in \left(0,\frac{1}{p}\right)$ and~$\theta_1$, $\theta_2\in (-1,\theta]$. 

Let also\begin{equation}\label{liibfdcvre5555}
\theta_*:=\min \left\lbrace \theta_1,\theta_2,-1+q \right\rbrace\quad\mbox{and}\quad \theta^*:=\max\left\lbrace  \theta_1,\theta_2,-1+q\right\rbrace.
\end{equation}

Assume that~$u$ is a minimizer for~$\mathcal{E}_{s}^p$ in~$\Omega$ and that, for some~$c_0$, $r_0\in (0,+\infty)$ such that~$B_{r_0}\subset \Omega$,
\begin{equation}\label{Fcon-812345}
\mathcal{L}^{n}(B_{r_0}\cap \left\lbrace u>\theta_1 \right\rbrace)>c_0.
\end{equation}

Then, there exist~$R^*:=R_{s,n,p,m,\theta_*,r_0}^*\in [r_0,+\infty)$, $\tilde{c}:=\tilde{c}_{s,n,p,m,\Lambda,c_1,\theta_*,r_0,c_0}\in (0,1)$, and~$c:=c_{m,\theta_*}\in (0,+\infty)$, such that, for any~$r\in \left[R^*,+\infty\right)$ with~$B_{\frac{3r}{2}}\subset \Omega$,
\begin{equation}\label{benz}
c\int_{B_r\cap \left\lbrace \theta_{*}<u\leq \theta^* \right\rbrace}\left|1+u(x)\right|^m\,dx+\mathcal{L}^n(B_r\cap \left\lbrace u>\theta_2 \right\rbrace)\geq \tilde{c}\, r^n.
\end{equation}
\end{thm}

Here above and in the rest of the paper, $\mathcal{L}^{n}(E)$ denotes
the Lebesgue measure of a set~$E\subset\R^n$.

The proof of the density estimates in Theorem~\ref{th:fracp>=2} relies
on the construction of a suitable barrier (see 
Theorem~\ref{lem:1} below) that is inspired by Lemma~3.1 in~\cite{3},
where a barrier is built in the case~$p=2$ and~$m=2$.
The extension to the cases~$p\neq 2$ and~$m\in \left[p,+\infty\right)$
that we provide in this paper
is nontrivial, due to the nonlinearity of the operator and the precence of
a degenerate potential. 
In the case of the classical~$p$-Laplacian, barriers have been built
in~\cite{MR2228294} in a different context.
The case~$p=2$ in the nonlocal setting with a degenerate potential has been
also considered in~\cite{Depas}.

The presence of the first term in the left-hand side of~\eqref{benz} is a consequence of the weak lower bound in~\eqref{potential}. More precisely, under some additional assumptions on the lower bound for~$W$ (see~\eqref{condiWzione} below) it is possible to reabsorb such term in the right hand-side of the inequality, obtaining the ``full density estimates''. This is made possible thanks to the following upper bound on the energy~$\mathcal{E}_s^p$. Its proof is contained in the forthcoming paper~\cite{DFVERPP}.  

\begin{thm}[\cite{DFVERPP}]\label{erfvbgt-098}
Let~$p\in (1,+\infty)$, $s\in \left(0,\frac{1}{p}\right)$ and~$u$ be a minimizer of~$\mathcal{E}_s^p$ in~$X^{s,p}(B_{R+2})$ with~$R\geq 2$. Then, there exists~$\bar{C}>0$, depending only on~$s$, $n$, $p$, $m$ and~$\Lambda$, such that 
\begin{equation}\label{BFAOTE}
\mathcal{E}_s^p(u,B_R)\leq \bar{C} R^{n-sp}.
\end{equation} 
\end{thm}

{F}rom Theorems~\ref{th:fracp>=2} and~\ref{erfvbgt-098} we obtain the  following density estimate:

\begin{cor}\label{coro-09}
Let~$\theta_1$, $\theta_2\in (-1,1)$ and~$u$ be as in the statement of~Theorem~\ref{th:fracp>=2}. 

Assume that, for every~$\mu\in (-1,1)$, there exists~$\lambda_\mu\in (0,+\infty)$ such that, for every~$x\in (-1,1)$,
\begin{equation}\label{condiWzione}
\lambda_\mu \chi_{\left(-\infty,\mu\right]}(x)\left|1+x\right|^m  \leq W(x).
\end{equation}
Moreover, let~$\theta_*$ and~$\theta^*$ be as in~\eqref{liibfdcvre5555}. 

Then, there exist~$\widetilde{R}:=\widetilde{R}_{s,n,p,m,\theta_*,r_0,\lambda_{\theta^*}}\in [r_0,+\infty)$ and~$\hat{c}:=\hat{c}_{s,n,p,m,\Lambda,c_1,\theta_*,r_0,c_0}\in (0,1)$ such that, for any~$r\in \left[\widetilde{R},+\infty\right)$ with~$B_{\frac{3r}{2}}\subset \Omega$,
\begin{equation}\label{benzina}
\mathcal{L}^n(B_r\cap \left\lbrace u>\theta_2 \right\rbrace)\geq \hat{c}\,r^n.
\end{equation}
\end{cor}

As already mentioned, an important consequence of the
density estimates is the uniform convergence of the interface of the minimizers of the rescaled energy functional to a hypersurface. To prove this result we require an additional condition on the potential~$W$. In particular, we assume that there exists some~$c_2,c_3\in (0,+\infty)$ such that for every~$x\in (-1,1)$
\begin{equation}\label{q874brygtefd}
c_3(1-x^2)^m\geq W(x)\geq c_2(1-x^2)^m .
\end{equation}

\begin{thm}[Compactness]\label{gliutgbedgty}
Let~$p\in(1,+\infty)$, $s\in \left(0,\frac{1}{p}\right)$ and~$W$ satisfy~\eqref{q874brygtefd}. Assume that the sequence~$\{u_\epsilon \}_\epsilon \subset X^{s,p}(\Omega)$ satisfies
\begin{equation}\label{unifbound}
\sup_{\epsilon\in(0,1)} \mathcal{F}_{s,\epsilon}^p(u_\epsilon,\Omega)<+\infty.
\end{equation}
Then, there exists a measurable set~$E\subset\R^n$ such that, up to a subsequence,
\begin{equation}
u_\epsilon\to u^* :=\chi_{E}-\chi_{E^c}\quad \mbox{in}\quad L_{\textrm{loc}}^1(\Omega).
\end{equation} 
\end{thm}
This compactness result is a consequence of the scaling chosen for~$\mathcal{F}_{s,\epsilon}^p$. Indeed, it follows from~\eqref{mathcalF} and~\eqref{unifbound} that the Gagliardo seminorm of~$u_\epsilon$ is bounded on all compact subsets of~$\Omega$. Then, it is enough to use compact embeddings of~$W^{s,p}$ to conclude. The case~$s\in \left[\frac{1}{p},1\right)$ is more involved, and it will be dealt with in the forthcoming paper~\cite{DFVERPP}.  

As a byproduct of the density estimates in~\eqref{benzina} and the compactness in Theorem~\ref{gliutgbedgty} we obtain the Hausdorff convergence for the interfaces:  

\begin{cor}[Hausdorff convergence]\label{rgeb}
Let~$p\in (1,+\infty)$, $s\in \left(0,\frac{1}{p}\right)$ and~$W$ satisfy~\eqref{q874brygtefd}. For~$\epsilon\to 0^+$ let~$\{ u_\epsilon\}_\epsilon\subset X^{s,p}(\Omega)$  be such that~$u_\epsilon$ is a minimizer of~$\mathcal{F}_{s,\epsilon}^p$ in~$\Omega$ and~\eqref{unifbound} holds. Moreover, let~$E\subset\R^n$ be as in the statement of Theorem~\ref{gliutgbedgty}. 

Then, for every~$\Theta\in (0,1)$, $\delta\in (0,+\infty)$, and~$R\in(0,+\infty)$ satisfying~$B_{R}\Subset \Omega$, there exists~$\epsilon_0:=\epsilon_0(\Theta,\delta,R)>0$ such that, if~$\epsilon\in (0,\epsilon_0)$,
\begin{equation*}
\left\lbrace \left|u_\epsilon\right|<\Theta \right\rbrace\cap B_R \subset \bigcup_{ x \in \partial E} B_{\delta}(x). 
\end{equation*}  
\end{cor}

We remark that the results presented in this paper are stated
and proved for all~$p\in(1,+\infty)$, but they are new
even for the case~$p=2$.\medskip

The rest of this paper is structured as follows. 
In Section~\ref{y43iueihdevwjfwe9876543} we recall some known results
that will be used in the proofs of the main statements.
In Section~\ref{desmos} we construct a suitable barrier, see Theorem~\ref{lem:1} below, which will be used to prove Theorem~\ref{th:fracp>=2}. 
The construction of this barrier relies also on some
technical results that are collected in Appendix~\ref{appebi}.

The proofs of Theorem~\ref{th:fracp>=2} and Corollary~\ref{coro-09} are contained in Section~\ref{cardi}. Finally, in Section~\ref{vdcbnhetrslkoiuojhytbirvec}  we prove Theorem~\ref{gliutgbedgty} and Corollary~\ref{rgeb}.

\section{Some auxiliary results}\label{y43iueihdevwjfwe9876543}

In this section we recall some auxiliary results that will be used throughout
the proofs of our main theorems.

\subsection{A general estimate}
The following statement is provided in Lemma~3.2. in~\cite{3}.
We point out that by a careful inspection of the proof presented in~\cite{3}
one can check that the dependance of the constants~$\tilde{c}$
and~$R_{*}$ is as claimed here below (this will be important in the
forthcoming paper~\cite{DFVERPP}).

\begin{lem}[Lemma~3.2. in~\cite{3}]\label{GE-XE-JU}
Let~$\sigma$, $R_0\in (0,+\infty)$, $\nu\in (\sigma,+\infty)$ and~$\gamma$, $\mu$, $C\in (1,+\infty)$.

Let~$V:(0,+\infty)\to (0,+\infty)$ be a nondecreasing function. For any~$r\in [R_0,+\infty)$, let 
\begin{equation*}
\alpha(r):=\min\left\lbrace 1,\frac{\log(V(r))}{\log(r)}\right\rbrace.
\end{equation*} 
Also, suppose that 
\begin{equation}\label{condi1}
V(R_0)\geq \mu
\end{equation}
and that, for any~$r\in [R_0,+\infty)$, 
\begin{equation}\label{epo-vdfceobgt}
r^{\sigma}\alpha(r)\big(V(r)\big)^{\frac{\nu-\sigma}{\nu}}\leq C\,V(\gamma r).
\end{equation}

Then, there exist~$\tilde{c}\in (0,1)$,
depending on~$C$, $\sigma$, $\mu$, $\gamma$, $\nu$ and~$R_0$,
and~$R_{*}\in [R_0,+\infty)$,
depending on~$R_0$ and~$\gamma$,
such that, for any~$r\in [R_*,+\infty)$,
\begin{equation}\label{pmecb.54tb6}
V(r)\geq \tilde{c}\,r^{\nu}.
\end{equation}
\end{lem}

\subsection{H\"older regularity for minimizers}\label{CH-HCBob}

In what follows, we let~$F\in L^\infty(\R,\R)$. Also, for every~$u\in \widehat{H}^{s,p}(\Omega)$ we consider the renormalized energy
\begin{equation*}
\bar{\mathcal{E}}_{s}^p(u,\Omega):=(1-s) \mathcal{K}_s^p(u,\Omega)+\int_{\Omega} F(u(x))\,dx.
\end{equation*}
As customary, given some open set~$B\subset\R^n$ and~$\beta\in (0,1)$, we denote the H\"older seminorm of a function~$f:B\to \R$ by
\begin{equation*}
\left[f\right]_{C^\beta(B)}:=\sup_{x,y\in B} \frac{\left|f(x)-f(y)\right|}{\left|x-y\right|^\beta}. 
\end{equation*}
If~$\beta=1$, we denote the Lipschitz seminorm of~$f:B\to \R$ by 
\begin{equation*}
\left[f\right]_{C^{0,1}(B)}:=\sup_{x,y\in B} \frac{\left|f(x)-f(y)\right|}{\left|x-y\right|}.
\end{equation*}
Also, given~$x_0\in\Omega$ and~$R\in\big(0,\operatorname{dist}\,(x_0,\partial \Omega)\big)$, we set   
\begin{equation*}
\mbox{Tail}(u,x_0,R):=\left[(1-s)R^{sp}\int_{\R^n\setminus B_{R(x_0)}}\frac{\left| u(y)\right|^{p-1}}{\left|y-x_0\right|^{n+sp}}\right]^{\frac{1}{p-1}}.
\end{equation*} 

With this notation, we now recall the following result on the H\"older regularity for the minimizers of~$\bar{\mathcal{E}}_{s}^p$. Such result is contained in~\cite{cozzi2017regularity}, where the H\"older regularity is established using fractional De Giorgi classes. 

\begin{thm}[Theorem~6.4 and Proposition~7.5 in~\cite{cozzi2017regularity}]\label{Cozzi}
Let~$n\in \N$, $s_0\in (0,1)$, $s\in \left[s_0,1\right)$ and~$p\in (1,+\infty)$.
Let~$u\in \widehat{H}^{s,p}(\Omega)$ be a minimizer for~$\bar{\mathcal{E}}_s^p$ in~$\Omega$. 

Then, $u\in C_{\textit{loc}}^\alpha(\Omega)$ for some~$\alpha\in (0,1)$. 

Moreover, for any~$x_0\in \Omega$ and~$R\in\left(0, \frac{\operatorname{dist}\,(x_0,\partial \Omega)}{8}\right)$,
\begin{equation*}
\left[u\right]_{C^\alpha(B_R(x_0))}\leq \frac{C}{R^\alpha} \left(\left\|u\right\|_{L^\infty(B_{4R}(x_0))}+\mbox{Tail}(u;x_0,4R)+ R^s\left\|F(u)\right\|_{L^\infty(B_{8R}(x_0))}^\frac{1}{p} \right)
\end{equation*} 
for some~$C\geq 1$. 

The constants~$\alpha$ and~$C$ depend only on~$n$, $p$, $s_0$ when~$n\geq 2$, and also on~$s$ when~$n=1$.
\end{thm}

\section{Construction of a suitable barrier}\label{desmos}

The following Theorem~\ref{lem:1} is the extension to the cases~$p\neq 2$ and~$m\in \left[p,+\infty\right)$
of Lemma~3.1 in~\cite{3}.

We point out that, in our construction,
the dependence of the constants with respect to~$s$ is explicit. In this way, as it will be discussed in detail in~\cite{DFVERPP}, it is possible to obtain density estimates that are stable as~$s \to 1^-$, recovering local density estimates from the nonlocal ones. 

In what follows, for every~$t\in\R$ we denote by~$[t]$ the integer part of~$t$. Also, for every~$z:\R^n\to \R$ measurable and~$x\in\R^n$, we denote its~$p$-fractional Laplacian at~$x\in\R^n$ as
\begin{equation*}
(-\Delta)_p^s z(x):=\int_{\R^n}\frac{\left(z(x)-z(y)\right)\left|z(x)-z(y)\right|^{p-2}}{\left|x-y\right|^{n+sp}}\,dy.
\end{equation*}
Throughout this section we will also make use of some technical results
collected in Appendix~\ref{appebi}.

\begin{thm}\label{lem:1}
For any~$n\in\N$, $\tau\in (0,+\infty)$, $p\in (1,+\infty)$, $s\in \left(0,1\right)$ and~$m\in \left[p,+\infty\right)$ there exists~$\widehat{C}_{n,p,m}\in (0,+\infty)$ such that 
if we define 
\begin{equation}\label{chnecdloro}
\bar{R}=\bar{R}_{s,\tau,n,p,m}:=\left(\frac{\widehat{C}_{n,p,m}}{s(1-s)\tau}\right)^{\frac{1}{ps}} 
\end{equation}
the following statement holds true.

For every~$R\in \left[\bar{R},+\infty\right)$ there exists a rotationally symmetric function~$w\in C(\R^n,(-1,1])$ such that 
\begin{equation}\label{primacosauwwgyg098765}
w\equiv 1\quad\mbox{in}\quad \R^n\setminus B_R
\end{equation}
and, for every~$x\in B_R$,
\begin{equation}\label{cet}
-(-\Delta)_p^s w(x)\leq \tau (1+w(x))^{m-1}.
\end{equation}
Also, if we set~$q:=\frac{p}{m-1}$ and 
\begin{equation}\label{straimpoconst}
C_{s,\tau,n,p,m}:=2^{q+1}\max\left\lbrace 1, \bar{R}^{qs}\right\rbrace,
\end{equation} then,
for every~$x\in B_R$,
\begin{equation}\label{grw}
1+w(x)\leq\frac{C_{s,\tau,n,p,m}}{\left(1+R-|x|\right)^{qs}}. 
\end{equation}
\end{thm}

The proof of Theorem~\ref{lem:1} will be based on the following construction.
We take a large~$r\in \left[1,+\infty\right)$, to be conveniently chosen with respect to~$R$ and~$\tau$ here below (see formula~\eqref{gldini}).

We consider the integer quantity  
\begin{equation}\label{deffk}
k=k(s,p):=
\begin{dcases}
0 \quad &\mbox{if}\quad p\in (1,2)\mbox{ and }s\in \left(0,\frac{p-1}{2p}\right),\\
\left[\frac{sp}{p-1}\right]+1\quad &\mbox{if}\quad p\in (1,2)\mbox{ and }s\in \left[\frac{p-1}{2p},1\right),\\
1  \quad &\mbox{if}\quad p\in [2,+\infty)\mbox{ and }s\in \left(0,1\right).
\end{dcases}
\end{equation}

We set~$q:=\frac{p}{m-1}$ and we define, for every~$t\in (0,r)$,
\begin{equation}\label{fuhweoiu34ii43ty43ugijkebjwk7684903}
\begin{split}
g(t) & :=t^{-qs}\\{\mbox{and }}\quad
f(t) & := g(r-t)+\sum_{i=0}^{k} \frac{(-1)^{i+1}}{i!}g^{(i)}\left(\frac{r}{2}\right)\left(t-\frac{r}{2}\right)^i.\end{split}\end{equation}
Furthermore, we also define the functions
\begin{equation}\label{vdefinition}\begin{split}
h(t) &:=\begin{dcases}
0\quad &\mbox{if}\quad t\in (0,r/2],\\
\min \left\lbrace f(t),1\right\rbrace  \quad &\mbox{if}\quad t\in (r/2,r),\\
1&\mbox{if}\quad t\in [r,+\infty)
\end{dcases}\\ {\mbox{and }}\quad
v(x) & :=h(\left|x\right|)\quad\mbox{for every}\quad x\in\R^n.
\end{split}
\end{equation}

We provide the following estimate for the function~$h$:

\begin{prop}\label{MLNINT}
Let~$s\in (0,1)$, $p\in (1,+\infty)$, $m\in \left[p,+\infty\right)$ and~$q:=\frac{p}{m-1}$. Let~$h:(0,+\infty)\to [0,1]$ be as in~\eqref{vdefinition}.

Then, there exists~$\bar{c}_{k,q,s}\in(1,+\infty)$ such that, for every~$t\in (0,+\infty)$,
\begin{equation}\label{honda}
\min\left\lbrace 1,(r-t)^{-ps} \right\rbrace\leq (h(t)+\bar{c}_{k,q,s}r^{-qs})^{m-1}.
\end{equation}
\end{prop}

\begin{proof}
We observe that if~$t\in \left(\frac{r}{2},r\right)$ and~$h(t)<1$, then
\begin{equation}\label{fr-fe-87}
\begin{split}
h(t)&=f(t)\\
&=g(r-t)+\sum_{i=0}^{k} \frac{(-1)^{i+1}}{i!}g^{(i)}\left(\frac{r}{2}\right)\left(t-\frac{r}{2}\right)^i\\
&\geq (r-t)^{-qs}-\sum_{i=0}^{k} \frac{\left|g^{(i)}\left(\frac{r}{2}\right)\right|}{i!} \left(t-\frac{r}{2}\right)^i\\
&\geq  (r-t)^{-qs}-\sum_{i=0}^{k}
c_{q,s,i}\left(\frac{r}{2}\right)^{-qs-i}\left(\frac{r}{2}\right)^i\\
&=(r-t)^{-qs}-c_{k,q,s}\left(\frac{r}{2}\right)^{-qs},
\end{split}
\end{equation}
where
\begin{equation}\label{contevbdkop}
c_{q,s,i}:=\frac{\displaystyle \prod_{j=0}^{i-1} (qs+j)}{i!} \quad\mbox{and}\quad c_{k,q,s}:=\sum_{i=0}^k c_{q,s,i}.
\end{equation}
On this account, if~$t\in \left( \frac{r}{2},r\right)$ and~$h(t)<1$ and we adopt the notation 
\begin{equation}\label{8132efdkunh}
\tilde{c}_{k,q,s}:=2^{qs}c_{k,q,s}, 
\end{equation}
we can use~\eqref{fr-fe-87} to deduce that 
\begin{equation}\label{clero}
(h(t)+\tilde{c}_{k,q,s}r^{-qs})^{\frac{p}{q}}\geq ((r-t)^{-qs})^{\frac{p}{q}}=(r-t)^{-ps}\geq \min\left\lbrace 1,(r-t)^{-ps} \right\rbrace.
\end{equation}

Moreover, if~$h(t)=1$, then we have that 
\begin{equation}\label{sentiero}
(h(t)+\tilde{c}_{k,q,s}r^{-qs})^{m-1}=(1+\tilde{c}_{k,q,s}r^{-qs})^{m-1}\geq 1\geq  \min\left\lbrace 1,(r-t)^{-ps} \right\rbrace.
\end{equation}
Finally, if~$t\in \left(0,\frac{r}{2}\right]$,
\begin{equation}\label{pero}
\begin{split}
(h(t)+ 2^{qs} r^{-qs})^{m-1}&=( 2^{qs}r^{-qs})^{\frac{p}{q}}
=2^{ps} r^{-ps}
\geq  (r-t)^{-ps}\geq  \min\left\lbrace 1, (r-t)^{-ps} \right\rbrace. 
\end{split}
\end{equation}
Gathering these pieces of information, we obtain the estimate in~\eqref{honda}, as desired.

In particular, from~\eqref{clero},~\eqref{sentiero} and~\eqref{pero} we evince that~\eqref{honda} holds true with~$\bar{c}_{k,q,s}$ given by
\begin{equation}\label{LASTORIADELSIORINTENTO}
\bar{c}_{k,q,s}:=\max\left\lbrace 2^{qs}, \tilde{c}_{k,q,s} \right\rbrace.
\end{equation}
This completes the proof of Proposition~\ref{MLNINT}.
\end{proof}

The next observation provides a uniform bound in~$s$
for~$\overline{c}_{k,q,s}^{\frac{1}{qs}}$. For this, 
we define the constant 
\begin{equation}\label{hatcpm}
\hat{c}_{p,m}:=2e^{c_{p,m}^{(1)}}, 
\end{equation}
where 
\begin{equation}\label{c1pm}
c_{p,m}^{(1)}:=\left(\left[\frac{p}{p-1}\right]+2\right)\prod_{j=1}^{\left[\frac{p}{p-1}\right]+1}\left(\frac{p}{m-1}+j\right).
\end{equation}

\begin{lem}\label{hole:lemma}
Let~$p\in (1,+\infty)$, $m\in \left[p,+\infty\right)$ and~$q:=\frac{p}{m-1}$. 
Let~$\bar{c}_{k,q,s}$ be given by Proposition~\ref{MLNINT}.

Then, for every~$s\in (0,1)$,
\begin{equation*}
\overline{c}_{k,q,s}^{\frac{1}{qs}}\leq \hat{c}_{p,m}. 
\end{equation*}
\end{lem}

\begin{proof}
To show this, from~\eqref{LASTORIADELSIORINTENTO}
(and recalling also~\eqref{contevbdkop} and~\eqref{8132efdkunh}), we evince that 
\begin{equation}\label{cola-cola}
\overline{c}_{k,q,s} =\max\left\lbrace 2^{qs},\tilde{c}_{k,q,s}\right\rbrace =2^{qs}\max\left\lbrace 1, c_{k,q,s} \right\rbrace\leq 2^{qs}\left(1+c_{k,q,s}\right).
\end{equation}

Moreover, we notice that, by~\eqref{deffk},
\begin{equation*}
k\leq \left[\frac{sp}{p-1}\right]+1.
\end{equation*}   
Therefore, recalling~\eqref{contevbdkop}, we see that
\begin{equation}\label{soimate}
\begin{split}
c_{k,q,s} &\leq (k+1) \prod_{j=0}^k(qs+j)\\
&\leq qs \left(\left[\frac{p}{p-1}\right]+2\right)\prod_{j=1}^{\left[\frac{p}{p-1}\right]+1}\left(\frac{p}{m-1}+j\right)\\
&=qs c_{p,m}^{(1)}.
\end{split}
\end{equation}
{F}rom this and~\eqref{cola-cola} we obtain that, for every~$s\in (0,1)$,
\begin{equation}\label{kitebcgr}
\overline{c}_{k,q,s}^{\frac{1}{qs}}\leq
2 \left(1+c_{k,q,s}\right)^{\frac{1}{qs}}\leq
2\left(1+qs c_{p,m}^{(1)}\right)^{\frac{1}{qs}}.
\end{equation}

Now, we define the function~$z:(0,1)\to (0,+\infty)$ as
\begin{equation*}
z(s):=\left(1+qsc_{p,m}^{(1)}\right)^{\frac{1}{qs}}
\end{equation*}
and we notice that~$z$ is decreasing in~$(0,1)$.
Therefore, from~\eqref{kitebcgr}, we find that
\begin{eqnarray*}
\overline{c}_{k,q,s}^{\frac{1}{qs}}\leq
2z(s)\le 2 \limsup_{s\to0^+}z(s)=  2e^{c_{p,m}^{(1)}},
\end{eqnarray*}
which, together with~\eqref{hatcpm}, gives the desired result.
\end{proof}

The barrier~$w$ in Theorem~\ref{lem:1} will be built by suitably
rescaling and translating the function~$v$ in~\eqref{vdefinition}.
Thus, as a preliminary step in this direction,
we now estimate the~$p$-fractional Laplacian
of~$v$, according to the following statement:

\begin{prop}\label{lemma:BigClaim}
Let~$s\in (0,1)$, $p\in (1,+\infty)$, $m\in \left[p,+\infty\right)$, $q:=\frac{p}{m-1}$ and~$r\in \left[1,+\infty\right)$. Let~$v:\R^n\to [0,1]$ be as in~\eqref{vdefinition}.

Then, there exists~$\widetilde{C}\in (0,+\infty)$ depending only on~$n$, $m$ and~$p$ such that, for every~$x\in B_r$, 
\begin{equation}\label{BigClaim}
-(-\Delta)_p^s v(x)  \leq \frac{\widetilde{C}}{s(1-s)}\left(v(x)+\bar{c}_{k,q,s} r^{-qs}\right)^{m-1}, 
\end{equation} 
where~$\bar{c}_{k,q,s}$ is given in~\eqref{LASTORIADELSIORINTENTO}.
\end{prop}

\begin{proof}
In order to prove~\eqref{BigClaim}, we show that there exists~$\bar{C}\in(0,+\infty)$, depending only on~$n$, $m$ and~$p$, such that, for every~$x\in B_r$,
\begin{equation}\label{l.recao5}
-(-\Delta)_{p}^sv(x) \leq \frac{\bar{C}}{s(1-s)}\left( \left(r-\left|x\right|\right)^{-qs(p-1)-ps}+\left(r-\left|x\right|\right)^{-ps}\right),
\end{equation} 
and also 
\begin{equation}\label{qw.bgte7756}
-(-\Delta)_p^s v(x)\leq \frac{\bar{C}}{s(1-s)}.
\end{equation}
If equations~\eqref{l.recao5} and~\eqref{qw.bgte7756} hold true, then using them together with the estimate in~\eqref{honda} we obtain that, for every~$x\in B_r$,
\begin{equation*}
\begin{split}
-(-\Delta)_p^s v(x)&\leq \frac{\bar{C}}{s(1-s)}\min\left\lbrace 1,  \left(r-\left|x\right|\right)^{-qs(p-1)-ps}+\left(r-\left|x\right|\right)^{-ps} \right\rbrace\\ 
&\leq \frac{2\bar{C}}{s(1-s)}\min\left\lbrace 1,\left(r-|x|\right)^{-ps} \right\rbrace\\
&\leq \frac{2\bar{C}}{s(1-s)}\left(h\left(\left|x\right|\right)+\bar{c}_{k,q,s}r^{-qs}\right)^{m-1}\\
&=\frac{2\bar{C}}{s(1-s)}\left(v(x)+\bar{c}_{k,q,s}r^{-qs}\right)^{m-1},
\end{split}
\end{equation*}
and therefore~\eqref{BigClaim} holds true with~$\widetilde{C}:=2\bar{C}$.

Hence, from now on, we focus on the proof of~\eqref{l.recao5} and~\eqref{qw.bgte7756}.
To accomplish this,
in what follows we denote by~$C\in(0,+\infty)$ a constant depending only on~$p$, $m$ and~$n$ and possibly changing from line to line.

If~$p\in (1,2)$ and~$s\in \left[\frac{p-1}{2p},1\right)$, 
we write that
$$ -(-\Delta)_{p}^s v(x)
\le \int_{B_\rho(x)}\frac{\left(v(y)-v(x)\right)\left|v(y)-v(x)\right|^{p-2}}{\left|x-y\right|^{n+sp}}\,dy
+ 2^{p-1}\|v\|_{L^\infty(\R^n)}\int_{\R^n\setminus B_\rho(x)}\frac{dy}{\left|x-y\right|^{n+sp}}.$$
As a result,
applying Lemma~\ref{ebgfncy654444} with~$\rho:=\frac{r-\left|x\right|}{2}$
we find that 
\begin{equation*}
\begin{split}
-(-\Delta)_{p}^s v(x)
&\le  \frac{ C_{n,p,m}}{1-s}\rho^{-qs(p-1)-ps}+\frac{c_{n,p}}{s} \rho^{-ps}\\
&\leq \frac{C}{s(1-s)}\left( (r-\left|x\right|)^{-qs(p-1)-ps}+(r-\left|x\right|)^{-ps} \right),
\end{split}
\end{equation*}
which establishes~\eqref{l.recao5}, while making use
of Corollary~\ref{bfdclaop74r-09} with~$\rho=\mu:=1$ we obtain that
\begin{equation*}
-(-\Delta)_p^s v(x)\leq \frac{C_{1,p,m,n}}{p-sp}+\frac{c_{n,p}}{s}\leq \frac{C}{s(1-s)},
\end{equation*}
which proves~\eqref{qw.bgte7756} in the case~$p\in (1,2)$ and~$s\in \left[\frac{p-1}{2p},1\right)$. 

Now we let~$p\in (1,2)$ and~$s\in \left(0,\frac{p-1}{2p}\right)$. In this setting,
recalling~\eqref{deffk} and~\eqref{vdefinition}, we have that~$k=0$ and
\begin{equation*}
v(y) :=\begin{dcases}
0\quad &\mbox{if}\quad y\in \overline{B_{\frac{r}{2}}},\\
(r-|y|)^{-qs} -\left(\frac{r}{2}\right)^{-qs} \quad &\mbox{if}\quad y\in B_{\hat{r}}\setminus \overline{B_{\frac{r}{2}}},\\
1&\mbox{if}\quad y\in \R^n\setminus \overline{B}_{\hat{r}},
\end{dcases}\end{equation*}
where~$\hat{r}\in\left(\frac{r}{2},r\right)$ is defined as
\begin{equation}\label{g874rfegt56}
\hat{r}:=\min\left\lbrace t\in \left(\frac{r}{2},r\right)\mbox{  s.t.  } f(t) \geq 1 \right\rbrace.
\end{equation}
Accordingly,
\begin{equation}\label{gradv}
\nabla v(y)=\begin{dcases}
0 \quad &\mbox{if}\quad y\in B_{\frac{r}{2}},\\
qs(r-\left|y\right|)^{-qs-1}\frac{y}{\left|y\right|} \quad &\mbox{if}\quad y\in B_{\hat{r}}\setminus \overline{B_{\frac{r}{2}}},\\
0 \quad &\mbox{if}\quad y\in \R^n\setminus \overline{B}_{\hat{r}},
\end{dcases}
\end{equation}
As a consequence, for a.e.~$y\in \R^n$,
\begin{equation}\label{trenoi965tgv4}
\left|\nabla v(y)\right|\leq qs\left|r-\left|y \right|\right|^{-qs-1}\chi_{B_{\hat{r}}\setminus \overline{B_{\frac{r}{2}}}}
.
\end{equation}

Also, if~$y\in  B_{\frac{r-|x|}{2}}(x)$ then
\begin{equation*}
\left|y\right|\leq \left|y-x\right|+\left|x\right|\leq \frac{r-|x|}{2}+\left|x\right|\leq \frac{r+|x|}{2}< r,
\end{equation*}
and thus 
\begin{equation}\label{koprec5410}
\left|r-\left|y\right|\right|= r-\left|y\right|\geq \frac{r-\left|x\right|}{2}.
\end{equation}
Hence, using~\eqref{trenoi965tgv4} and~\eqref{koprec5410} we obtain that, if~$y\in B_{\frac{r-\left|x\right|}{2}}(x)$,
\begin{equation*}
\left|\nabla v(y)\right|\leq qs 2^{qs+1} \left(r-\left|x\right|\right)^{-qs-1},
\end{equation*} and therefore
\begin{equation}\label{C2est}
\begin{split}
\left\|\nabla v\right\|_{L^\infty\left(B_{\frac{r-|x|}{2}}(x)\right)}<q 2^{q+1}(r-|x|)^{-qs-1}.
\end{split}
\end{equation}

We thus apply Lemma~\ref{crocodile} with 
\begin{equation*}
\widehat{K}:=\left\| \nabla v \right\|_{L^\infty\left(B_{\frac{r-\left|x\right|}{2}}\right)}\quad\mbox{and}\quad \rho:=\frac{r-\left|x\right|}{2}
\end{equation*}
and, thanks to~\eqref{C2est}, we find that
\begin{equation*}
\begin{split}
\left|(-\Delta)_{p}^sv(x)\right|&\leq  \frac{c_{n,p}}{s(p-1-ps)}\left(\left\| \nabla v \right\|_{L^\infty\left(B_{\frac{r-\left|x\right|}{2}}\right)}^{p-1} \left(\frac{r-\left|x\right|}{2}\right)^{-ps+p-1}+\left(\frac{r-\left|x\right|}{2}\right)^{-ps}\right)\\
& \leq \frac{C}{s(1-s)}\left( \left(r-\left|x\right|\right)^{-qs(p-1)-ps}+\left(r-\left|x\right|\right)^{-ps}\right).
\end{split}
\end{equation*}
This concludes the proof of~\eqref{l.recao5} for the case~$p\in (1,2)$ and~$s\in \left(0,\frac{p-1}{2p}\right)$.

Furthermore, from~\eqref{trenoi965tgv4} and Proposition~\ref{telviv}
we evince that, for every~$y\in \R^n$, 
\begin{equation*}
\begin{split}
\left|\nabla v(y) \right|\leq qs(r-\hat{r})^{-qs-1}\leq q \widehat{C}_{p,m}^{-qs-1}\leq C,
\end{split}
\end{equation*}
and so~$
\left\|\nabla v\right\|_{L^\infty(\R^n)}\leq C$.
This allows us to use  Lemma~\ref{crocodile} with~$\widehat{K}:=\left\|\nabla  v\right\|_{L^\infty(\R^n)}$ and~$\rho:=1$,
obtaining that, for every~$x\in B_r$, 
\begin{equation*}
-(-\Delta)_p^s v(x)\leq \frac{C}{s(1-s)}. 
\end{equation*}
This concludes the proof of~\eqref{qw.bgte7756} for the case~$p\in (1,2)$ and~$s\in \left(0,\frac{p-1}{2p}\right)$. 

Now, we let~$p\in [2,+\infty)$ and~$s\in \left(0,1\right)$. In this case, recalling~\eqref{deffk}, we have that~$k=1$
and
\begin{equation*}
v(y) :=\begin{dcases}
0\quad &\mbox{if}\quad y\in \overline{B_{\frac{r}{2}}},\\
(r-|y|)^{-qs} -\left(\frac{r}{2}\right)^{-qs} 
-qs \left(\frac{r}{2}\right)^{-qs-1} \left(|y|-\frac{r}{2}\right)
\quad &\mbox{if}\quad y\in B_{\hat{r}}\setminus \overline{B_{\frac{r}{2}}},\\
1&\mbox{if}\quad y\in \R^n\setminus \overline{B}_{\hat{r}},
\end{dcases}\end{equation*}
where~$\hat{r}$ has been defined in~\eqref{g874rfegt56}.
We can thereby compute that
\begin{equation}\label{gvcfkio}
\nabla v(y)=\begin{dcases}
0\quad &\mbox{if}\quad y\in\overline{ B_{\frac{r}{2}}},\\
\frac{y}{\left|y\right|}\left(  qs(r-\left|y\right|)^{-qs-1} -qs \left(\frac{r}{2}\right)^{-qs-1}\right)\quad &\mbox{if}\quad y\in B_{\hat{r}}\setminus \overline{B_{\frac{r}{2}}},\\
0\quad &\mbox{if}\quad y\in \R^n\setminus B_{\hat{r}}.
\end{dcases}
\end{equation}
Consequently, for every~$y\in B_r$, 
\begin{equation}\label{7t54983r2pihfbkjwbiuf4}
\left|\nabla v(y)\right|\leq qs\left(
(r-\left|y\right|)^{-qs-1} - \left(\frac{r}{2}\right)^{-qs-1}\right)
\chi_{B_{\hat{r}}\setminus \overline{B_{\frac{r}{2}}}}\leq C\left(r-\left|y\right|\right)^{-qs-1}\chi_{B_{\hat{r}}\setminus \overline{B_{\frac{r}{2}}}}.
\end{equation}
{F}rom this and~\eqref{koprec5410} we infer that
\begin{equation}\label{DFVre6712}
\left\|\nabla v\right\|_{L^\infty\left(B_{\frac{r-\left|x\right|}{2}}(x)\right)}\leq C \left(r-\left|x\right|\right)^{-qs-1}.
\end{equation}
Also, \eqref{7t54983r2pihfbkjwbiuf4} and Proposition~\ref{telviv} give that
\begin{equation}\label{iolsacr}
\left\|\nabla v\right\|_{L^\infty(\R^n)}\leq
C \left(r-\hat{r}\right)^{-qs-1}
\leq C. 
\end{equation}

Moreover, for every~$y\in B_{\hat{r}}\setminus B_{\frac{r}{2}}$ and~$\mu$, $j\in \{ 1,\dots, n \}$,
\begin{equation*}
\begin{split}
D_{j,\mu}^2 v(y)=qs\left(\frac{\delta_{\mu,j}}{\left|y\right|}-\frac{y_\mu y_j}{\left|y\right|^3}\right)\left((r-\left|y\right|)^{-qs-1}-\left(\frac{r}{2}\right)^{-qs-1} \right)+ \frac{y_{\mu}y_j}{\left|y\right|^2}qs(qs+1)(r-\left|y\right|)^{-qs-2},
\end{split}
\end{equation*}
and thus, for every~$y\in B_{\hat{r}}\setminus \overline{B}_{\frac{r}{2}}$,
\begin{equation}\label{dv2tagss}
\left|D^2 v(y)\right|\leq C\left(\frac{ \left(r-\left|y\right|\right)^{-qs-1}}{\left|y\right|}+(r-\left|y\right|)^{-qs-2}\right) \leq C\left(r-\left|y\right|\right)^{-qs-2}. 
\end{equation}
Accordingly,
making use of this and~\eqref{koprec5410},
\begin{equation}\label{DRFE2389}
\left\| D^2v  \right\|_{L^\infty\left(B_{\frac{r-\left|x\right|}{2}}(x)\cap B_{\hat{r}}\right)}\leq C\left(r-\left|x\right|\right)^{-qs-2}.
\end{equation} 
In addition, 
we deduce from~\eqref{dv2tagss} and Proposition~\ref{telviv} that 
\begin{equation}\label{olepper43cv6}
\left\|D^2 v\right\|_{L^\infty(B_{\hat{r}})}\leq  C\left(r-\hat{r}\right)^{-qs-2}\leq C.
\end{equation}

Now, we set 
\begin{equation*}
\Sigma(y):=\begin{dcases}
\nabla v(y)\quad &\mbox{if}\quad y\in B_{\hat{r}},\\
0\quad &\mbox{if}\quad y\in \R^n\setminus B_{\hat{r}}
\end{dcases}
\end{equation*}
and we claim that there exists~$C\in(0,+\infty)$ such that, for every~$y\in B_{\frac{r-\left|x\right|}{2}}(x)$,
\begin{equation}\label{pqasnyrve56}
v(y)-v(x)-\Sigma(x)\cdot(y-x)\leq C\left(r-\left|x\right|\right)^{-qs-2}\left|x-y\right|^2
\end{equation} 
and, for every~$y\in \R^n$,
\begin{equation}\label{illliadfre50123}
v(y)-v(x)-\Sigma(x)\cdot(y-x)\leq C\left|x-y\right|^2.
\end{equation}
To check~\eqref{pqasnyrve56} and~\eqref{illliadfre50123} we proceed as follows.
If~$x\in B_r \setminus B_{\hat{r}}$, the left-hand sides of both~\eqref{pqasnyrve56} and~\eqref{illliadfre50123} are non positive, so the inequalities trivially hold. If instead~$x$, $y\in B_{\hat{r}}$ the inequality in~\eqref{illliadfre50123} follows from~\eqref{olepper43cv6}. Analogously, if~$x\in B_{\hat{r}}$ and~$y\in B_{\hat{r}}\cap B_{\frac{r-\left|x\right|}{2}}(x)$ the inequality in~\eqref{pqasnyrve56} follows from~\eqref{DRFE2389}. 

Now we consider a radial smooth extension~$\bar{v}$ of~$v$ outside~$B_{\hat{r}}$ such that~$1\leq \bar{v} \leq 2$ and~$D^2\bar{v}(y)\leq C$ for every~$y\in \R^n\setminus B_{\hat{r}} $.
Then, from~\eqref{olepper43cv6}, we have that if~$x\in B_{\hat{r}}$, for every~$y\in \R^n \setminus B_{\hat{r}}$, 
\begin{equation*}
\begin{split}
v(y)-v(x)-\Sigma(x)\cdot(y-x)&=1-\bar{v}(x)-\nabla  \bar{v}(x)\cdot(y-x)\\
&\leq \bar{v}(y)-\bar{v}(x)-\nabla \bar{v}(x)\cdot(y-x)\\
&\leq C\left|x-y\right|^2.
\end{split}
\end{equation*}
This concludes the proof of~\eqref{illliadfre50123}. 

Similarly, if we define the extension
\begin{equation*}
\hat{v}(y):=\begin{dcases}
v(y)\quad &\mbox{if}\quad y\in B_{\hat{r}},\\
f(\left|y\right|)\quad &\mbox{if}\quad y\in B_{r}\setminus  B_{\hat{r}},\\
1\quad &\mbox{if}\quad y\in \R^n\setminus B_{r},
\end{dcases}
\end{equation*} 
we notice that, by~\eqref{DRFE2389}
and the definition of~$f$ in~\eqref{fuhweoiu34ii43ty43ugijkebjwk7684903},
if~$x\in B_{\hat{r}}$, 
\begin{equation*}
\left\| D^2 \hat{v}\right\|_{L^\infty\left(B_{\frac{r-\left|x\right|}{2}}(x)\right)}\leq C\left(r-\left|x\right|\right)^{-qs-2}. 
\end{equation*}
Then, for every~$y\in B_{\frac{r-\left|x\right|}{2}}(x)\setminus B_{\hat{r}}$,
\begin{equation*}
\begin{split}
v(y)-v(x)-\Sigma(x)\cdot(y-x)&=1-\hat{v}(x)-\nabla \hat{v}(x)\cdot(y-x)\\
&\leq \hat{v}(y)-\hat{v}(x)-\nabla \hat{v}(x)\cdot(y-x)\\
&\leq  C\left(r-\left|x\right|\right)^{-qs-2}\left|x-y\right|^2.
\end{split}
\end{equation*}
This concludes the proof of~\eqref{pqasnyrve56}.

Now, we apply Lemma~\ref{p2s01} with
\begin{equation*}
\rho:=\frac{r-\left|x\right|}{2},\quad \widehat{K}:=\left\|\nabla v\right\|_{L^\infty \left(B_{\frac{r-\left|x\right|}{2}}(x)\right)}\quad\mbox{and}\quad \widetilde{C}:=C\left(r-\left|x\right|\right)^{-qs-2}
\end{equation*}
and obtain, thanks to~\eqref{DFVre6712} and~\eqref{pqasnyrve56}, that, if~$x\in B_r$,
\begin{equation*}
\begin{split}&
-(-\Delta)_p^s v(x) \\&\leq \frac{c_{n,p}}{s(1-s)}\left( \widetilde{C} \max\left\lbrace \left\|\nabla v\right\|_{L^\infty \left(B_{\frac{r-\left|x\right|}{2}}(x)\right)},\left|\Sigma(x)\right| \right\rbrace^{p-2}
\left(r-\left|x\right|\right)^{-sp+p-qs-2}                               
+\left(r-\left|x\right|\right)^{-sp} \right)\\
&\leq  \frac{c_{n,p}}{1-s}\left( \widetilde{C}  \left\|\nabla v\right\|_{L^\infty \left(B_{\frac{r-\left|x\right|}{2}}(x)\right)}^{p-2}\left(r-\left|x\right|\right)^{-sp+p-qs-2} +\left(r-\left|x\right|\right)^{-sp} \right)\\
&\leq \frac{C}{1-s}\left( \left(r-\left|x\right|\right)^{-(qs+1)(p-2)-sp+p-qs-2}                                +(r-\left|x\right|)^{-sp}\right)\\
&=\frac{C}{1-s}\left(\left(r-\left|x\right|\right)^{-qs(p-1)-sp}+(r-\left|x\right|)^{-sp}\right).
\end{split}
\end{equation*}
This concludes the proof of~\eqref{l.recao5} when~$p\in [2,+\infty)$ and~$s\in \left[\frac{p-1}{p},1\right)$. 

Finally, applying Lemma~\ref{p2s01} with 
\begin{equation*}
\rho:=1\quad {\mbox{ and }}\quad \widehat{K}:=\left\|\nabla v\right\|_{L^\infty(\R^n)},
\end{equation*}
and recalling also equations~\eqref{iolsacr} and~\eqref{illliadfre50123}, we obtain that, if~$x\in B_r$,
\begin{equation*}
-(-\Delta)_p^s v(x)\leq \frac{C}{1-s}.
\end{equation*}
This concludes the proof of~\eqref{qw.bgte7756} when~$p\in [2,+\infty)$ and~$s\in\left(0,1\right)$. 
\end{proof}

With the work done so far, we can now complete the construction of the barrier
in Theorem~\ref{lem:1}.

\begin{proof}[Proof of Theorem~\ref{lem:1}]
Let~$\bar{c}_{k,q,s}$ and~$\widetilde{C}$ be as in
Propositions~\ref{MLNINT} and~\ref{lemma:BigClaim} respectively.
We define the constants  
\begin{equation}\label{cntefg543888}
C_0:=\left(\frac{\widetilde{C}}{s(1-s)\tau}\right)^{\frac{1}{ps}}\quad\mbox{and}\quad \beta:=2\bar{c}_{k,q,s} r^{-qs}.
\end{equation}
Also, we rescale and translate~$v$ as follows 
\begin{equation}\label{defofw}
w(x):=(2-\beta)v\left(\frac{x}{C_0}\right)+\beta-1.
\end{equation}
Moreover, we set 
\begin{equation}\label{gldini}
r:=\frac{R}{C_0} 
\end{equation}
and we define the constants 
\begin{equation}
\widehat{C}_{n,p,m}:=\hat{c}_{p,m}^p \widetilde{C}\quad\mbox{and}\quad\bar{R}_{s,\tau,n,p,m}:=\left(\frac{\widehat{C}_{n,p,m}}{s(1-s)\tau}\right)^\frac{1}{ps},
\end{equation}
where~$\hat{c}_{p,m}\in(1,+\infty)$ is provided in~\eqref{hatcpm}. 

We prove that~$w$, as defined in~\eqref{defofw},
satisfies the properties listed in the statement of Theorem~\ref{lem:1}.

We first observe that, for every~$R\in  \left[\bar{R}_{s,\tau,n,p,m},+\infty\right)$,
if~$x\in  \R^n\setminus B_R$ then
$$ \frac{|x|}{C_0}\ge \frac{R}{C_0}=r
$$ and so, recalling the definition of~$v$ in~\eqref{vdefinition}, we have that
$$ w(x)=(2-\beta)v\left(\frac{x}{C_0}\right)+\beta-1=
2-\beta+\beta-1 =1.$$
This establishes~\eqref{primacosauwwgyg098765}.

We now check that~\eqref{grw} holds true for every~$R\in  \left[\bar{R}_{s,\tau,n,p,m},+\infty\right)$.
For this, we make use of Lemma~\ref{hole:lemma} (recall also that~$\hat{c}_{p,m}>1$ thanks to~\eqref{hatcpm})
to deduce that
\begin{equation*}
R \geq \bar{R}_{s,\tau,n,p,m}=C_0\hat{c}_{p,m}^\frac{1}{s}\geq C_0\hat{c}_{p,m} \geq C_0\bar{c}_{k,q,s}^{\frac{1}{qs}}. 
\end{equation*}
In light of this, \eqref{cntefg543888} and~\eqref{gldini} we obtain that
$$ \beta = 2\bar{c}_{k,q,s}\left(\frac{C_0}{R}\right)^{qs}\le 2.$$
{F}rom this and~\eqref{defofw} we infer that  
\begin{equation}\label{solfrop}
-1< w \leq 1.
\end{equation}

Moreover, 
recalling the definitions in~\eqref{fuhweoiu34ii43ty43ugijkebjwk7684903}
and~\eqref{vdefinition},
for every~$t\in \left(\frac{r}{2},r\right)$, we have that
\begin{equation*}
\begin{split}
h(t) &\leq f(t)\\ 
&=g(r-t)+\sum_{i=0}^{k} \frac{(-1)^{i+1}}{i!}g^{(i)}\left(\frac{r}{2}\right)\left(t-\frac{r}{2}\right)^i\\
&\leq (r-t)^{-qs}+\sum_{i=0}^k \frac{\left|g^{(i)}(r/2)\right|}{i!}\left(t-\frac{r}{2}\right)^i\\
&\leq (r-t)^{-qs}+\sum_{i=0}^{k} c_{q,s,i}\left(\frac{r}{2}\right)^{-qs-i} \left(\frac{r}{2}\right)^i\\
&= (r-t)^{-qs}+ \tilde{c}_{k,q,s}r^{-qs}, 
\end{split}
\end{equation*}
where~$c_{q,s,i}$ and~$\tilde{c}_{k,q,s}$ have been defined respectively in~\eqref{contevbdkop} and~\eqref{8132efdkunh}. 

{F}rom this we deduce that, for every~$x\in B_r$,
\begin{equation*}
v(x)\leq (r-|x|)^{-qs}+ \tilde{c}_{k,q,s}r^{-qs}.
\end{equation*}
For this reason, for any~$x\in B_R$,
\begin{equation}\label{hulio-09876}\begin{split}&
1+w(x)\leq 2v\left(\frac{x}{C_0}\right)+\beta \leq 2C_0^{qs}(R-\left|x\right|)^{-qs}+ 2\tilde{c}_{k,q,s} r^{-qs}+\beta\\&\qquad
= 2C_0^{qs}(R-\left|x\right|)^{-qs}+ 2\tilde{c}_{k,q,s} r^{-qs}+
2\bar{c}_{k,q,s} r^{-qs}.\end{split}
\end{equation}
Since
\begin{equation*}
r=\frac{R}{C_0}\geq \frac{R-\left|x\right|}{C_0},
\end{equation*}
we obtain from~\eqref{hulio-09876} that, for all~$x\in B_R$,
\begin{equation}\label{dlop12}
\begin{split}
1+w(x)
&\leq 2C_0^{qs}(R-\left|x\right|)^{-qs}+2\tilde{c}_{k,q,s} C_0^{qs} (R-\left|x\right|)^{-qs} +2\bar{c}_{k,q,s} C_0^{qs}(R-\left|x\right|)^{-qs}\\
&\leq C_1(R-\left|x\right|)^{-qs},
\end{split}
\end{equation}
where~$C_1:=2C_0^{qs}\max\left\lbrace  1, \,\tilde{c}_{k,q,s} ,\,\bar{c}_{k,q,s}\right\rbrace =2 C_0^{qs}\bar{c}_{k,q,s}$, 
in light of~\eqref{LASTORIADELSIORINTENTO}.

Now, if~$x\in B_{R-1}$ then~$R-|x|\geq \frac{1+R-|x|}{2}$.
Thus, by~\eqref{dlop12},
\begin{equation}\label{EWAF-lol}
1+w(x)\leq 2^{qs+1}C_{0}^{qs}\bar{c}_{k,q,s} (1+R-\left|x\right|)^{-qs}.
\end{equation}
Additionally, if~$x\in B_R\setminus B_{R-1}$, then~$R+1-|x|\leq 2$.
Thus, by~\eqref{solfrop}, 
\begin{equation}\label{glen-gtrbv-8766}
1+w(x)\leq 2=2^{qs+1}2^{-qs}\leq 2^{qs+1}(1+R-\left|x\right|)^{-qs}.
\end{equation}
{F}rom~\eqref{EWAF-lol} and~\eqref{glen-gtrbv-8766}, and recalling Lemma~\ref{hole:lemma}, we see that, for every~$x\in B_R$,
\begin{equation*}
\begin{split}
1+w(x)&\leq 2^{qs+1}\max\left\lbrace 1, C_{0}^{qs}\bar{c}_{k,q,s} \right\rbrace \left(1+R-\left|x\right|\right)^{-qs}\\
&\leq 2^{q+1}\max\left\lbrace 1, C_0^{qs}\hat{c}_{p,m}^{qs}\right\rbrace \left(1+R-\left|x\right|\right)^{-qs}\\
&\leq 2^{q+1}\max\left\lbrace 1, C_0^{qs}\hat{c}_{p,m}^{q}\right\rbrace \left(1+R-\left|x\right|\right)^{-qs}\\
&=2^{q+1}\max\left\lbrace 1, \bar{R}_{s,\tau,n,p,m}^{qs}\right\rbrace \left(1+R-\left|x\right|\right)^{-qs},
\end{split}
\end{equation*}
which proves~\eqref{grw}.

Now we prove~\eqref{cet}. For this, we use Lemma~\ref{lemma:BigClaim} and equation~\eqref{cntefg543888} and we see that, for every~$x\in B_R$, 
\begin{equation*}
\begin{split}
-(-\Delta)_p^s w(x)=\;&-(2-\beta)^{p-1}C_0^{-sp}
(-\Delta)_p^s v\left(\frac{x}{C_0}\right)\\
\le\; & (2-\beta)^{p-1}C_0^{-sp} \frac{\widetilde{C}}{s(1-s)}
\left(v\left(\frac{x}{C_0}\right)+\bar{c}_{k,q,s}r^{-qs}\right)^{m-1}\\
=\;&\tau (2-\beta)^{p-1}\left(v\left(\frac{x}{C_0}\right)+\bar{c}_{k,q,s}r^{-qs}\right)^{m-1}\\
\leq\; & \tau  ( 2-\beta)^{m-1}\left(v\left(\frac{x}{C_0}\right)+\bar{c}_{k,q,s}r^{-qs}\right)^{m-1}\\
\leq\; & \tau\left( (2-\beta)v\left(\frac{x}{C_0}\right)+ (2-\beta)\bar{c}_{k,q,s}r^{-qs}  \right)^{m-1}\\
\leq\; & \tau  \left((2-\beta)v\left(\frac{x}{C_0}\right)+2\bar{c}_{k,q,s}r^{-qs} \right)^{m-1}\\
=  \;  & \tau \left(1+w(x)\right)^{m-1}.
\end{split}
\end{equation*}
This establishes the claim in~\eqref{cet} and completes the proof of Theorem~\ref{lem:1}.
\end{proof}

\section{Proofs of Theorem~\ref{th:fracp>=2} and Corollary~\ref{coro-09}}\label{cardi}

In the section we prove the density estimates in Theorem~\ref{th:fracp>=2}
and Corollary~\ref{coro-09}. 

We begin with the following observation: 

\begin{rem}\label{undcond2llop} {\rm
We notice that in~\eqref{Fcon-812345} we can assume, without loss of generality, that~$c_0\in (1,+\infty)$.
Indeed, if~$c_0\in (0,1]$, we can choose some~$\tilde{c}_0\in (1,+\infty)$  and for~$t:=\big(c_0/{\tilde{c}_0}\big)^{\frac{1}{n}}$ we can consider the rescaled function 
\begin{equation*}
\tilde{u}(x):=u(t x).
\end{equation*}
Then, if we define~$\tilde{r}_0:={r_0}/{t}$, by scaling we obtain that
\begin{equation*}
\mathcal{L}^n\left(B_{\tilde{r}_0} \cap \left\lbrace \tilde{u}>\theta_1 \right\rbrace  \right)=\frac{1}{t^n} \mathcal{L}^n\left(B_{r_0}\cap \left\lbrace u>\theta_1 \right\rbrace\right)
> \frac{c_0}{t^n}=\tilde{c}_0.
\end{equation*} 
As a consequence, being~$\tilde{u}$ a minimizer\footnote{In this paper, for any~$r>0$, we use the notation
$$ \Omega_r:=r\Omega:=\big\{ rx \;{\mbox{ with }}\; x\in\Omega\big\}.$$} for~$\mathcal{E}^p_s(\cdot,\Omega_{{1}/{t}})$ with potential~$\widetilde{W}:=t^{sp}W$, according to~\eqref{benz} we obtain the existence of some~$R^*$ and~$\tilde{c}$, depending only on~$n$, $s$, $p$, $c_0$, $W$, $r_0$, $\theta_1$, $\theta_2$ and~$\tilde{c}_0$, such that, for any~$r\geq R^*$ satisfying~$B_{\frac{3}{2}r} \subset \Omega_{{1}/{t}}$,
\begin{equation*}
\begin{split}
\tilde{c}r^n &<c_{m,\theta_*}\int_{B_r\cap \left\lbrace \theta_{*}<\tilde{u}\leq \theta^* \right\rbrace}\left|1+\tilde{u}(x)\right|^m\,dx+
\mathcal{L}^n(B_r\cap \left\lbrace \tilde{u}>\theta_2 \right\rbrace)\\
&= \frac{c_{m,\theta_*}}{t^n}\int_{B_{tr}\cap \left\lbrace \theta_{*}< u \leq \theta^* \right\rbrace}\left|1+u(x)\right|^m\,dx
+\frac{1}{t^n}\mathcal{L}^n\left( B_{t r}\cap \left\lbrace u>\theta_2 \right\rbrace\right)
\end{split}
\end{equation*} 
thus proving the density estimates for~$u$ with constants~$R^* t$ and~$\tilde{c}$.
}\end{rem}

\begin{proof}[Proof of Theorem~\ref{th:fracp>=2}]
Thanks to Remark~\ref{undcond2llop}, we can assume, without loss of generality, that 
\begin{equation}\label{c0grande1}
c_0\in (1,+\infty).  
\end{equation}
Let~$R\in (0,+\infty)$ such that~$B_R\subset\Omega$ and~$w\in C(\R^n,(-1,1])$ such that~$w\equiv 1$ in~$\R^n\setminus B_R$.
A more specific choice of~$w$ will be made
later on in the proof (namely~$w$
will be the barrier constructed in Theorem~\ref{lem:1}
with a suitable choice of the parameter~$\tau$). 

Moreover, we define 
\begin{equation*}
v:=\min\left\lbrace  u,w \right\rbrace
\end{equation*}
and we observe that 
\begin{equation}\label{eq:boundcond}
v=u \quad\mbox{in}\quad \R^n\setminus B_R. 
\end{equation}

To ease notation, we also set
\begin{equation*}
D:=(\R^n\times \R^n)\setminus \big((\R^n\setminus B_{R})\times(\R^n\setminus B_{R})\big).
\end{equation*}
In this way, recalling the definition of the interaction energy~$\mathcal{K}_s^p$
in~\eqref{kinetic},
we see that, for every~$f:\R^n \to \R$,
\begin{equation}\label{ssebgted}
\mathcal{K}_s^p(f,B_R)=\frac{1}{2} \int_{D} \frac{\left|f(x)-f(y)\right|^p}{\left|x-y\right|^{n+sp}}\,dx\,dy.
\end{equation}

Now we introduce the following exponents 
\begin{equation}\label{holeo}
\hat{p}:=\max\left\lbrace 2,p \right\rbrace\quad\mbox{and}\quad l:=\min\left\lbrace 2,p \right\rbrace.
\end{equation}
We claim that, for~$\tau:=\frac{c_1}{p}$, $\bar{R}_{\tau}:=\bar{R}_{s,\tau,n,p,m}\in(0,+\infty)$ given in~\eqref{chnecdloro} and~$w\in C(\R^n,(-1,1])$ given by Theorem~\ref{lem:1},
there exists~$\hat{c}_p\in(0,+\infty)$ such that, for every~$R\in \left[\bar{R}_{\tau},+\infty\right)$ satisfying~$B_{R}\subset\Omega$,
\begin{equation}\label{coleottero}
\begin{split}
\mathcal{K}_{\frac{sl}{2}}^{\hat{p}}(u-v,B_R) \leq \frac{\Lambda+2c_1}{\hat{c}_p} \int_{B_R\cap \left\lbrace u>\max\left\lbrace \theta_*,w \right\rbrace   \right\rbrace} \left|1+w(x)\right|^{m-1} \,dx.
\end{split}
\end{equation}
Here above, $\Lambda$ is given in~\eqref{potential} and~$c_1$ in~\eqref{potential1}.

To prove this, if~$p\in [2,+\infty)$ we use formula~(4.3) in~\cite{cad} to deduce that, for all~$x,y\in \R^n$,
\begin{equation}\label{sflop56}
\begin{split}&
\frac{1}{2^{p-1}-1}\left|(u-v)(x)-(u-v)(y)\right|^p  +\left|v(x)-v(y)\right|^p-\left|u(x)-u(y)\right|^p\\ &\qquad
\leq  p (v(x)-v(y))\left|v(x)-v(y)\right|^{p-2}\big(v(x)-v(y)-(u(x)-u(y))\big),
\end{split}
\end{equation}
while if~$p\in (1,2)$
we use formula~(4.4) in~\cite{cad} to obtain that, for all~$x,y\in\R^n$, 
\begin{equation*}
\begin{split}&
\frac{3p(p-1)}{4^2}\frac{\left|(u-v)(x)-(u-v)(y)\right|^2}{(\left|v(x)-v(y)\right|+\left|u(x)-u(y)\right|)^{2-p}}+\left|v(x)-v(y)\right|^p-\left|u(x)-u(y)\right|^p \\&\qquad
\leq  p (v(x)-v(y))\left|v(x)-v(y)\right|^{p-2}\big(v(x)-v(y)-(u(x)-u(y))\big). 
\end{split}
\end{equation*}
Since~$\left\|v\right\|_{L^\infty(\R^n)}$, $\left\|u\right\|_{L^\infty(\R^n)}\leq 1$, from the last equation we obtain that, for every~$p\in (1,2)$ and~$x,y\in\R^n$,
\begin{equation}\label{eqr}
\begin{split}&
\frac{3p(p-1)}{4^{4-p}} \left|(u-v)(x)-(u-v)(y)\right|^2 +\left|v(x)-v(y)\right|^p-\left|u(x)-u(y)\right|^p \\&\qquad
\leq p (v(x)-v(y))\left|v(x)-v(y)\right|^{p-2}\big(v(x)-v(y)-(u(x)-u(y))\big).
\end{split}
\end{equation}

Therefore, we define the constant 
\begin{equation*}
\hat{c}_p:=\begin{dcases} \frac{1}{2^{p-1}}\quad &\mbox{if}\quad p\in [2,+\infty),\\
\frac{3p(p-1)}{4^{4-p}}\quad &\mbox{if}\quad p\in (1,2)
\end{dcases}
\end{equation*} and we infer
from~\eqref{sflop56} and~\eqref{eqr} that, for every~$p\in(1,+\infty)$
and for every~$x,y \in\R^n$, 
\begin{equation}\label{fygduhqyu3r7832ifghegf3748gfveh}
\begin{split}&
\hat{c}_p \left|(u-v)(x)-(u-v)(y)\right|^{\hat{p}} +\left|v(x)-v(y)\right|^p-\left|u(x)-u(y)\right|^p \\
&\qquad \leq p (v(x)-v(y))\left|v(x)-v(y)\right|^{p-2}\left(v(x)-v(y)-(u(x)-u(y))\right).
\end{split}
\end{equation}

Now, recalling~\eqref{ssebgted}, we obtain that
\begin{eqnarray*}
&& \hat{c}_p \mathcal{K}_{\frac{s l}{2}}^{\hat{p}}(u-v,B_R)+\mathcal{K}_{s}^p(v,B_R)-\mathcal{K}_s^p(u,B_R)\\
&=&\frac{\hat{c}_p}{2}\int_{D} \frac{\left|(u-v)(x)-(u-v)(y)\right|^{\hat{p}}}{\left|x-y\right|^{n+\hat{p}\frac{sl}{2}}}\,dx\,dy\\&&\qquad+\frac{1}{2}\int_{D} \frac{\left|v(x)-v(y)\right|^p}{\left|x-y\right|^{n+sp}}\,dx\,dy-\frac{1}{2}\int_{D} \frac{\left|u(x)-u(y)\right|^p}{\left|x-y\right|^{n+sp}}\,dx\,dy.
\end{eqnarray*}
We remark that, in light of~\eqref{holeo},
$$  \hat{p}\frac{sl}{2}=\frac{2sp}{2}=sp,$$
and therefore
\begin{eqnarray*}
&& \hat{c}_p \mathcal{K}_{\frac{s l}{2}}^{\hat{p}}(u-v,B_R)+\mathcal{K}_{s}^p(v,B_R)-\mathcal{K}_s^p(u,B_R)\\
&=&\frac{\hat{c}_p}{2}\int_{D} \frac{\left|(u-v)(x)-(u-v)(y)\right|^{\hat{p}}}{\left|x-y\right|^{n+sp}}\,dx\,dy\\&&\qquad+\frac{1}{2}\int_{D} \frac{\left|v(x)-v(y)\right|^p}{\left|x-y\right|^{n+sp}}\,dx\,dy-\frac{1}{2}\int_{D} \frac{\left|u(x)-u(y)\right|^p}{\left|x-y\right|^{n+sp}}\,dx\,dy.
\end{eqnarray*}
Employing~\eqref{fygduhqyu3r7832ifghegf3748gfveh}
and then recalling~\eqref{eq:boundcond} we thereby obtain that
\begin{eqnarray*}
&& \hat{c}_p \mathcal{K}_{\frac{s l}{2}}^{\hat{p}}(u-v,B_R)+\mathcal{K}_{s}^p(v,B_R)-\mathcal{K}_s^p(u,B_R)\\&
\leq & \frac{p}{2} \int_{D} \frac{(v(x)-v(y))\left|v(x)-v(y)\right|^{p-2}\left(v(x)-v(y)-(u(x)-u(y))\right)}{\left|x-y\right|^{n+sp}}\,dx\,dy
\\&=& \frac{p}{2}\int_{\R^n\times\R^n} \frac{(v(x)-v(y))\left|v(x)-v(y)\right|^{p-2}\left(v(x)-v(y)-(u(x)-u(y))\right)}{\left|x-y\right|^{n+sp}}\,dx\,dy\\
&= & p \int_{\R^n\times\R^n} \frac{(v(y)-v(x))\left|v(x)-v(y)\right|^{p-2}\left(u(x)-v(x)\right)}{\left|x-y\right|^{n+sp}}\,dx\,dy.
\end{eqnarray*}
As a result, making again use of~\eqref{eq:boundcond} and of the fact that~$w\equiv1$ in~$\R^n\setminus B_R$,
\begin{equation}\label{987654jhgfdjhgfdnbvcytr}\begin{split}
& \hat{c}_p \mathcal{K}_{\frac{s l}{2}}^{\hat{p}}(u-v,B_R)+\mathcal{K}_{s}^p(v,B_R)-\mathcal{K}_s^p(u,B_R)\\
\leq\;& p \int_{\R^n}\left(
(u(x)-v(x))\int_{\R^n} \frac{(v(y)-v(x))\left|v(x)-v(y)\right|^{p-2}}{\left|x-y\right|^{n+sp}}\,dy\right)\,dx\\
=\;& p \int_{\left\lbrace u>v=w\right\rbrace}\left(
(u(x)-v(x))\int_{\R^n} \frac{(v(y)-w(x))\left|w(x)-v(y)\right|^{p-2}}{\left|x-y\right|^{n+sp}}\,dy\right)\,dx\\
=\;& p \int_{B_R\cap \left\lbrace u>v=w\right\rbrace}\left(
(u(x)-v(x))\int_{\R^n} \frac{(v(y)-w(x))\left|w(x)- v(y)\right|^{p-2}}{\left|x-y\right|^{n+sp}}\,dy\right)\,dx.
\end{split}\end{equation}

Now, we define the function~$\phi_p(t):=t|t|^{p-2}$ and we observe that~$\phi_p$ is nondecreasing, since
$$ \phi_p'(t)=|t|^{p-2}+t(p-2) |t|^{p-4} t= |t|^{p-2}+(p-2) |t|^{p-2} =
(p-1) |t|^{p-2} \ge0.
$$
As a consequence, since~$v(y)\le w(y)$, we see that
$$ \phi_p\big(v(y)-w(x)\big)\le \phi_p\big(w(y)-w(x)\big). $$
Plugging this information into~\eqref{987654jhgfdjhgfdnbvcytr}, we thus find that
\begin{eqnarray*}
&& \hat{c}_p \mathcal{K}_{\frac{s l}{2}}^{\hat{p}}(u-v,B_R)+\mathcal{K}_{s}^p(v,B_R)-\mathcal{K}_s^p(u,B_R)\\
&\le & p \int_{B_R\cap \left\lbrace u>v=w\right\rbrace}\left(
(u(x)-v(x))\int_{\R^n} \frac{(w(y)-w(x))\left|w(x)- w(y)\right|^{p-2}}{\left|x-y\right|^{n+sp}}\,dy\right)\,dx\\
&= & -p\int_{B_R\cap \left\lbrace u>v=w\right\rbrace}(u(x)-v(x))(-\Delta)_p^s w(x)\,dx.
\end{eqnarray*}

{F}rom this and the definition of the energy in~\eqref{energy}
we obtain that 
\begin{equation*}
\begin{split}
\hat{c}_p \mathcal{K}_{\frac{sl}{2}}^{\hat{p}}(u-v, B_R)
\le\;& \mathcal{E}(u,B_R)-\mathcal{E}(v,B_R)+\int_{B_R} W(v(x))-W(u(x))\,dx\\
& \qquad-p\int_{B_R\cap \left\lbrace u>v=w\right\rbrace}(u(x)-v(x))(-\Delta)_p^s w(x)\,dx.
\end{split}
\end{equation*}
Thus, the minimality of~$u$ in~$B_R$ and~\eqref{eq:boundcond} give that
\begin{equation}\label{7t584923h2345678hgfjdhueriwu}
\begin{split}&
\hat{c}_p \mathcal{K}_{\frac{sl}{2}}^{\hat{p}}(u-v, B_R) \\
\leq\;& \int_{B_R} W(v(x))-W(u(x))\,dx-p\int_{B_R\cap \left\lbrace u>v=w\right\rbrace}(u(x)-v(x))(-\Delta)_p^s w(x)\,dx\\
=\;&\int_{B_R \cap \left\lbrace u>w \right\rbrace} W(w(x))-W(u(x))\,dx-p\int_{B_R\cap \left\lbrace u>v=w\right\rbrace}(u(x)-v(x))(-\Delta)_p^s w(x)\,dx.
\end{split}
\end{equation}

Now, we observe that, thanks to~\eqref{potential1}, if~$x\in\{u>w\}$,
\begin{eqnarray*}
&& W(u(x))- W(w(x))=\int^{u(x)}_{w(x)} W'(\xi)\,d\xi\ge
c_1 \int^{u(x)}_{w(x)} |1+\xi|^{m-1}\,d\xi\\&&\qquad \qquad \ge
c_1 |1+w(x)|^{m-1}(u(x)-w(x)).
\end{eqnarray*}
Therefore, using also~\eqref{potential},
\begin{equation*}
\begin{split}
& \int_{B_R\cap \left\lbrace u>w \right\rbrace} W(w(x))-W(u(x))\,dx\\
=\;&\int_{B_R\cap \left\lbrace w<u\leq \theta_* \right\rbrace}W(w(x))-W(u(x))\,dx+ \int_{B_R\cap \left\lbrace u>\max\left\lbrace \theta_*,w \right\rbrace   \right\rbrace} W(w(x))-W(u(x)) \,dx\\
\leq\; & -c_1\int_{B_R\cap \left\lbrace w<u\leq \theta_* \right\rbrace}\left|1+w(x)\right|^{m-1}(u(x)-w(x))\,dx+ \Lambda \int_{B_R\cap \left\lbrace u>\max\left\lbrace \theta_*,w \right\rbrace   \right\rbrace} \left|1+w(x)\right|^m \,dx .
\end{split}
\end{equation*}
Accordingly, using this fact into~\eqref{7t584923h2345678hgfjdhueriwu},
\begin{equation*}
\begin{split}&
\hat{c}_p \mathcal{K}_{\frac{sl}{2}}^{\hat{p}}(u-v, B_R)\\
\leq\;&  -c_1\int_{B_R\cap \left\lbrace w<u\leq \theta_* \right\rbrace}\left|1+w(x)\right|^{m-1}(u(x)-w(x))\,dx+ \Lambda \int_{B_R\cap \left\lbrace u>\max\left\lbrace \theta_*,w \right\rbrace   \right\rbrace} \left|1+w(x)\right|^{m-1} \,dx\\
&\qquad-p\int_{B_R\cap \left\lbrace u>v=w\right\rbrace}(u(x)-v(x))(-\Delta)_p^s w(x)\,dx\\
=\;&-c_1\int_{B_R\cap \left\lbrace u>w \right\rbrace}\left|1+w(x)\right|^{m-1}(u(x)-w(x))\,dx\\&\qquad
+c_1 \int_{B_R\cap \left\lbrace u>\max\left\lbrace w,\theta_*\right\rbrace \right\rbrace}\left|1+w(x)\right|^{m-1}(u(x)-w(x))\,dx\\
 &\qquad + \Lambda \int_{B_R\cap \left\lbrace u>\max\left\lbrace \theta_*,w \right\rbrace   \right\rbrace} \left|1+w(x)\right|^{m-1} \,dx-p\int_{B_R\cap \left\lbrace u>v=w\right\rbrace}(u(x)-v(x))(-\Delta)_p^s w(x)\,dx.
\end{split}
\end{equation*}

The computations done so far hold true for any function~$w\in C(\R^n,(-1,1])$ such that~$w\equiv 1$ in~$\R^n\setminus B_R$. We now
choose~$w$ as in Theorem~\ref{lem:1} with~$\tau:=\frac{c_1}{p}$.
Then, if~$\bar{R}_{\tau}:=\bar{R}_{s,\tau,n,p,m}\in(0,+\infty)$ is given as in~\eqref{chnecdloro},
thanks to Theorem~\ref{lem:1}
we obtain that, for every~$R\in\left[\bar{R}_{\tau},+\infty\right)$ such that~$B_R\subset\Omega$,
\begin{equation*}
\begin{split}&
\hat{c}_{p} \mathcal{K}_{\frac{sl}{2}}^{\hat{p}}(u-v,B_R)\\ 
\leq\;&  -c_1\int_{B_R\cap \left\lbrace u>w \right\rbrace}\left|1+w(x)\right|^{m-1}(u(x)-w(x))\,dx\\&\qquad
+c_1 \int_{B_R\cap \left\lbrace u>\max\left\lbrace w,\theta_*\right\rbrace \right\rbrace}\left|1+w(x)\right|^{m-1}(u(x)-w(x))\,dx\\
 &\qquad  + \Lambda \int_{B_R\cap \left\lbrace u>\max\left\lbrace \theta_*,w \right\rbrace   \right\rbrace} \left|1+w(x)\right|^{m-1} \,dx+c_1\int_{B_R\cap \left\lbrace u>v=w\right\rbrace}(u(x)-v(x))\left|1+w(x)\right|^{m-1}\,dx\\
=\;& \Lambda \int_{B_R\cap \left\lbrace u>\max\left\lbrace \theta_*,w \right\rbrace   \right\rbrace} \left|1+w(x)\right|^{m-1} \,dx+ c_1 \int_{B_R\cap \left\lbrace u>\max\left\lbrace w, \theta_{*}\right\rbrace\right\rbrace}(u(x)-v(x))\left|1+w(x)\right|^{m-1}\,dx\\
\leq\;&  \left(\Lambda+2c_1\right) \int_{B_R\cap \left\lbrace u>\max\left\lbrace \theta_*,w \right\rbrace   \right\rbrace} \left|1+w(x)\right|^{m-1}.
\end{split}
\end{equation*} 
This concludes the proof of the claim in~\eqref{coleottero}. 

Now, let~$C_{\tau}=C_{s,\tau,n,p,m}$ be as in~\eqref{straimpoconst}
with~$\tau=\frac{c_1}{p}$, and define 
\begin{equation*}
k_0:=\left(\frac{2 C_{\tau} }{1+\theta_*}\right)^{\frac{m-1}{ps}}-1.
\end{equation*}
Note that by the definitions of~$C_{\tau}$ and~$\theta_*$ we have that
\begin{equation}\label{k0estimatao}\begin{split}&
k_0\geq C_{\tau}^\frac{m-1}{ps}-1
= \left(2^{\frac{p}{m-1}+1}\max\left\lbrace 1, \bar{R}_\tau^{\frac{ps}{m-1}}\right\rbrace\right)^\frac{m-1}{ps}-1\\&\qquad 
=2^{\frac{1}{s}+\frac{m-1}{ps}}
\max\left\lbrace 1, \bar{R}_\tau
\right\rbrace-1\ge  \max\left\lbrace 1, \bar{R}_\tau
\right\rbrace .
\end{split}
\end{equation}
Thanks to this and~\eqref{grw},
we have that if~$R\in(k_0,+\infty)$ and~$k\in (k_0,R)$, for all~$x\in B_{R-k}$,
\begin{equation*}
w(x)\leq -1+\frac{C_{\tau }}{(1+R-\left|x\right|)^{\frac{ps}{m-1}}}\leq  -1+\frac{C_{\tau}}{(1+k)^{\frac{ps}{m-1}}}\leq -1+\frac{1+\theta_*}{2}.
\end{equation*}
This gives that, for all~$x\in B_{R-k}\cap\{ u>\theta_* \}$,
\begin{equation}\label{zeroe9}
\left|u(x)-v(x)\right|=u(x)-v(x)\geq u(x)-w(x)\geq \frac{1+\theta_*}{2}. 
\end{equation}

Now, for every~$t\in (0,+\infty)$ we denote by 
\begin{equation}\label{t74r8o21Vdasaunyin}
V(t):=\mathcal{L}^n(B_t\cap \left\lbrace u>\theta_* \right\rbrace).
\end{equation}
With this notation, using~\eqref{zeroe9}, we have that, for every~$R\in (k_0,+\infty)$ and~$k\in (k_0,R)$ satisfying~$B_{R}\subset\Omega$,
\begin{eqnarray*}
\big(V(R-k)\big)^{\frac{n-sp}{n}}&=&\left(\int_{B_{R-k}\cap \left\lbrace u>\theta_*\right\rbrace}\,dx\right)^{\frac{n-sp}{n}}\\
&\leq& \left(\frac{2}{1+\theta_{*}}\right)^{\hat{p}} \left(\int_{B_{R-k}\cap \left\lbrace u>\theta_*\right\rbrace}\left|u(x)-v(x)\right|^{\frac{n\hat{p}}{n-sp}} \,dx\right)^{\frac{n-sp}{n}},
\end{eqnarray*}
where~$\hat{p}$ has been defined in~\eqref{holeo}.

We also denote by~$s':=\frac{sl}{2}$, with~$l$ given in~\eqref{holeo}, and we observe that
$$  sp = \frac{2sp}2 =\frac{\hat{p}sl }{2}=\hat{p}s'.
$$ Hence, we can write that
\begin{equation}\label{mnbvcxasdfghjk987654}\begin{split}
\big(V(R-k)\big)^{\frac{n-sp}{n}}&\leq 
\left(\frac{2}{1+\theta_{*}}\right)^{\hat{p}} \left(\int_{B_{R-k}\cap \left\lbrace u>\theta_*\right\rbrace}\left|u(x)-v(x)\right|^{\frac{n\hat{p}}{n-\hat{p}s'}} \,dx\right)^{\frac{n-\hat{p}s'}{n}}
\\ &\leq \left(\frac{2}{1+\theta_{*}}\right)^{\hat{p}} \left(\int_{\R^n}\left|u(x)-v(x)\right|^{\frac{n\hat{p}}{n-\hat{p}s'}} \,dx\right)^{\frac{n-\hat{p}s'}{n}}.\end{split}\end{equation}

Now, we recall the Sobolev inequality of Theorem~1 in~\cite{MR1940355}, according to which
there exists~$c_{n,p}\in(0,+\infty)$ such that, for every~$f\in W_0^{s,p}(\R^n)$,
\begin{equation*}
\left(\int_{\R^n}\left|f(x)\right|^\frac{pn}{n-sp}\,dx\right)^\frac{n-sp}{n}\leq c_{n,p}\frac{s(1-s)}{(n-sp)^{p-1}}\int_{\R^n}\int_{\R^n}\frac{\left|f(x)-f(y)\right|^p}{\left|x-y\right|^{n+sp}}\,dx\,dy.
\end{equation*}
{F}rom this inequality (used with~$f:=u-v$) and~\eqref{mnbvcxasdfghjk987654}, we deduce that
$$ \big(V(R-k)\big)^{\frac{n-sp}{n}}\le
c_{n,\hat{p}}\frac{s'(1-s')}{(n-s'\hat{p})^{\hat{p}-1}}\left(\frac{2}{1+\theta_{*}}\right)^{\hat{p}}       \int_{\R^n}\int_{\R^n} \frac{\left|(u-v)(x)-(u-v)(y)\right|^{\hat{p}}}{\left|x-y\right|^{n+s'\hat{p}}}\,dy\,dx.$$
Thus, making use of~\eqref{eq:boundcond},
\begin{eqnarray*} \big(V(R-k)\big)^{\frac{n-sp}{n}}&\le&
 c_{n,\hat{p}}\frac{s'(1-s')}{(n-sp)^{\hat{p}-1}}\left(\frac{2}{1+\theta_{*}}\right)^{\hat{p}}\int_{D} \frac{\left|(u-v)(x)-(u-v)(y)\right|^{\hat{p}}}{\left|x-y\right|^{n+s'\hat{p}}}\,dy\,dx
\\ &=& c_{n,\hat{p}}\frac{s'(1-s')}{(n-sp)^{\hat{p}-1}}\left(\frac{2}{1+\theta_{*}}\right)^{\hat{p}}\mathcal{K}_{s'}^{\hat{p}}(u-v,B_R).
\end{eqnarray*}
This and \eqref{coleottero} give that
$$ \big(V(R-k)\big)^{\frac{n-sp}{n}}\le
 c_{n,\hat{p}}\frac{s'(1-s')}{(n-sp)^{\hat{p}-1}}\left(\frac{2}{1+\theta_{*}}\right)^{\hat{p}}\frac{\Lambda+2c_1}{\hat{c}_p}\int_{B_R\cap \left\lbrace u>\max\left\lbrace \theta_*,w \right\rbrace   \right\rbrace} \left|1+w(x)\right|^{m-1} \,dx.$$
Therefore, in light of~\eqref{grw},
\begin{eqnarray*}
&&\big(V(R-k)\big)^{\frac{n-sp}{n}}
\\&\leq& c_{n,\hat{p}}\frac{s'(1-s')}{(n-sp)^{\hat{p}-1}}\left(\frac{2}{1+\theta_{*}}\right)^{\hat{p}}\frac{\Lambda+2c_1}{\hat{c}_p}C_{\tau}^{m-1}\int_{B_R\cap \left\lbrace u>\max\left\lbrace \theta_*,w \right\rbrace   \right\rbrace} (1+R-\left|x\right|)^{-ps} \,dx\\
&\leq& c_{n,\hat{p}}\frac{s'(1-s')}{(n-sp)^{\hat{p}-1}}\left(\frac{2}{1+\theta_{*}}\right)^{\hat{p}}\frac{\Lambda+2c_1}{\hat{c}_p}C_{\frac{c_1}{p}}^{m-1} \int_{B_R\cap \left\lbrace u> \theta_* \right\rbrace} (1+R-\left|x\right|)^{-ps} \,dx
.\end{eqnarray*}
As a result, using the coarea formula and setting
\begin{equation*}
c=c_{s,n,p,m,\Lambda,c_1,\theta_*}:=c_{n,\hat{p}}\frac{s'(1-s')}{(n-sp)^{\hat{p}-1}}\left(\frac{2}{1+\theta_{*}}\right)^{\hat{p}}\frac{\Lambda+2c_1}{\hat{c}_p}C_{\frac{c_1}{p}}^{m-1},
\end{equation*}
we conclude that, for every~$R\in (k_0,+\infty)$ and~$k\in (k_0,R)$ satisfying~$B_{R}\subset\Omega$,
\begin{equation}\label{sncvfbythuii8}
\big(V(R-k)\big)^{\frac{n-sp}{n}}
\leq c\int_{0}^R (1+R-\left|t\right|)^{-ps}V'(t) \,dt.
\end{equation}

Now we integrate~\eqref{sncvfbythuii8} in~$R\in \left[\rho,\frac{3}{2}\rho\right]$ with~$\rho\geq  4k $ and~$B_{\frac{3}{2}\rho}\subset \Omega$. Also, we use the fact that~$s\in \left(0,\frac{1}{p}\right)$ together with~\eqref{k0estimatao} and obtain that  
\begin{equation*}
\begin{split}
\frac{\rho}{2}\big( V(\rho-k)\big)^{\frac{n-sp}{n}} &\leq  \int_{\rho}^{\frac{3}{2}\rho} \big(V(R-k)\big)^\frac{n-sp}{n}\,dR\\
&\leq c\int_{\rho}^{\frac{3}{2}\rho} \left(\int_{0}^R (1+R-\left|t\right|)^{-sp}V'(t) \,dt\right)\,dR\\
&\leq c \int_{0}^{\frac{3}{2}\rho}\left(
V'(t) \int_{t}^{\frac{3}{2}\rho}(1+R-\left|t\right|)^{-sp} \,dR\right)\,dt\\
&=\frac{c }{1-sp} \int_{0}^{\frac{3}{2}\rho}V'(t) \left[\left(1+\frac{3}{2}\rho-\left|t\right|\right)^{1-sp}-1\right]\,dt\\
&\leq  \frac{c }{1-sp}2^{1-sp}\rho^{1-sp}\int_{0}^{\frac{3}{2}\rho}V'(t) \,dt\\
&\leq \frac{2c }{1-sp}\rho^{1-sp} V\left(\frac{3}{2}\rho\right).
\end{split}
\end{equation*}
In particular, setting~$r:=\rho-k$ and recalling that~$\rho \geq 4k$ we deduce that 
\begin{equation}\label{BuffSpri}
r^{sp}\left( V(r)\right)^{\frac{n-sp}{n}}\leq \hat{c} \,V(2r),
\end{equation}
with
$$\hat{c}=\hat{c}_{s,n,p,m,\Lambda,c_1,\theta_*}:=\frac{4c}{1-sp}.$$
{F}rom now~$k>k_0$ is fixed once and for all.

{F}rom~\eqref{Fcon-812345}, \eqref{c0grande1} and~\eqref{BuffSpri} we see that~\eqref{condi1} and~\eqref{epo-vdfceobgt} in Lemma~\ref{GE-XE-JU} hold true with~$V$ as in~\eqref{t74r8o21Vdasaunyin} and with the following choices
\begin{equation*}
\sigma:=sp,\quad \nu:=n,\quad \gamma:=2,\quad \mu:=c_0,\quad R_0:=\max\left\lbrace 3k, r_0 \right\rbrace \quad\mbox{and}\quad C:=\hat{c}, 
\end{equation*}
as far as~$R\in \left[R_0,+\infty\right)$ is such that~$B_{\frac{3R}{2}}\subset\Omega$. 

Therefore, in virtue of~\eqref{pmecb.54tb6} in Lemma~\ref{GE-XE-JU}, we obtain that there exist~$\tilde{c}\in (0,1)$, depending only on~$s$, $n$, $p$, $m$, $\Lambda$, $c_1$, $\theta_*$, $c_0$, $R_0$ and~$k$,
and~$R_*\in \left[R_0,+\infty\right)$, depending only on~$R_0$, such that, for every~$r\in \left[R_*,+\infty\right)$ satisfying~$B_{\frac{3r}{2}}\subset \Omega$,
\begin{equation*}
V(r)\geq \tilde{c} \,r^n.
\end{equation*}
As a result, for any~$r\in \left[R_*,+\infty\right)$ such that~$B_{\frac{3r}{2}} \subset\Omega$,
\begin{equation}\label{euno}
\tilde{c}\, r^n\leq \mathcal{L}^n(B_r\cap \left\lbrace u>\theta_*\right\rbrace)=\mathcal{L}^n(B_r\cap \left\lbrace u>\theta^*\right\rbrace)+\mathcal{L}^n(B_r\cap \left\lbrace \theta_*<u\leq \theta^*\right\rbrace).
\end{equation}

Now we observe that
\begin{equation*}
\begin{split}
\mathcal{L}^n\left(B_r\cap \left\lbrace \theta_{*}<u\leq \theta^* \right\rbrace\right)&=\int_{B_r\cap \left\lbrace \theta_{*}<u\leq \theta^* \right\rbrace}\,dx\\
& \leq (1+\theta_*)^{-m}\int_{B_r\cap \left\lbrace \theta_{*}<u\leq \theta^* \right\rbrace}\left|1+u(x)\right|^m\,dx.
\end{split}
\end{equation*}
Thus, it follows from this and~\eqref{euno} that, for every~$r\in \left[R_*,+\infty\right)$  such that~$B_{\frac{3r}{2}}\subset \Omega$,
\begin{equation*}
\mathcal{L}^n\left(B_r\cap \left\lbrace u>\theta^* \right\rbrace\right)
\geq c\, r^n- (1+\theta_*)^{-m}\int_{B_r\cap \left\lbrace \theta_{*}<u\leq \theta^* \right\rbrace}\left|1+u(x)\right|^m\,dx.
\end{equation*}
This entails the desired result in~\eqref{benz}.
\end{proof}

\begin{proof}[Proof of Corollary~\ref{coro-09}]
Thanks to the claim~\eqref{BFAOTE} in Theorem~\ref{erfvbgt-098}
and the assumption in~\eqref{condiWzione},
we have that, for every~$r\in (0,+\infty)$ such that~$B_{r+2}\subset \Omega$,
\begin{equation}\label{limirelop}
\begin{split}&
\int_{B_r\cap \left\lbrace \theta_{*}<u\leq \theta^* \right\rbrace}\left|1+u(x)\right|^m\,dx \leq \lambda_{\theta^*}^{-1} \int_{B_r\cap \left\lbrace \theta_{*}<u\leq \theta^* \right\rbrace} W(u(x))\,dx\\   
&\qquad\qquad \leq \lambda_{\theta^*}^{-1}\mathcal{E}_s^p(u,B_r)
\leq \lambda_{\theta^*}^{-1}\bar{C}r^{n-sp}.
\end{split}
\end{equation}
Also, we observe that for every~$r\geq 4$
it holds that~$\frac{3}{2}r\geq r+2$. Hence, from~\eqref{limirelop} and~\eqref{benz} we deduce that, for every~$r\geq \max\left\lbrace 4,R^* \right\rbrace$ such that~$B_{\frac{3}{2}r}\subset \Omega$,
\begin{equation*}
\begin{split}
\tilde{c}\,r^n &\leq c_{m,\theta_*}\int_{B_r\cap \left\lbrace \theta_{*}<u\leq \theta^* \right\rbrace}\left|1+u(x)\right|^m\,dx+\mathcal{L}^n(B_r\cap \left\lbrace u>\theta_2 \right\rbrace)\\
&\leq \lambda_{\theta^*}^{-1}c_{m,\theta_*}\bar{C} r^{n-sp}+\mathcal{L}^n(B_r\cap \left\lbrace u>\theta_2 \right\rbrace).
\end{split}
\end{equation*}
{F}rom this, the claim in~\eqref{benzina} readily follows.\qedhere 
\end{proof}

\section{Proofs of Theorem~\ref{gliutgbedgty} and Corollary~\ref{rgeb}}\label{vdcbnhetrslkoiuojhytbirvec}

We begin this section by establishing the compactness result for sequences of functions~$\left\lbrace u_\epsilon\right\rbrace_\epsilon \subset X^{s,p}(\Omega)$ satisfying the boundedness assumption~\eqref{unifbound}, as stated
in Theorem~\ref{gliutgbedgty}. As already mentioned in the introduction, the convergence is a straightforward consequence of the scaling of~$\mathcal{F}_{s,\epsilon}^p$, the uniform bound in~\eqref{unifbound} and the compact embedding of~$W^{s,p}(\Omega')$ into~$L^1(\Omega')$ for~$\Omega' \Subset \Omega$. The details of the proof are as follows.

\begin{proof}[Proof of Theorem~\ref{gliutgbedgty}]
As a consequence of~\eqref{unifbound}, for every open and bounded set~$\Omega'\Subset \Omega$, we have that
\begin{equation}\label{boundonthekinetic}
\sup_{\epsilon\in (0,1)}\int_{\Omega'}\int_{\Omega'}\frac{\left|u_\epsilon(x)-u_\epsilon(y)\right|^p}{\left|x-y\right|^{n+sp}}\,dx\,dy<+\infty.
\end{equation}
Moreover, since~$u_\epsilon\in X^{s,p}(\Omega)$, we find that, for every~$\epsilon\in (0,1)$,
\begin{equation}\label{geliredcdl}
\left\|u_\epsilon\right\|_{L^p(\Omega')}\leq \left|\Omega'\right|^\frac{1}{p}.
\end{equation}
Thus, it follows from~\eqref{boundonthekinetic},
\eqref{geliredcdl} and Theorem~7.1 in~\cite{MR2944369} that there exists some~$u_0\in L^1(\Omega')$ such that, up to a subsequence,
$u_\epsilon$ converges to~$ u_0$ in~$L^1(\Omega')$ and
pointwise a.e. in~$\Omega'$, as~$\epsilon\to0$.

Furthermore, using~\eqref{q874brygtefd}, \eqref{unifbound} and Fatou's Lemma, we have that 
\begin{equation*}
\begin{split}
0&=\lim_{\epsilon\to 0^+} \epsilon^{sp} \mathcal{F}_{\epsilon}(u_\epsilon,\Omega')\\
&=\lim_{\epsilon \to 0^+} \epsilon^{sp} \mathcal{K}_{s}^p(u_\epsilon,\Omega')+\int_{\Omega'}W(u_\epsilon(x))\,dx\\
&\geq c_2 \lim_{\epsilon \to 0^+} \int_{\Omega'} (1-u_{\epsilon}^2(x))^m\,dx\\
&\geq  c_2\int_{\Omega'} (1-u_{0}^2(x))^m\,dx.
\end{split}
\end{equation*}
This yields that~$u_0$ must be equal to~$-1$ and~$1$ a.e. in~$\Omega'$. Accordingly, if we denote by 
\begin{equation*}
E:=\left\lbrace  x\in \Omega'\mbox{  s.t.  } u_0(x)=1  \right\rbrace
\end{equation*} 
we obtain that~$u_0|_{\Omega'}=\chi_{E}-\chi_{E^c}$.
The proof of Theorem~\ref{gliutgbedgty} is thereby complete.
\end{proof}

In what follows we prove the Hausdorff convergence of the interface
stated in Corollary~\ref{rgeb}.
This result follows from the density estimates in Corollary~\ref{coro-09} together with the H\"older regularity result for minimizers of~$\mathcal{E}_s^p$ that we have recalled in Section~\ref{CH-HCBob}. 

\begin{proof}[Proof of Corollary~\ref{rgeb}]
We argue by contradiction, and we assume that there exist~$\Theta\in (0,1)$, $\delta\in(0,+\infty)$, $R\in(0,+\infty)$ and a sequence~$x_k\in \left\lbrace \left|u_{\epsilon_k}\right|<\Theta\right\rbrace\cap B_R$ such that
either~$B_{\delta}(x_k)\subset E$ or~$B_{\delta}(x_k)\subset E^c$. Without loss of generality, we assume that for every~$k$ it holds that~$B_{\delta}(x_k)\subset E^c$.

Then, according to Theorem~\ref{gliutgbedgty}, up to a subsequence we have that
\begin{equation}\label{12345678098765oiuytrkjhgfdzxcvb}
\begin{split}
0 &= \lim_{k\to+\infty}\int_{B_{\delta}(x_k)\cap \Omega}\left|u_{\epsilon_k}(x)-\chi_E(x)+\chi_{E^c}(x)\right|\,dx\\
&=\lim_{k\to+\infty}\int_{B_{\delta}(x_k)\cap \Omega} \left|u_{\epsilon_k}(x)+1\right|\,dx. 
\end{split}
\end{equation}

Now, we define the rescaled set 
\begin{equation*}
\Omega_k:=\left\lbrace \frac{x-x_k}{\epsilon_k}\mbox{  s.t.  }\, x\in \Omega \right\rbrace,
\end{equation*}
and the sequence of functions 
\begin{equation*}
w_k(x):=u_{\epsilon_k}\left(x\epsilon_k+x_k\right). 
\end{equation*}
In this way, $w_k$ is a minimizer for~$\mathcal{E}_s^p$ in~$\Omega_k$. Moreover,
\begin{equation}\label{thnlo}
w_k(0)=u_{\epsilon_k}(x_k)>-\Theta. 
\end{equation}

Now, we apply Theorem~\ref{Cozzi} to~$w_k$ with the choice~$
F(t):=(1-s) W(t)$ and
we obtain that~$w_k\in C_{\textit{loc}}^\alpha(\Omega_k)$ for some~$\alpha\in (0,1)$, depending only
on~$n$, $s$ and~$p$, and
there exists~$C\geq 1$, depending only on~$n$, $s$ and~$p$, such that, for every~$R_0\in\left(0,\frac{\operatorname{dist}\,(0,\partial \Omega_k)}{8}\right)$,
\begin{equation}\label{bdcereghyui}
\begin{split}
\left[w_k\right]_{C^\alpha(B_{R_0})} &\leq \frac{C}{R_0^\alpha}\left[\left\|w_k\right\|_{L^\infty(\R^n)}+ \mbox{Tail}(w_k;0,4R_0)+   R_0^s(1-s)^\frac{1}{p}\left\| W\right\|_{L^\infty((-1,1))}^{\frac{1}{p}}\right]\\
&\leq \frac{C}{R_0^{\alpha}} \left[1+\left((1-s)(4R_0)^{sp}\int_{\R^n\setminus B_{4R_0}}\frac{dy}{\left|y\right|^{n+sp}}\right)^\frac{1}{p-1}+R_0^s(1-s)^\frac{1}{p}\left\| W\right\|_{L^\infty((-1,1))}^{\frac{1}{p}}\right]\\
&=\frac{C}{R_0^\alpha}\left[1+ \left((1-s)\frac{\omega_{n-1}}{sp}\right)^\frac{1}{p-1}+ R_0^s(1-s)^\frac{1}{p}\left\| W\right\|_{L^\infty((-1,1))}^{\frac{1}{p}}\right].
\end{split}
\end{equation} 
Furthermore, if we set~$R_0:=\frac{\operatorname{dist}\,(\partial B_R, \partial \Omega)}{8}\in(0,+\infty)$, then~$R_0\in\left(0,\frac{\operatorname{dist}\,(0,\partial\Omega_k)}{8}\right)$ for every~$k$. {F}rom this, \eqref{thnlo} and~\eqref{bdcereghyui}, it follows that there exists~$r_0$ such that~$w_k(x)>-\Theta$ for every~$x\in B_{r_0}$.

Therefore, we conclude that
\begin{equation*}
\mathcal{L}^n(B_{r_0}\cap \left\lbrace w_k>-\Theta \right\rbrace)\geq \mathcal{L}^n(B_{r_0})=:c_0.
\end{equation*}
We can thereby apply Corollary~\ref{coro-09} and obtain that for any~$\theta\in (-1,1)$ there exist~$\widetilde{R}:=\widetilde{R}_{s,n,p,m,\theta_*,r_0,c_2}\in \left[r_0,+\infty\right)$ and~$\hat{c}:=\hat{c}_{s,n,p,m,\Lambda,c_1,\theta_*,r_0,c_0}\in (0,1)$  such that, for every~$r\in[\widetilde{R},+\infty)$ satisfying~$B_\frac{3r}{2}\subset \Omega_k$,
\begin{equation}\label{12345678098765oiuytrkjhgfdzxcvb2}
\mathcal{L}^n(B_r\cap \left\lbrace w_k> \theta \right\rbrace)\geq \hat{c}\,r^n.
\end{equation} 

Now, we choose~$\theta:=-\frac{1}{2}$ and~$\delta_0\in \left(0,\delta\right]$ such that~$B_{\frac{3}{2}\delta_0}(x_k)\subset\Omega$, and we obtain
from~\eqref{12345678098765oiuytrkjhgfdzxcvb}
and~\eqref{12345678098765oiuytrkjhgfdzxcvb2} that
\begin{equation*}
\begin{split}
0&=\lim_{k\to+\infty}  \int_{B_\delta(x_k)}\left|1+u_{\epsilon_k}(x)\right|\,dx\\
&\geq \lim_{k\to+\infty}  \int_{B_{\delta_0}(x_k)}\left|1-u_{\epsilon_k}(x)\right|\,dx\\
&=\lim_{k\to+\infty}  \epsilon_k^n \int_{B_{\frac{\delta_0}{\epsilon_k}}}\left|1+u_{\epsilon_k}(x\epsilon_k+x_k)\right|\,dx\\
&=\lim_{k\to+\infty} \epsilon_k^n\int_{B_{\frac{\delta_0}{\epsilon_k}}}\left|1+w_k(x)\right|\,dx\\
&\geq \lim_{k\to+\infty} \epsilon_k^n\int_{B_{\frac{\delta_0}{\epsilon_k}}\cap \left\lbrace w_k> \theta \right\rbrace }\left|1+w_k(x)\right|\,dx\\
&\geq \frac{\hat{c}}{2}\,\delta_0^n,
\end{split}
\end{equation*}
thus providing a contradiction.
\end{proof}

\begin{appendix}

\section{Some technical results towards the proof of Theorem~\ref{lem:1}}\label{appebi}

In this section we collect some technical results that are used 
throughout Section~\ref{desmos}
in order to prove Theorem~\ref{lem:1}. 

We start with two estimates for the fractional~$p$-Laplacian of a bounded and
locally Lipschitz function.

\begin{lem}\label{crocodile}
Let~$\rho\in(0,+\infty)$, $p\in (1,+\infty)$, $s\in \left(0,\frac{p-1}{p}\right)$ and~$\psi \in L^\infty\left(\R^n\right)$.

Let~$x\in\R^n$ and assume that
there exists~$\widehat{K}\in (0,+\infty)$
such that, for every~$y\in B_{\rho}(x)$, 
\begin{equation*}
\left|\psi(y)-\psi(x)\right|\leq \widehat{K}\left|y-x\right|.
\end{equation*}

Then, there exists~$c_{n,p}\in (0,+\infty)$ such that
\begin{equation}\label{cnderw}
\left|(-\Delta)_p^s\psi(x)\right|\leq \frac{c_{n,p}}{s(p-1-sp)}\left(\widehat{K}^{p-1} \rho^{-sp+p-1}+\left\|\psi\right\|_{L^\infty(\R^n)}^{p-1}\rho^{-sp}\right).
\end{equation}
\end{lem}

\begin{proof}
We compute 
\begin{equation*}
\begin{split}&
\left|(-\Delta)_p^s\psi(x)\right|\\&=\left|\int_{B_{\rho}(x)}\frac{\left(\psi(y)-\psi(x)\right)\left|\psi(y)-\psi(x)\right|^{p-2}}{\left|x-y\right|^{n+sp}}\,dy+\int_{\R^n \setminus B_{\rho}(x)}\frac{\left(\psi(y)-\psi(x)\right)\left|\psi(y)-\psi(x)\right|^{p-2}}{\left|x-y\right|^{n+sp}}\,dy\right| \\
&\leq \widehat{K}^{p-1}\int_{B_{\rho}(x)}\frac{dy}{\left|x-y\right|^{n+sp-p+1}}+2^{p-1}\left\|\psi\right\|_{L^\infty(\R^n)}^{p-1}\int_{\R^n \setminus B_{\rho}(x)}\frac{dy}{\left|x-y\right|^{n+sp}} \\
&=\frac{\widehat{K}^{p-1}\omega_{n-1}}{p-1-sp}\rho^{-sp+p-1}+2^{p-1}\left\|\psi\right\|_{L^\infty(\R^n)}^{p-1} \frac{\omega_{n-1}}{sp}\rho^{-sp}.
\end{split}
\end{equation*}
Now since~$sp<p-1$ we conclude that
\begin{equation*}
\begin{split}
\left|(-\Delta)_p^s\psi(x)\right|
&\leq \frac{(p-1)\widehat{K}^{p-1}\omega_{n-1}}{sp(p-1-sp)}\rho^{-sp+p-1}+2^{p-1}\left\|\psi\right\|_{L^\infty(\R^n)}^{p-1} \frac{(p-1)\omega_{n-1}}{sp(p-1-sp)}\rho^{-sp}\\
&\leq \frac{c_{n,p}}{s(p-1-sp)}\left(\widehat{K}^{p-1} \rho^{-sp+p-1}+\left\|\psi\right\|_{L^\infty(\R^n)}^{p-1}\rho^{-sp}\right),
\end{split}
\end{equation*}
which is the desired result.
\end{proof}

\begin{lem}\label{p2s01}
Let~$\rho\in(0,+\infty)$, $p\in [2,+\infty)$, $s\in (0,1)$ and~$\psi\in L^\infty(\R^n)\cap C(\R^n)$.

Let~$x\in\R^n$ and assume that
there exist~$\widehat{K},\widetilde{C}\in (0,+\infty)$ and~$\Sigma(x)\in\R^n$ such that, for every~$y\in B_{\rho}(x)$, 
\begin{equation}\label{condo-87}
\left|\psi(x)-\psi(y)\right|\leq \widehat{K}\left|x-y\right|\quad\mbox{and}\quad \psi(y)-\psi(x)-\Sigma(x)\cdot (y-x)\leq \widetilde{C}\left|x-y\right|^2.
\end{equation}

Then, there exists~$c_{n,p}\in(0,+\infty)$ such that
\begin{equation*}
-(-\Delta)_p^s \psi(x)\leq \frac{c_{n,p}}{s(1-s)}\left(\widetilde{C}\max\left\lbrace\left|\Sigma(x)\right|, \widehat{K} \right\rbrace^{p-2} \rho^{-sp+p}+\left\|\psi\right\|_{L^\infty(\R^n)}^{p-1}\rho^{-sp}\right).
\end{equation*}
\end{lem}

\begin{proof}
We define the affine approximation of~$\psi$ at~$x$ as 
\begin{equation*}
\ell(y):=\psi(x)+\Sigma(x)\cdot(y-x).
\end{equation*}
Then, denoting~$\phi_p(t):=t|t|^{p-2}$, we obtain that
\begin{equation*}
P.V.\int_{B_{\rho}(x)} \frac{\phi_p\left(\ell(y)-\ell(x)\right)}{\left|x-y\right|^{n+sp}}\,dy=0.
\end{equation*} 
Consequently,
\begin{equation*}
\begin{split}
&\int_{\R^n}\frac{\phi_p\left(\psi(y)-\psi(x)\right)}{\left|x-y\right|^{n+sp}}\,dy\\
=&\int_{B_{\rho}(x)}\frac{\phi_p\left(\psi(y)-\psi(x)\right)}{\left|x-y\right|^{n+sp}}\,dy+\int_{\R^n\setminus B_{\rho}(x)}\frac{\phi_p\left(\psi(y)-\psi(x)\right)}{\left|x-y\right|^{n+sp}}\,dy\\
=&\int_{B_{\rho}(x)}\frac{\phi_p\left(\psi(y)-\psi(x)\right)-\phi_p\left(\ell(y)-\ell(x)\right)}{\left|x-y\right|^{n+sp}}\,dy+\int_{\R^n\setminus B_{\rho}(x)}\frac{\phi_p\left(\psi(y)-\psi(x)\right)}{\left|x-y\right|^{n+sp}}\,dy\\
=&:I+I\!I.
\end{split}
\end{equation*}

To estimate~$I$, for every~$y\in B_{\rho}(x)$ we denote by
\begin{equation}\label{def-xi}
m:=\min\left\lbrace \psi(y)-\psi(x), \ell(y)-\ell(x) \right\rbrace\quad\mbox{and}\quad M:=\max\left\lbrace  \psi(y)-\psi(x), \ell(y)-\ell(x) \right\rbrace.
\end{equation}
Then, thanks to the Mean Value Theorem and~\eqref{condo-87}, we obtain that there exists some~$\xi \in [m,M]$ such that 
\begin{equation}\label{kiop76}
\begin{split}
\phi_{p}\left(\psi(y)-\psi(x)\right)-\phi_p\left(\ell(y)-\ell(x)\right)&\;=\phi_{p}'(\xi)\big(\psi(y)-\psi(x)-\ell(y)+\ell(x)\big)\\
&\;=\phi_p'(\xi) \left(\psi(y)-\ell(y)\right)\\
&\;=\phi_p'(\xi)\big( \psi(y)-\psi(x)-\Sigma(x)\cdot (y-x)\big)\\
&\;=(p-1)\left|\xi\right|^{p-2}\big(\psi(y)-\psi(x)-\Sigma(x)\cdot (y-x)\big)\\
&\;\leq \widetilde{C}(p-1)\left|\xi\right|^{p-2}\left|x-y\right|^{2}.
\end{split}
\end{equation}
Also, according to~\eqref{condo-87} and~\eqref{def-xi},
\begin{equation*}
\left|\xi\right|\leq \max\left\lbrace\left|\Sigma(x)\right|, \widehat{K} \right\rbrace \left|x-y\right|, 
\end{equation*}
and employing~\eqref{kiop76} we obtain that 
\begin{equation*}
\phi_{p}\left(\psi(y)-\psi(x)\right)-\phi_p\left(\ell(y)-\ell(x)\right)\leq \widetilde{C}(p-1)\max\left\lbrace\left|\Sigma(x)\right|, \widehat{K} \right\rbrace^{p-2}\left|x-y \right|^p.
\end{equation*}
On this account, we have that
\begin{equation*}
\begin{split}
I &\;\leq \widetilde{C}(p-1)\max\left\lbrace\left|\Sigma(x)\right|, \widehat{K} \right\rbrace^{p-2}\int_{B_{\rho}(x)} \frac{dy}{\left|x-y\right|^{n+sp-p}}\\
&\;=\widetilde{C}(p-1)\max\left\lbrace\left|\Sigma(x)\right|, \widehat{K} \right\rbrace^{p-2} \frac{\omega_{n-1}}{p-sp}\rho^{-sp+p}.
\end{split}
\end{equation*}

We now estimate~$I\!I$, as follows 
\begin{equation*}
\begin{split}
I\!I &\leq 2^{p-2}\left\|\psi\right\|_{L^\infty(\R^n)}^{p-1} \int_{\R^n} \frac{dy}{\left|x-y\right|^{n+sp}} =2^{p-2}\left\|\psi\right\|_{L^\infty(\R^n)}^{p-1}\frac{\omega_{n-1}}{sp}\rho^{-sp}.
\end{split}
\end{equation*}
We thus conclude that 
\begin{equation*}
\begin{split}&\!\!\!
-(-\Delta)_p^s \psi(x) \\&\leq \widetilde{C}(p-1)\max\left\lbrace\left|\Sigma(x)\right|, \widehat{K} \right\rbrace^{p-2} \frac{\omega_{n-1}}{p-sp}\rho^{-sp+p}+2^{p-1}\left\|\psi\right\|_{L^\infty(\R^n)}^{p-1}\frac{\omega_{n-1}}{sp}\rho^{-sp}\\
&=\widetilde{C}s(p-1)\max\left\lbrace\left|\Sigma(x)\right|, \widehat{K} \right\rbrace^{p-2} \frac{\omega_{n-1}}{sp(1-s)}\rho^{-sp+p}+2^{p-1}(1-s)\left\|\psi\right\|_{L^\infty(\R^n)}^{p-1}\frac{\omega_{n-1}}{sp(1-s)}\rho^{-sp}\\
&\leq\frac{c_{n,p}}{s(1-s)}\left(\widetilde{C}\max\left\lbrace\left|\Sigma(x)\right|, \widehat{K} \right\rbrace^{p-2} \rho^{-sp+p}+\left\|\psi\right\|_{L^\infty(\R^n)}^{p-1}\rho^{-sp}\right),
\end{split}
\end{equation*} as desired.
\end{proof}

We now focus on the specific construction of the barrier in Theorem~\ref{lem:1}.
We recall the setting in~\eqref{deffk}, \eqref{fuhweoiu34ii43ty43ugijkebjwk7684903}
and~\eqref{vdefinition} and we show the following:

\begin{prop}\label{corihenry5t3ed}
Let~$s\in(0,1)$ and~$p\in(1,+\infty)$. Let~$k$ and~$f$ be respectively as  in~\eqref{deffk} and~\eqref{fuhweoiu34ii43ty43ugijkebjwk7684903}.

Then, for every~$i\in \left\lbrace 0,\dots, k+1\right\rbrace$ there exists~$\widetilde{C}_{q,s,k,i}\in(0,+\infty)$ such that, if~$\rho_0\in \left(0,\frac{r}{2}\right)$ and~$t\in \left(\frac{r}{2},r-\rho_0\right)$,
\begin{equation}\label{stimederivate}
\widetilde{C}_{q,s,k,i}\left(\frac{r}{2}\right)^{-qs-k-1}\left(t-\frac{r}{2}\right)^{k+1-i} \leq f^{(i)}(t)\leq \widetilde{C}_{q,s,k,i}\,\rho_0^{-qs-k-1}\left(t-\frac{r}{2}\right)^{k+1-i}.
\end{equation}

Also, for every~$i\in \left\lbrace 0,\dots, k+1 \right\rbrace$ there exists~$C_{i,q,s}\in(0,+\infty)$ such that, if~$x\in \left(\frac{r}{2},r\right)$ and~$\rho\in \left(0,x-\frac{r}{2}\right)$,
\begin{equation}\label{ewlop545454}
\widetilde{C}_{q,s,k,i}(r-x+\rho)^{-qs-k-1}\rho^{k+1-i} \leq f^{(i)}(x)\leq C_{i,q,s} (r-x)^{-qs-i}. 
\end{equation}
\end{prop}

\begin{proof}
In order to prove~\eqref{stimederivate} and~\eqref{ewlop545454}, we notice that,
owing to~\eqref{fuhweoiu34ii43ty43ugijkebjwk7684903},
\begin{equation}\label{colevfgbrte5543}
\begin{dcases}
f^{(i)}\left(\frac{r}{2}\right)=0\quad &\mbox{for all}\quad i \in \left\lbrace 0, 1,\dots,k \right\rbrace\\
f^{(k+1)}(\eta)=(-1)^{(k+1)}g^{(k+1)}(r-\eta)\quad &\mbox{for all}\quad \eta\in \left(\frac{r}{2},r\right]. 
\end{dcases}
\end{equation}
Furthermore, for every~$\eta\in (0,+\infty)$ we have that
\begin{equation*}
g^{(k+1)}(\eta)=(-1)^{k+1}\prod_{i=0}^{k}(qs+i) \eta^{-qs-k-1}.
\end{equation*}
{F}rom this, it follows that, given~$0<\rho_0<\rho_1$, for all~$\eta\in \left(\rho_0, \rho_1\right)$,
\begin{equation}\label{Glopsdfert}
C_{q,s,k} \rho_1^{-qs-k-1}\leq (-1)^{(k+1)}g^{(k+1)}(\eta)\leq C_{q,s,k}\rho_0^{-qs-k-1}
\end{equation}
where
\begin{equation}\label{trihytg}
C_{q,s,k}:=\prod_{i=0}^{k} (qs+i).\end{equation}

Gathering these pieces of information, we deduce that, for every~$t\in \left(\frac{r}{2},r-\rho_0\right)$,
\begin{equation*}
C_{q,s,k} \left(\frac{r}{2}\right)^{-qs-k-1}  \leq   f^{(k+1)}(t)\leq C_{q,s,k}\rho_0^{-qs-k-1}.
\end{equation*}
{F}rom this we obtain that,
for every~$i\in \left\lbrace 0,\dots, k\right\rbrace$ and~$y\in \left(\frac{r}{2},r-\rho_0\right)$,
\begin{equation*}
\begin{split}
f^{(i)}(y)&=\int_{\frac{r}{2}}^y f^{(i+1)}(t_1)\,dt_1\\
&=\int_{\frac{r}{2}}^y\int_{\frac{r}{2}}^{t_1} f^{(i+2)}(t_2)\,dt_2\,dt_1\\
&=\int_{\frac{r}{2}}^y\int_{\frac{r}{2}}^{t_1} \dots \int_{\frac{r}{2}}^{t_{k-i}} f^{(k+1)}(t_{k+1-i})\,dt_{k+1-i}\dots dt_2\,dt_1\\
&\geq \widetilde{C}_{q,s,k,i}  \left(\frac{r}{2}\right)^{-qs-k-1}\left(y-\frac{r}{2}\right)^{k+1-i},
\end{split}
\end{equation*}
where 
\begin{equation}\label{le-bambooo}
\widetilde{C}_{q,s,k,i}:=\frac{C_{q,s,k}}{(k+1-i)!}.\end{equation}

Similarly, one can show that for every~$i\in \left\lbrace 0,\dots, k \right\rbrace$ and~$y\in \left(\frac{r}{2},r-\rho_0\right)$
\begin{equation*}
f^{(i)}(y)\leq \widetilde{C}_{q,s,k,i}\,\rho_0^{-qs-k-1}\left(y-\frac{r}{2}\right)^{k+1-i}.
\end{equation*}
This concludes the proof of~\eqref{stimederivate}.

Now we prove~\eqref{ewlop545454}.
To do so, we make use of~\eqref{colevfgbrte5543} and~\eqref{Glopsdfert} to see that, for every~$x\in \left(\frac{r}{2},r\right)$ and~$\rho\in \left(0,x-\frac{r}{2}\right)$, 
\begin{equation*}
C_{q,s,k}(r-x+\rho)^{-qs-k-1} \leq f^{(k+1)}(x)\leq C_{q,s,k}\left(r-x\right)^{-qs-k-1}. 
\end{equation*}
Moreover, we point out that~$f^{(i)}(x-\rho)\ge0$ for all~$i\in \left\lbrace 0,\dots, k \right\rbrace$.
Therefore, for every~$i\in \left\lbrace 0,\dots, k \right\rbrace$,
\begin{equation*}
\begin{split}
f^{(i)}(x)
&\ge\int_{x-\rho}^x f^{(i+1)}(t_1)\,dt_1\\ 
&\ge \int_{x-\rho}^x\int_{x-\rho}^{t_1} f^{(i+2)}(t_2)\,dt_2\,dt_1\\
&\ge
\int_{x-\rho}^x\int_{x-\rho}^{t_1}\cdots \int_{x-\rho}^{t_{k-i}} 
f^{(k+1)}(r-t_{k-i+1})\,dt_{k-i+1}\dots dt_2\,dt_1)\\
&\geq\widetilde{C}_{q,s,k,i} \left(r-x+\rho\right)^{-qs-k-1}\rho^{k+1-i}.
\end{split}
\end{equation*}
This gives the lower bound in~\eqref{ewlop545454}.

Additionally, recalling~\eqref{fuhweoiu34ii43ty43ugijkebjwk7684903}, we find that, for every~$i\in \left\lbrace 0, 1,\dots,k \right\rbrace$ and~$x\in \left(\frac{r}{2},r\right)$,
\begin{equation*}
\begin{split}
f^{(i)}(x)&=(-1)^{i}g^{(i)}(r-x)-(-1)^i g^{(i)}\left(\frac{r}{2}\right)+\sum_{j=i+1}^{k}\frac{(-1)^{j+1}}{(j-2)!}g^{(j)}\left(\frac{r}{2}\right)\left(x-\frac{r}{2}\right)^{j-2}\\ 
&\leq (-1)^{i}g^{(i)}(r-x)\\
&=\prod_{j=0}^{i-1}(qs+j)(r-x)^{-qs-i}.
\end{split}
\end{equation*}
Thsi gives the upper bound in~\eqref{ewlop545454} with
\begin{equation}\label{clero-123}
C_{i,q,s}:=\prod_{j=0}^{i-1}(qs+j),
\end{equation} as desired.
\end{proof}

\begin{rem}
{\rm We notice that if~$p\in (1,2)$, $s\in \left[\frac{p-1}{2p},1\right)$ and~$q:=\frac{p}{m-1}$ it follows from~\eqref{deffk} that
\begin{equation*}
1\leq k\leq \left[\frac{p}{p-1}\right]+1.
\end{equation*}
Hence, if we recall~\eqref{trihytg} and we define 
\begin{equation}
\label{CPM} \qquad C_{p,m}:=
\prod_{i=0}^{\left[\frac{p}{p-1}\right]+1}\left(\frac{p}{m-1}+i\right),
\end{equation}
we see that 
\begin{equation}\label{hygvredf}
\frac{p-1}{2(m-1)}\left(\frac{p-1}{2(m-1)}+1\right)\leq C_{q,s,k}\leq C_{p,m}.
\end{equation}
Thus, recalling also~\eqref{le-bambooo},
for every~$i\in \left\lbrace 0,\dots, k+1 \right\rbrace$, we evince that 
\begin{equation*}
c_{p,m}:=\frac{\displaystyle\frac{p-1}{2(m-1)}\left(\frac{p-1}{2(m-1)}+1\right)}{\displaystyle \left(\left[\frac{p}{p-1}\right]+2\right)!} \leq \widetilde{C}_{q,s,k,i}\leq C_{p,m}.
\end{equation*}
Hence, using also~\eqref{stimederivate}, we deduce that for every~$i\in \left\lbrace 0,\dots, k+1\right\rbrace$, $\rho_0\in \left(0,\frac{r}{2}\right)$ and~$t\in \left(\frac{r}{2},r-\rho_0\right)$ 
\begin{equation}\label{yllwdcfere}
c_{p,m}\left(\frac{r}{2}\right)^{-qs-k-1}\left(t-\frac{r}{2}\right)^{k+1-i} \leq f^{(i)}(t)\leq C_{p,m}\,\rho_0^{-qs-k-1}\left(t-\frac{r}{2}\right)^{k+1-i}. 
\end{equation}
Now we notice that, by~\eqref{clero-123} and~\eqref{hygvredf},
\begin{equation}\label{hre-oper}
C_{i,q,s}\leq C_{q,s,k}\leq C_{p,m}.
\end{equation}We infer from this and~\eqref{ewlop545454} that for every~$x\in \left(\frac{2}{3}r,r\right)$, $\rho:=\frac{r-x}{2}$ and~$i\in \left\lbrace 0,\dots, k+1 \right\rbrace$
\begin{equation*}
c_{p,m}3^{-qs-k-1} \rho^{-qs-i}\leq f^{(i)}(x)\leq C_{p,m}2^{-qs-i}\rho^{-qs-i}. 
\end{equation*}
Thus, if we define the constant 
\begin{equation}\label{tildecpm}
\tilde{c}_{p,m}:=c_{p,m}\left(\frac{1}{3}\right)^{\frac{p}{m-1}+\left[\frac{p}{p-1}\right]+2},
\end{equation}
we see that, for every~$x\in \left(\frac{2}{3}r,r\right)$ and~$i\in \left\lbrace 0,\dots, k+1 \right\rbrace$,
\begin{equation}\label{p124353}
\tilde{c}_{p,m}\rho^{-qs-i}\leq f^{(i)}(x)\leq C_{p,m}\rho^{-qs-i},
\end{equation}where~$\rho:=\frac{r-x}{2}$.
}\end{rem}

Thanks to these observations, we now provide some auxiliary results
that allow us to estimate the fractional~$p$-Laplacian of the function~$v$
in~\eqref{vdefinition}, as stated in Proposition~\ref{lemma:BigClaim}.

\begin{lem}\label{ebgfncy654444}
Let~$p\in (1,2)$, $s\in \left[\frac{p-1}{2p},1\right)$ and~$r\in (1,+\infty)$.  Let~$v:\R^n\to [0,1]$ be as in~\eqref{vdefinition}. Let~$x\in B_r$ and
\begin{equation}\label{fy4ui32i34iuhdasjcbsabv7685498964053}
\rho:=\frac{r-\left|x\right|}{2}.
\end{equation} 

Then, there exists~$C_{n,p,m}\in(0,+\infty)$ such that  
\begin{equation}\label{frl54t4t409}
\int_{B_{\rho}(x)}\frac{\left(v(y)-v(x)\right)\left|v(y)-v(x)\right|^{p-2}}{\left|x-y\right|^{n+sp}}\,dy \leq \frac{C_{n,p,m}}{1-s} \rho^{-qs(p-1)-sp}.   
\end{equation} 
\end{lem}

\begin{proof}
In what follows, for simplicity of notation, if~$A$, $B\in(0,+\infty)$ we adopt the notation
\begin{equation*}
A \lesssim B \quad\mbox{if}\quad A \leq C_{n,p,m} B 
\end{equation*} 
for some~$C_{n,p,m}\in(0,+\infty)$ depending at most on~$n$, $p$ and~$m$. 

Moreover, for every~$p\in (1,2)$, $s\in \left[\frac{p-1}{2p},1\right)$ 
and~$k$ as in~\eqref{deffk}, we have that
\begin{equation}\label{oler-0988776}
(k+1)(p-1)-sp \geq (1-s)(p-1).  
\end{equation}
Indeed,
\begin{eqnarray*}&&
(k+1)(p-1)-sp -(1-s)(p-1) = k(p-1)-s\\&&\qquad=
\left(\left[\frac{sp}{p-1}\right]+1\right)(p-1) -s
\geq \frac{sp}{p-1}(p-1) -s=s(p-1)\geq 0,
\end{eqnarray*} as desired.

Now, we complete the proof of~\eqref{frl54t4t409}, which is based on a long and tedious computation, that we provide here below for the facility of the reader.
We will deal separately with the cases~$x\in \overline{B}_{\frac{r}{2}}$, $x\in B_{\hat{r}}\setminus \overline{B}_{\frac{r}{2}}$ and~$x\in B_{r}\setminus B_{\hat{r}}$, where~$\hat{r}$ has been defined in~\eqref{g874rfegt56}. 
\medskip

\noindent{\em{ {\bf{(i) Case $x\in \overline{B}_{\frac{r}{2}}$.}}}}
In this case, for every~$y\in B_{\rho}(x)$ it holds that  
\begin{equation*}
\left|y\right|\leq \rho+\left|x\right|=\frac{r-\left|x\right|}{2}+\left|x\right|= \frac{r+\left|x\right|}{2}\leq \frac{3}{4}r.
\end{equation*}
Hence, recalling~\eqref{vdefinition}, if~$y\in B_{\rho}(x)$,
\begin{eqnarray*}
&&\left(v(y)-v(x)\right)\left|v(y)-v(x)\right|^{p-2}
=\left(h(|y|)-h(|x|)\right)\left|h(|y|)-h(|x|)\right|^{p-2}\\
&&\qquad\qquad =h(|y|)|h(|y|)|^{p-2}\le |f(|y|)|^{p-1}\chi_{\R^n\setminus B_\frac{r}{2}}(y).
\end{eqnarray*}
As a consequence,
\begin{equation*}
\int_{B_{\rho}(x)}\frac{\left(v(y)-v(x)\right)\left|v(y)-v(x)\right|^{p-2}}{\left|x-y\right|^{n+sp}}\,dy \leq  \int_{B_{\rho}(x)\setminus B_\frac{r}{2}}  \frac{|f(|y|)|^{p-1}}{\left|x-y\right|^{n+sp}}\,dy.
\end{equation*}

We now use~\eqref{yllwdcfere} with~$\rho_0:=r/4$ and~$i:=0$ and we 
obtain that
\begin{equation*}\begin{split}&
\int_{B_{\rho}(x)}\frac{\left(v(y)-v(x)\right)\left|v(y)-v(x)\right|^{p-2}}{\left|x-y\right|^{n+sp}}\,dy\\ \leq\;& 
C_{p,m}^{p-1}\left(\frac{r}{4}\right)^{-(qs+k+1)(p-1)} \int_{B_{\rho}(x)\setminus B_\frac{r}{2}}  \frac{\left(\left|y\right|-\frac{r}{2}\right)^{(p-1)(k+1)}}{\left|x-y\right|^{n+sp}}\,dy.
\end{split}
\end{equation*}
Thus, changing variable gives that
\begin{equation*}\begin{split}&
\int_{B_{\rho}(x)}\frac{\left(v(y)-v(x)\right)\left|v(y)-v(x)\right|^{p-2}}{\left|x-y\right|^{n+sp}}\,dy\\ 
\le \;&
C_{p,m}^{p-1}\left(\frac{r}{4}\right)^{-(qs+k+1)(p-1)} \int_{
B_{\rho}\setminus B_\frac{r}{2}(-x)}  \frac{(\left|y+x\right|-\frac{r}{2})^{(k+1)(p-1)}}{\left|y\right|^{n+sp}}\,dy\\
\leq\;&C_{p,m}^{p-1}\left(\frac{r}{4}\right)^{-(qs+k+1)(p-1)}\int_{B_{\rho}\setminus B_{\frac{r}{2}}(-x)} \frac{(\left|y\right|+\left|x\right|-\frac{r}{2})^{(k+1)(p-1)}}{\left|y\right|^{n+sp}}\,dy
.\end{split}
\end{equation*}
Since~$|x|\le r/2$, from this we find that
\begin{equation}\label{758sflkashfkfhklq76543mknoknjbhiuj}\begin{split}&
\int_{B_{\rho}(x)}\frac{\left(v(y)-v(x)\right)\left|v(y)-v(x)\right|^{p-2}}{\left|x-y\right|^{n+sp}}\,dy\\ 
\le \;&C_{p,m}^{p-1}\left(\frac{r}{4}\right)^{-(qs+k+1)(p-1)} \int_{B_{\rho}} \frac{\left|y\right|^{(k+1)(p-1)}}{\left|y\right|^{n+sp}}\,dy\\
=\;&
\frac{C_{p,m}^{p-1}\,\omega_{n-1}}{(k+1)(p-1)-sp} \left(\frac{r}{4}\right)^{-(qs+k+1)(p-1)}\rho^{(k+1)(p-1)-sp}
.\end{split}
\end{equation}
Also, in light of~\eqref{fy4ui32i34iuhdasjcbsabv7685498964053}
we have that~$\rho\le r/2$, and therefore
\begin{eqnarray*}
r^{-(qs+k+1)(p-1)}\rho^{(k+1)(p-1)-sp}\le
\left(2{\rho}\right)^{-(qs+k+1)(p-1)}\rho^{(k+1)(p-1)-sp}
=2^{-(qs+k+1)(p-1)}\rho^{-qs(p-1)-sp }.
\end{eqnarray*}
Using this into~\eqref{758sflkashfkfhklq76543mknoknjbhiuj}
and recalling~\eqref{oler-0988776},
we see that
\begin{equation*}\begin{split}
\int_{B_{\rho}(x)}\frac{\left(v(y)-v(x)\right)\left|v(y)-v(x)\right|^{p-2}}{\left|x-y\right|^{n+sp}}\,dy
\leq\;& 
\frac{C_{p,m}^{p-1}\,\omega_{n-1}2^{(qs+k+1)(p-1)}
}{(k+1)(p-1)-sp}\rho^{-qs(p-1)-sp }\\ \leq\;&
\frac{C_{p,m}^{p-1}\,\omega_{n-1}2^{(qs+k+1)(p-1)}
}{(1-s)(p-1)}\rho^{-qs(p-1)-sp }
,\end{split}
\end{equation*}
which gives the desired estimate in~\eqref{frl54t4t409}.
\medskip

\noindent{\em{ {\bf{(ii) Case $x\in B_{\hat{r}}\setminus \overline{B}_{\frac{r}{2}}$.}}}}
If we denote by
\begin{eqnarray*}
B_\rho^{+}&:=&\big\{ y\in B_{\rho}\mbox{  s.t.  }\, y\cdot x\geq 0  \big\} \\
{\mbox{and }}\qquad B_\rho^{-}&:=&\big\{ y\in B_{\rho}\mbox{  s.t.  }\, y\cdot x< 0  \big\} ,
\end{eqnarray*} 
we have that
\begin{equation}\label{mnbvc09876hlgiffahj}
{\mbox{$|y+x|\geq |x|$
for every~$y\in B_{\rho}^{+}$.}}\end{equation}

Now, we write
\begin{equation*}
\begin{split}
&\int_{B_{\rho}}\frac{(v(y+x)-v(x)) |v(y+x)-v(x)|^{p-2}}{|y|^{n+sp}}\,dy\\
=\;& \frac{1}{2}\int_{B_\rho} \frac{(v(y+x)-v(x)) |v(y+x)-v(x)|^{p-2}+
(v(x-y)-v(x)) |v(x-y)-v(x)|^{p-2}}{|y|^{n+sp}}\,dy\\
=\;& \frac{1}{2}\int_{B_\rho^+} \frac{(v(y+x)-v(x)) |v(y+x)-v(x)|^{p-2}+
(v(x-y)-v(x)) |v(x-y)-v(x)|^{p-2}}{|y|^{n+sp}}\,dy\\
&\quad+
\frac{1}{2}\int_{B_\rho^-} \frac{(v(y+x)-v(x)) |v(y+x)-v(x)|^{p-2}+
(v(x-y)-v(x)) |v(x-y)-v(x)|^{p-2}}{|y|^{n+sp}}\,dy.
\end{split}\end{equation*}
Since~$v$ is a radial function, changing variable in the last integral, we find that
\begin{eqnarray*}
&&\int_{B_\rho^-} \frac{(v(y+x)-v(x)) |v(y+x)-v(x)|^{p-2}+
(v(x-y)-v(x)) |v(x-y)-v(x)|^{p-2}}{|y|^{n+sp}}\,dy\\
&=&
\int_{B_\rho^+} \frac{(v(y+x)-v(x)) |v(y+x)-v(x)|^{p-2}+
(v(x-y)-v(x)) |v(x-y)-v(x)|^{p-2}}{|y|^{n+sp}}\,dy, \end{eqnarray*}
and therefore
\begin{equation*}
\begin{split}
&\int_{B_{\rho}}\frac{(v(y+x)-v(x)) |v(y+x)-v(x)|^{p-2}}{|y|^{n+sp}}\,dy\\
=\;&
\int_{B_\rho^+} \frac{(v(y+x)-v(x)) |v(y+x)-v(x)|^{p-2}+
(v(x-y)-v(x)) |v(x-y)-v(x)|^{p-2}}{|y|^{n+sp}}\,dy\\
=\;&\int_{B_{\rho}^{+}} \frac{\left(h(|y+x|)-h(|x|)\right)\left|h(|y+x|)-h(|x|)\right|^{p-2}+\left(h(|x-y|)-h(|x|)\right)\left|h(|x-y|)-h(|x|)\right|^{p-2}}{|y|^{n+sp}}\,dy.
\end{split}\end{equation*}
As a result, thanks to~\eqref{mnbvc09876hlgiffahj}
and the fact that~$h$ is a non decreasing function, 
\begin{equation}\label{decompos43fr60912}
\begin{split}
&\int_{B_{\rho}}\frac{\left(v(y+x)-v(x)\right)\left|v(y+x)-v(x)\right|^{p-2}}{\left|y\right|^{n+sp}}\,dy\\
=& \int_{B_{\rho}^{+}\cap B_{\left|x\right|}(x)\cap B_{\frac{r}{2}}^c(x)} \frac{\left|h(\left|y+x\right|)-h(\left|x\right|)\right|^{p-1}-\left|h(\left|x\right|)-h(\left|x-y\right|)\right|^{p-1}}{\left|y\right|^{n+sp}}\,dy\\
&+\int_{B_{\rho}^{+}\cap B_{\left|x\right|}(x)\cap B_{\frac{r}{2}}(x)} \frac{\left|h(\left|y+x\right|)-h(\left|x\right|)\right|^{p-1}-\left|h(\left|x\right|)-h(\left|x-y\right|)\right|^{p-1}}{\left|y\right|^{n+sp}}\,dy\\
&+ \int_{B_{\rho}^{+}\cap B_{\left|x\right|}^c(x)} \frac{\left|h(\left|y+x\right|)-h(\left|x\right|)\right|^{p-1}+\left|h(\left|x-y\right|)-h(\left|x\right|)\right|^{p-1}}{\left|y\right|^{n+sp}}\,dy.
\end{split}
\end{equation}

Now, we claim that
\begin{equation}
\int_{B_{\rho}^{+}\cap B_{\left|x\right|}(x)\cap B_{\frac{r}{2}}^c(x)} \frac{\left|h(\left|y+x\right|)-h(\left|x\right|)\right|^{p-1}-\left|h(\left|x\right|)-h(\left|x-y\right|)\right|^{p-1}}{\left|y\right|^{n+sp}}\,dy\lesssim \frac{\rho^{-qs(p-1)-sp}}{1-s}\label{CL1},\end{equation}
\begin{equation}
\int_{B_{\rho}^{+}\cap B_{\left|x\right|}(x)\cap B_{\frac{r}{2}}(x)} \frac{\left|h(\left|y+x\right|)-h(\left|x\right|)\right|^{p-1}-\left|h(\left|x\right|)-h(\left|x-y\right|)\right|^{p-1}}{\left|y\right|^{n+sp}}\,dy \lesssim \frac{\rho^{-qs(p-1)-sp}}{1-s}\label{CL2}\end{equation}
and
\begin{equation}
\int_{B_{\rho}^{+}\cap B_{\left|x\right|}^c(x)} \frac{\left|h(\left|y+x\right|)-h(\left|x\right|)\right|^{p-1}+\left|h(\left|x-y\right|)-h(\left|x\right|)\right|^{p-1}}{\left|y\right|^{n+sp}}\,dy \lesssim \frac{\rho^{-qs(p-1)-sp}}{1-s}\label{CL3}.
\end{equation}
We point out that
if the claims in~\eqref{CL1}, \eqref{CL2} and~\eqref{CL3} hold true, then from~\eqref{decompos43fr60912} the estimate in~\eqref{frl54t4t409} for the case~$x\in B_{\hat{r}}\setminus \overline{B}_{\frac{r}{2}}$ readily follows. 

Hence, from now on, we focus on the proofs of the claims in~\eqref{CL1}, \eqref{CL2} and~\eqref{CL3}.
\medskip

\underline{Proof of~\eqref{CL1}}:
For this, we recall~\eqref{vdefinition} and
we use Proposition~\ref{corihenry5t3ed} to find that, for every~$y\in B_{\rho}^{+}\cap B_{\left|x\right|}(x)$,
\begin{equation}\label{22-22}
\begin{split}&
h(\left|x+y\right|)-h(\left|x\right|)\leq f(\left|x+y\right|)-f(\left|x\right|)
=\int_{\left|x\right|}^{\left|x+y\right|}f'(t)\,dt\\
&\qquad=\int_{\left|x\right|}^{\left|x+y\right|}\left(
f'(\left|x\right|)+\int_{\left|x\right|}^{t}f''(\xi)\,d\xi\right)\,dt\\
&\qquad\leq f'(\left|x\right|)\left(\left|y+x\right|-\left|x\right|\right)+ f''(\left|x+y\right|)\int_{\left|x\right|}^{\left|x+y\right|}\left(\int_{\left|x\right|}^{t}\,d\xi\right)\,dt\\
&\qquad = f'(\left|x\right|)\left(\left|y+x\right|-\left|x\right|\right)+ f''(\left|x+y\right|)\frac{\left(\left|x+y\right|-\left|x\right|\right)^2}{2}.
 \end{split}
\end{equation}
Also, for every~$y\in B_{\rho}^{+}\cap B_{\left|x\right|}(x)\cap B_{\frac{r}{2}}^c(x)$, we have that~$|x-y|\in \left(\frac{r}2,\hat{r}\right)$, and so
\begin{equation}\label{33-33}
\begin{split}
&h(\left|x\right|)-h(\left|x-y\right|)=f(\left|x\right|)-f(\left|x-y\right|)
=\int_{\left|x-y\right|}^{\left|x\right|}f'(t)\,dt\\
&\qquad=\int_{\left|x-y\right|}^{\left|x\right|}\left(
f'(\left|x\right|)-\int_{t}^{\left|x\right|} f''(\xi)\,d\xi\right)\,dt\\
&\qquad\geq f'(\left|x\right|)\left(\left|x\right|-\left|x-y\right|\right)-f''(\left|x\right|)\int_{\left|x-y\right|}^{\left|x\right|}\left(
\int_{t}^{\left|x\right|}\,d\xi\right)\,dt\\
&\qquad=f'(\left|x\right|)\left(\left|x\right|-\left|x-y\right|\right)-f''(\left|x\right|)\frac{\left(\left|x\right|-\left|x-y\right|\right)^2}{2}.
\end{split}
\end{equation}

Now, we define
\begin{equation}\label{8t43toifaf098765jmdsbkvfdkvs}
u(x):=2\frac{f'(\left|x\right|)}{f''(\left|x\right|)}
\end{equation}
and we notice that, for every~$y\in B_{u(x)}$, 
\begin{equation*}
f'(\left|x\right|)-f''(\left|x\right|)\frac{\left(\left|x\right|-\left|x-y\right|\right)}{2}\geq f'(\left|x\right|) -f''(\left|x\right|)\frac{\left|y\right|}{2}\geq 0.
\end{equation*}
{F}rom this, \eqref{22-22} and~\eqref{33-33}
we deduce that, for every~$y\in B_{\rho}^{+}\cap B_{\left|x\right|}(x)\cap B_{\frac{r}{2}}^c(x)\cap B_{u(x)}$,
\begin{equation}\label{koplierty}
\begin{split}
&\left|h(\left|y+x\right|)-h(\left|x\right|)\right|^{p-1}-\left|h(\left|x\right|)-h(\left|x-y\right|)\right|^{p-1}\\
\leq & \left(\left|x+y\right|-\left|x\right|\right)^{p-1}\left|f'(\left|x\right|)+ f''(\left|x+y\right|)\frac{\left|x+y\right|-\left|x\right|}{2}\right|^{p-1}\\
&\quad- \left(\left|x\right|-\left|x-y\right|\right)^{p-1}\left|f'(\left|x\right|)-f''(\left|x\right|)\frac{\left(\left|x\right|-\left|x-y\right|\right)}{2}\right|^{p-1}\\
=&\left(\left|x+y\right|-\left|x\right|\right)^{p-1}\\
&\quad\times\left(\left|f'(\left|x\right|)+ f''(\left|x+y\right|)\frac{\left|x+y\right|-\left|x\right|}{2}\right|^{p-1}- \left|f'(\left|x\right|)-f''(\left|x\right|)\frac{\left(\left|x\right|-\left|x-y\right|\right)}{2}\right|^{p-1}\right)\\
&+\big((|x+y|-|x|)^{p-1}-(|x|-|x-y|)^{p-1}\big)
\left|f'(\left|x\right|)-f''(\left|x\right|)\frac{\left(\left|x\right|-\left|x-y\right|\right)}{2}\right|^{p-1}\\
\leq & |y|^{p-1}\left(\left|f'(\left|x\right|)+ f''(\left|x+y\right|)\frac{\left|y\right|}{2}\right|^{p-1}- \left|f'(\left|x\right|)-f''(\left|x\right|)\frac{\left|y\right|}{2}\right|^{p-1}\right)\\
&\quad+\big((|x+y|-|x|)^{p-1}- (|x|-|x-y|)^{p-1}\big)
\left|f'(\left|x\right|)\right|^{p-1}.
\end{split}
\end{equation}
Applying Taylor's Theorem with the Lagrange remainder we obtain that
\begin{equation}\label{kwi7ebt4.kju}
\begin{split}
&\left|f'(\left|x\right|)+f''(\left|x+y\right|)\frac{\left|y\right|}{2}\right|^{p-1}\\=\;&\left|f'(\left|x\right|)\right|^{p-1}+(p-1)\left|f'(\left|x\right|)\right|^{p-2}f''(\left|x+y\right|)\frac{\left|y\right|}{2}\\&\qquad
+\frac{(p-1)(p-2)}{8}\left|f'(\left|x\right|)+b\right|^{p-3}f''(\left|x+y\right|)^2 \left|y\right|^2\\
\leq\; & \left|f'(\left|x\right|)\right|^{p-1}+(p-1)\left|f'(\left|x\right|)\right|^{p-2}f''(\left|x+y\right|)\frac{\left|y\right|}{2}\\&\qquad
+\frac{(p-1)(2-p)}{8}\left|f'(\left|x\right|)+b\right|^{p-3}(f''(\left|x+y\right|))^2 \left|y\right|^2\\
\leq \;& \left|f'(\left|x\right|)\right|^{p-1}+(p-1)\left|f'(\left|x\right|)\right|^{p-2}f''(\left|x+y\right|)\frac{\left|y\right|}{2}\\&\qquad
+\frac{(p-1)(2-p)}{8}\left|f'(\left|x\right|)\right|^{p-3}(f''(\left|x+y\right|))^2 \left|y\right|^2,
\end{split}
\end{equation}
for some~$b\in \left(0,f''(\left|x+y\right|)\frac{\left|y\right|}{2}\right)$.

Also, for every~$\xi\in \left(0,f''(\left|x\right|)\frac{\left|y\right|}{2}\right)$,
\begin{equation*}
\left|f'(\left|x\right|)-\xi\right|\geq \left|f'(\left|x\right|)\right|-\xi\geq f'(\left|x\right|)-f''(\left|x\right|)\frac{\left|y\right|}{2}\geq 0. 
\end{equation*}
Hence, thanks to Taylor's Theorem we find that, for some~$a\in \left(0,f''(\left|x\right|)\frac{\left|y\right|}{2}\right)$,
\begin{equation*}
\begin{split}
&\left|f'(\left|x\right|)-f''(\left|x\right|)\frac{\left|y\right|}{2}\right|^{p-1}\\
=\;& \left|f'(\left|x\right|)\right|^{p-1}-(p-1)\left|f'(\left|x\right|)\right|^{p-2}f''(\left|x\right|)\frac{\left|y\right|}{2}
+\frac{(p-1)(p-2)}{8}\left|f'(\left|x\right|)-a\right|^{p-3}(f''(|x|))^2 |y|^2. 
\end{split}
\end{equation*} 
{F}rom this, \eqref{koplierty} and~\eqref{kwi7ebt4.kju} we conclude that, for every~$y\in B_{\rho}^{+}\cap B_{\left|x\right|}(x)\cap B_{\frac{r}{2}}^c(x)\cap B_{u(x)}$,
\begin{equation*}
\begin{split}
&\left|h(\left|y+x\right|)-h(\left|x\right|)\right|^{p-1}-\left|h(\left|x\right|)-h(\left|x-y\right|)\right|^{p-1}\\
\leq\; & \frac{p-1}{2} \left|y\right|^{p}\left|f'(\left|x\right|)\right|^{p-2}\big(f''(|x+y|)+f''(|x|)\big)\\
&+ \frac{(p-1)(2-p)}{8}\left|y\right|^{p+1} \left(f'(\left|x\right|)^{p-3}(
f''(|x+y|))^2+\left|f'(\left|x\right|)-f''(\left|x\right|)\frac{\left|y\right|}{2}\right|^{p-3}(f''(|x|))^2\right)\\
&+\big((|x+y|-|x|)^{p-1}- (|x|-|x-y|)^{p-1}\big)\left|f'(\left|x\right|)\right|^{p-1}.
\end{split}
\end{equation*}
Therefore, we deduce that, for every~$w\in \left(0,u(x)\right)\cap (0,\left|x\right|-\frac{r}{2})\cap (0,\rho)$,
\begin{equation}\label{p3cfre}
\begin{split}
&\int_{B_{\rho}^{+}\cap B_{\left|x\right|}(x)\cap B_{\frac{r}{2}}^c(x)} \frac{\left|h(\left|y+x\right|)-h(\left|x\right|)\right|^{p-1}-\left|h(\left|x\right|)-h(\left|x-y\right|)\right|^{p-1}}{\left|y\right|^{n+sp}}\,dy\\
\leq\; & \frac{p-1}{2}\int_{B_{w}^{+}\cap B_{\left|x\right|}(x)\cap B_{\frac{r}{2}}^c(x)}\frac{\left|f'(\left|x\right|)\right|^{p-2}\left(f''(\left|x+y\right|)+f''(\left|x\right|)\right)}{\left|y\right|^{n+sp-p}}\,dy\\
&+\frac{(p-1)(2-p)}{8}\int_{B_{w}^{+}\cap B_{\left|x\right|}(x)\cap B_{\frac{r}{2}}^c(x)}\frac{f'(\left|x\right|)^{p-3}f''(\left|x+y\right|)^2+\left|f'(\left|x\right|)-f''(\left|x\right|)\frac{\left|y\right|}{2}\right|^{p-3}f''(\left|x\right|)^2}{\left|y\right|^{n+sp-p-1}}\,dy\\
&+\int_{B_{w}^{+}\cap B_{\left|x\right|}(x)\cap B_{\frac{r}{2}}^c(x)} \frac{\left(\left(\left|x+y\right|-\left|x\right|\right)^{p-1}- \left(\left|x\right|-\left|x-y\right|\right)^{p-1}\right)\left|f'(\left|x\right|)\right|^{p-1}}{\left|y\right|^{n+sp}}\,dy\\
&+\int_{\left(B_{\rho}^{+}\setminus B_{w}^{+}\right)\cap B_{\left|x\right|}(x)\cap B_{\frac{r}{2}}^c(x)} \frac{\left|h(\left|y+x\right|)-h(\left|x\right|)\right|^{p-1}-\left|h(\left|x\right|)-h(\left|x-y\right|)\right|^{p-1}}{\left|y\right|^{n+sp}}\,dy.
\end{split}
\end{equation}

Now we adopt the notation  
\begin{equation}\label{fyuedwiqryeuwiyuwdtilde}
\tilde{x}:=\left|x\right|-\frac{r}{2}.
\end{equation}
Also, we assume that~$\hat{r}\in\left(\frac{2}{3}r,+\infty\right)$, the other case being similar. We notice that
\begin{eqnarray}\label{rho-tildex.kju}
&&\tilde{x}<\rho \quad{\mbox{ if~$x\in B_{\frac{2}{3}r}\setminus \overline{B}_{\frac{r}{2}}$}}\\
{\mbox{and }} &&\tilde{x}\geq \rho\quad{\mbox{ if~$x\in B_{\hat{r}}\setminus B_{\frac{2}{3}r}$.}}
\label{jnurefv540as}
\end{eqnarray}
Moreover, we define 
\begin{equation*}
\alpha_{p,m}:=\frac{\tilde{c}_{p,m}}{C_{p,m}}
\end{equation*}
and the function
\begin{equation}\label{gallgbffer430}
w(x):=\begin{dcases}
\alpha_{p,m}\,\tilde{x}\quad &\mbox{for}\quad x \in B_{\frac{2}{3}r}\setminus \overline{B}_{\frac{r}{2}},\\
\alpha_{p,m}\,\rho  \quad &\mbox{for}\quad x \in B_{\hat{r}}\setminus B_{\frac{2}{3}r},
\end{dcases}
\end{equation}
where~$\tilde{c}_{p,m}$ and~$C_{p,m}$ are given respectively in~\eqref{tildecpm} and~\eqref{CPM}. 

We notice that~$\alpha_{p,m}\le1$, thanks to~\eqref{p124353}.
Moreover, recalling~\eqref{8t43toifaf098765jmdsbkvfdkvs} and making use
of~\eqref{p124353},
\begin{eqnarray*}
&&u(x)=2\frac{f'(|x|)}{f''(|x|)}
\ge 2\frac{ \tilde{c}_{p,m}\left(\frac{r-|x|}{2}\right)^{-qs-1}}{
C_{p,m}\left(\frac{r-|x|}{2}\right)^{-qs-2}
}=\alpha_{p,m}(r-|x|)
=\alpha_{p,m}\left(\frac{r}2-\tilde x
\right)\ge w(x).
\end{eqnarray*}
Therefore,
using also~\eqref{rho-tildex.kju} and~\eqref{jnurefv540as}, we have that, if~$x\in B_{\hat{r}}\setminus \overline{B}_{\frac{r}{2}}$,
\begin{equation}\label{papapsdfre}
w(x)\leq \min\left\lbrace u(x),\tilde{x},\rho \right\rbrace.
\end{equation} 
Accordingly, we can exploit~\eqref{p3cfre} with~$w:=w(x)$.

In this way, we deduce that, in order to show~\eqref{CL1}, it is enough to prove that, if~$x\in B_{\hat{r}}\setminus B_{\frac{r}{2}}$,
\begin{equation}
\int_{B_{w(x)}^{+}\cap B_{\left|x\right|}(x)\cap B_{\frac{r}{2}}^c(x)}\frac{\left|f'(\left|x\right|)\right|^{p-2}\left(f''(\left|x+y\right|)+f''(\left|x\right|)\right)}{\left|y\right|^{n+sp-p}}\,dy\lesssim \frac{\rho^{-qs(p-1)-sp}}{1-s},\label{iure}\end{equation}
\begin{equation}\begin{split}&
\int_{B_{w(x)}^{+}\cap B_{\left|x\right|}(x)\cap B_{\frac{r}{2}}^c(x)}\frac{f'(\left|x\right|)^{p-3}f''(\left|x+y\right|)^2+\left|f'(\left|x\right|)-f''(\left|x\right|)\frac{\left|y\right|}{2}\right|^{p-3}f''(\left|x\right|)^2}{\left|y\right|^{n+sp-p-1}}\,dy \\&\qquad\qquad\qquad\lesssim \frac{\rho^{-qs(p-1)-sp}}{1-s},\label{spoart}\end{split}\end{equation}
\begin{equation}
\int_{B_{w(x)}^{+}\cap B_{\left|x\right|}(x)\cap B_{\frac{r}{2}}^c(x)} \frac{\left(\left(\left|x+y\right|-\left|x\right|\right)^{p-1}- \left(\left|x\right|-\left|x-y\right|\right)^{p-1}\right)\left|f'(\left|x\right|)\right|^{p-1}}{\left|y\right|^{n+sp}}\,dy \lesssim \frac{\rho^{-qs(p-1)-sp}}{1-s}\label{add-76yh4}
\end{equation}
and \begin{equation}
\int_{\left(B_{\rho}^{+}\setminus B_{w(x)}^{+}\right)\cap B_{\left|x\right|}(x)\cap B_{\frac{r}{2}}^c(x)} \frac{\left|h(\left|y+x\right|)-h(\left|x\right|)\right|^{p-1}-\left|h(\left|x\right|)-h(\left|x-y\right|)\right|^{p-1}}{\left|y\right|^{n+sp}}\,dy \lesssim \frac{\rho^{-qs(p-1)-sp}}{1-s}.\label{rhy654vf}
\end{equation}
Hence, from now on we focus on the proofs of these claims.

We first prove~\eqref{iure}. To do so,
if~$x\in B_{\frac{2}{3}r}\setminus \overline{B}_{\frac{r}{2}}$, we observe that, in light of~\eqref{yllwdcfere},
\begin{eqnarray*}
&& |f'(|x|)|^{p-2}\big(f''(|x+y|)+f''(|x|)\big)
\\&\leq& c_{p,m}^{p-2}\left(\frac{r}{2}\right)^{-(qs+k+1)(p-2)}
\left(|x|-\frac{r}{2}\right)^{k(p-2)} C_{p,m}
\left(\frac{r}6\right)^{-qs-k-1}
\left(
\left(|x+y|-\frac{r}{2}\right)^{k-1}+ \left(|x|-\frac{r}{2}\right)^{k-1}\right)\\
&\leq &
c_{p,m}^{p-2} C_{p,m}\left(\frac{r}{2}\right)^{-(qs+k+1)(p-2)}
\left(\frac{r}6\right)^{-qs-k-1}
{\tilde{x}}^{k(p-2)}
\big((\tilde{x}+|y|)^{k-1}+ {\tilde{x}}^{k-1}\big).
\end{eqnarray*}
As a consequence,
\begin{equation*}
\begin{split}
&\int_{B_{w(x)}^{+}\cap B_{\left|x\right|}(x)\cap B_{\frac{r}{2}}^c(x)}\frac{\left|f'(\left|x\right|)\right|^{p-2}\left(f''(\left|x+y\right|)+f''(\left|x\right|)\right)}{\left|y\right|^{n+sp-p}}\,dy\\
\leq \;& c_{p,m}^{p-2}C_{p,m}\left(\frac{r}{2}\right)^{-(qs+k+1)(p-2)}\left(\frac{r}{6}\right)^{-qs-k-1}\int_{B_{w(x)}^{+}\cap B_{\left|x\right|}(x)\cap B_{\frac{r}{2}}^c(x)}\frac{\tilde{x}^{k(p-2)}\left((\tilde{x}+\left|y\right|)^{k-1}+\tilde{x}^{k-1}\right)}{\left|y\right|^{n+sp-p}}\,dy\\
\lesssim \;& r^{-(p-1)(qs+k+1)}\int_{B_{w(x)}}\frac{\tilde{x}^{k(p-2)}\left((\tilde{x}+\left|y\right|)^{k-1}+\tilde{x}^{k-1}\right)}{\left|y\right|^{n+sp-p}}\,dy
.\end{split}\end{equation*}
Thus, changing variable~$z:=y/\tilde{x}$, we obtain that
\begin{equation*}
\begin{split}
&\int_{B_{w(x)}^{+}\cap B_{\left|x\right|}(x)\cap B_{\frac{r}{2}}^c(x)}\frac{\left|f'(\left|x\right|)\right|^{p-2}\left(f''(\left|x+y\right|)+f''(\left|x\right|)\right)}{\left|y\right|^{n+sp-p}}\,dy\\
\lesssim \;&
r^{-(p-1)(qs+k+1)}\tilde{x}^{(k+1)(p-1)-sp}
\int_{B_{w(x)/\tilde{x}}}\frac{\left((1+|z|)^{k-1}+1\right)}{|z|^{n+sp-p}}\,dz\\
\lesssim \;&
r^{-(p-1)(qs+k+1)}\tilde{x}^{(k+1)(p-1)-sp}\int_{B_{\alpha_{p,m}}}\frac{\left((1+|z|)^{k-1}+1\right)}{|z|^{n+sp-p}}\,dz\\
\lesssim & \frac{r^{-(p-1)(qs+k+1)}}{p-sp}\tilde{x}^{(k+1)(p-1)-sp}.
\end{split}\end{equation*}

Now, we observe that~$(k+1)(p-1)-sp\ge0$
thanks to~\eqref{oler-0988776}, and therefore we deduce from~\eqref{rho-tildex.kju} that
\begin{equation*}
\begin{split}&
\int_{B_{w(x)}^{+}\cap B_{\left|x\right|}(x)\cap B_{\frac{r}{2}}^c(x)}\frac{\left|f'(\left|x\right|)\right|^{p-2}\left(f''(\left|x+y\right|)+f''(\left|x\right|)\right)}{\left|y\right|^{n+sp-p}}\,dy
\\&\qquad\qquad\lesssim  \frac{r^{-(p-1)(qs+k+1)}}{1-s}\rho^{(k+1)(p-1)-sp}
\leq  \frac{\rho^{-qs(p-1)-sp}}{1-s}.
\end{split}
\end{equation*}
This proves~\eqref{iure}
when~$x\in B_{\frac{2}{3}r}\setminus \overline{B}_{\frac{r}{2}}$.

If instead~$x\in B_{\hat{r}}\setminus \overline{B}_{\frac{2}{3}r}$, 
we exploit~\eqref{p124353} to obtain the estimate
\begin{equation}\label{poiuytrelkjhgfd0987654}\begin{split}
&\left|f'(x)\right|^{p-2}\left(f''(x+y)+f''(x)\right)\\ \leq\;&
{\tilde{c}_{p,m}}^{p-2}\left(\frac{r-|x|}{2}\right)^{(qs+1)(2-p)}
C_{p,m}\left( \left(\frac{r-|x+y|}{2}\right)^{-qs-2}
+\left(\frac{r-|x|}{2}\right)^{-qs-2}
\right).
\end{split}\end{equation}
Also, if~$y\in B_{w(x)}$
then~$|x+y|\le |x|+|y|\le |x|+w(x)
\le |x|+\rho$, thanks to~\eqref{papapsdfre}. 
Consequently,
\begin{equation}\label{05689dgjwahsavfyewt7834}
\frac{r-|x+y|}{2}\ge \frac{r-|x|-\rho}{2}=
\rho-\frac\rho{2}=\frac\rho{2}.
\end{equation}
{F}rom this and~\eqref{poiuytrelkjhgfd0987654}, we infer that
\begin{eqnarray*}
\left|f'(x)\right|^{p-2}\left(f''(x+y)+f''(x)\right)\lesssim
\rho^{(qs+1)(2-p)}\rho^{-qs-2}=\rho^{-qs(p-1)-p} 
.
\end{eqnarray*}
As a result,
\begin{equation*}
\begin{split}
&\int_{B_{w(x)}^{+}\cap B_{\left|x\right|}(x)\cap B_{\frac{r}{2}}^c(x)} \frac{\left|f'(x)\right|^{p-2}\left(f''(x+y)+f''(x)\right)}{\left|y\right|^{n+sp-p}}\,dy 
\lesssim   \rho^{-qs(p-1)-p} \int_{B_{w(x)}} \frac{dy}{\left|y\right|^{n+sp-p}}\\
&\qquad\qquad \leq  \rho^{-qs(p-1)-p} \int_{B_{\rho}} \frac{dy}{\left|y\right|^{n+sp-p}}
\lesssim  \frac{\rho^{-qs(p-1)-sp}}{1-s}.
\end{split}
\end{equation*}
This concludes the proof of~\eqref{iure}. 

Next, we show that~\eqref{spoart} holds true. For this, we observe that, as a consequence of~\eqref{yllwdcfere} and~\eqref{tildecpm}, if~$x\in B_{\frac{2}{3}r}\setminus \overline{B}_{\frac{r}{2}}$ and~$y\in B_{w(x)}$, then 
\begin{equation*}
\begin{split}
f'(\left|x\right|)-f''(\left|x\right|)\frac{\left|y\right|}{2}&\geq c_{p,m}\left(\frac{r}{2}\right)^{-qs-k-1}\tilde{x}^{k}-C_{p,m}\left(\frac{r}{6}\right)^{-qs-k-1} \tilde{x}^{k-1}\frac{\left|y\right|}{2}\\
&\geq  c_{p,m}\left(\frac{r}{2}\right)^{-qs-k-1}\tilde{x}^{k}-C_{p,m}\left(\frac{r}{6}\right)^{-qs-k-1} \tilde{x}^{k-1}\frac{w(x)}{2}\\
&\geq c_{p,m}\left(\frac{r}{2}\right)^{-qs-k-1}\tilde{x}^{k}-\frac{\tilde{c}_{p,m}}{2}\left(\frac{r}{6}\right)^{-qs-k-1} \tilde{x}^{k}\\ 
&=c_{p,m}\left(\frac{r}{2}\right)^{-qs-k-1}\tilde{x}^{k}-\frac{c_{p,m}}{2}\left(\frac{1}{3}\right)^{q+\left[\frac{p}{p-1}\right]+1}\left(\frac{r}{6}\right)^{-qs-k-1} \tilde{x}^{k}\\ 
& \geq c_{p,m}\left(\frac{r}{2}\right)^{-qs-k-1}\tilde{x}^{k}-\frac{c_{p,m}}{2}\left(\frac{1}{3}\right)^{qs+k+1}\left(\frac{r}{6}\right)^{-qs-k-1} \tilde{x}^{k}\\ 
&=\frac{c_{p,m}}{2}\left(\frac{r}{2}\right)^{-qs-k-1}\tilde{x}^{k}.
\end{split}
\end{equation*}
Thus, making again use of~\eqref{yllwdcfere},
\begin{eqnarray*}
&&f'(|x|)^{p-3} f''(|x+y|)^2+
\left|f'(|x|)-f''(|x|)
\frac{|y|}{2}\right|^{p-3}f''(|x|)^2
\\ &\leq &
c_{p,m}^{p-3}\left(\frac{r}{2}\right)^{-(qs+k+1)(p-3)}
\left(|x|-\frac{r}{2}\right)^{k(p-3)}
C_{p,m}^2\left(\frac{r}6\right)^{-2(qs+k+1)}\left(|x+y|-\frac{r}{2}\right)^{2(k-1)}
\\&&\qquad +
\frac{c_{p,m}^{p-3}}{2^{p-3}}\left(\frac{r}{2}\right)^{-(qs+k+1)(p-3)}
\tilde{x}^{k(p-3)}
C_{p,m}^2\left(\frac{r}6\right)^{-2(qs+k+1)}\left(|x|-\frac{r}{2}\right)^{2(k-1)}\\
&\lesssim &
r^{-(p-1)(qs+k+1)} {\tilde{x}}^{k(p-3)}
\left( \left(|x+y|-\frac{r}{2}\right)^{2(k-1)} +
{\tilde{x}}^{2(k-1)}\right)
\\&\leq & r^{-(p-1)(qs+k+1)} {\tilde{x}}^{k(p-3)}
\big( (\tilde{x}+|y|)^{2(k-1)} +
{\tilde{x}}^{2(k-1)}\big).
\end{eqnarray*}
As a consequence,
using also~\eqref{rho-tildex.kju}, we infer that
\begin{equation*}
\begin{split}
&\int_{B_{w(x)}^{+}\cap B_{\left|x\right|}(x)\cap B_{\frac{r}{2}}^c(x)}\frac{f'(\left|x\right|)^{p-3}f''(\left|x+y\right|)^2+\left|f'(\left|x\right|)-f''(\left|x\right|)\frac{\left|y\right|}{2}\right|^{p-3}f''(\left|x\right|)^2}{\left|y\right|^{n+sp-p-1}}\,dy\\
\lesssim \;& r^{-(p-1)(qs+k+1)}\tilde{x}^{k(p-3)}
\left( \int_{B_{w(x)}}\frac{ \left(\tilde{x}+\left|y\right|\right)^{2(k-1)}}{|y|^{n+sp-p-1}}\,dy +\int_{B_{w(x)}}\frac{\tilde{x}^{2(k-1)}}{|y|^{n+sp-p-1}}\,dy\right)\\
\leq \;& r^{-(p-1)(qs+k+1)} \tilde{x}^{(k+1)(p-1)-sp}\left( \int_{B_{\alpha_{p,m}}}\frac{ \left(1+\left|y\right|\right)^{2(k-1)}}{\left|y\right|^{n+sp-p-1}}\,dy +\int_{B_{\alpha_{p,m}}}\frac{dy}{\left|y\right|^{n+sp-p-1}}\right)\\ 
\leq \;& r^{-(p-1)(qs+k+1)} \rho^{(k+1)(p-1)-sp} \int_{B_{\alpha_{p,m}}}\frac{dy}{\left|y\right|^{n+sp-p-1}}\\
\lesssim \;& \frac{\rho^{-qs(p-1)-sp}}{1+p-sp}\\
\leq \;& \frac{\rho^{-qs(p-1)-sp}}{1-s},
\end{split}
\end{equation*}
which establishes~\eqref{spoart}
when~$x\in B_{\frac{2}{3}r}\setminus \overline{B}_{\frac{r}{2}}$,

Similarly, using~\eqref{p124353}, we find that, for all~$x\in B_{\hat{r}}\setminus B_{\frac{2}{3}r}$ and~$y\in B_{w(x)}$,
\begin{equation*}
\begin{split}
f'(\left|x\right|)-f''(\left|x\right|)\frac{\left|y\right|}{2}&\geq f'(\left|x\right|)-f''(\left|x\right|)\frac{w(x)}{2}\\
&=f'(\left|x\right|)-f''(\left|x\right|)\frac{\alpha_{p,m} \rho}{2}\\
&\geq \tilde{c}_{p,m} \rho^{-qs-1}-C_{p,m}\rho^{-qs-1}\frac{\alpha_{p,m}}{2}\\
&=\rho^{-qs-1}\frac{\tilde{c}_{p,m}}{2}. 
\end{split}
\end{equation*}
This and~\eqref{p124353} give that
\begin{eqnarray*}
&&f'(|x|)^{p-3}f''(|x+y|)^2
+\left|f'(|x|)-f''(|x|)\frac{|y|}{2}\right|^{p-3}f''(|x|)^2
\\
&\leq&
\tilde{c}_{p,m}^{p-3}\rho^{-(qs+1)(p-3)}
C_{p,m}^2\left(  \frac{r-|x+y|}2 \right)^{-2(qs+2)}
+
\rho^{-(qs+1)(p-3)}\frac{\tilde{c}_{p,m}^{p-3}}{2^{p-3}}
C_{p,m}^2\rho^{-2(qs+2)}
\end{eqnarray*} 
Hence, recalling~\eqref{05689dgjwahsavfyewt7834},
\begin{eqnarray*}
&&f'(|x|)^{p-3}f''(|x+y|)^2
+\left|f'(|x|)-f''(|x|)\frac{|y|}{2}\right|^{p-3}f''(|x|)^2
\lesssim
\rho^{-(qs+1)(p-3)-2(qs+2)}.
\end{eqnarray*}
{F}rom this we thereby find that
\begin{equation*}
\begin{split}
&\int_{B_{w(x)}^{+}\cap B_{\left|x\right|}(x)\cap B_{\frac{r}{2}}^c(x)}\frac{f'(\left|x\right|)^{p-3}f''(\left|x+y\right|)^2+\left|f'(\left|x\right|)-f''(\left|x\right|)\frac{\left|y\right|}{2}\right|^{p-3}f''(\left|x\right|)^2}{|y|^{n+sp-p-1}}\,dy\\ &\qquad
\lesssim   \rho^{(qs+1)(1-p)}\int_{B_{w(x)}} \frac{dy}{\left|y\right|^{n+sp-p-1}}\\
&\qquad=  \rho^{(qs+1)(1-p)}\rho^{-sp+p-1}\int_{B_{\alpha_{p,m}}}\frac{dy}{\left|y\right|^{n+sp-p-1}}
\lesssim \frac{\rho^{-qs(p-1)-sp}}{1+p-sp}
\leq \frac{\rho^{-qs(p-1)-sp}}{1-s}.
\end{split}
\end{equation*}
Therefore the proof of~\eqref{spoart} is complete.

Now we prove~\eqref{add-76yh4}. To accomplish this goal, we prove first that there exists~$C_p\in(0,+\infty)$ such that, for every~$x\in B_{\hat{r}}$ and~$y\in B_{\frac{w(x)}{\left|x\right|}}^{+}\cap B_{1}(x/\left|x\right|)\cap B_{\frac{r}{2\left|x\right|}}^c(x/\left|x\right|)$,
\begin{equation}\label{Krema}
\left(\left|\frac{x}{\left|x\right|}+y\right|-1\right)^{p-1}-\left(1-\left|y-\frac{x}{\left|x\right|}\right|\right)^{p-1}\leq C_p \left|y\right|^p.
\end{equation}
To show the claim in~\eqref{Krema}, we assume without loss of generality that~$x=\left|x\right|e_n$. Then, if we write~$y=(y',y_n)$ with~$y':=(y_1,\dots, y_{n-1})$, we have that 
\begin{equation*}
\frac{d}{dy_n}\left|\frac{x}{\left|x\right|}+y\right|=\frac{d}{dy_n}\sqrt{\left|y'\right|^2+y_n^2+2y_n+1}=\frac{1+y_n}{\sqrt{\left|y'\right|^2+y_n^2+2y_n+1}}
\end{equation*}
and similarly
\begin{equation}\label{delo}
\begin{split}
\frac{d^2}{dy_n^2}\left|\frac{x}{\left|x\right|}+y\right| &=\frac{(\left|y'\right|^2+y_n^2+2y_n+1)-(1+y_n)^2}{\left(\left|y'\right|^2+y_n^2+2y_n+1\right)^{\frac{3}{2}}}\\
&=\frac{\left|y'\right|^2}{\left(\left|y'\right|^2+y_n^2+2y_n+1\right)^{\frac{3}{2}}}.
\end{split}
\end{equation}
We also compute that 
\begin{equation*}
\frac{d}{dy_n}\left|y-\frac{x}{\left|x\right|}\right|   =\frac{d}{dy_n}\sqrt{\left|y'\right|^2+y_n^2-2y_n+1}=\frac{y_n-1}{\sqrt{\left|y'\right|^2+y_n^2-2y_n+1}} 
\end{equation*}
and 
\begin{equation}\label{golo}
\frac{d^2}{dy_n^2}\left|y-\frac{x}{\left|x\right|}\right|=\frac{\left|y'\right|^2}{\left(\left|y'\right|^2+y_n^2-2y_n+1\right)^{\frac{3}{2}}}.
\end{equation}

Now we denote by~$\pi_n:\R^n\to \R$ the projection along the~$n$-th axis
and we set
$$A:=\pi_n \left(B_{\frac{w(x)}{\left|x\right|}}^{+}\cap B_{1}(x/\left|x\right|)\cap B_{\frac{r}{2\left|x\right|}}^c(x/\left|x\right|)\right)
.$$ Then, in light of~\eqref{delo} and~\eqref{golo},
\begin{equation*}
\sup_{y_n\in A}\left|\frac{d^2}{dy_n^2}\left|\frac{x}{\left|x\right|}+y\right|\right|\leq \frac{\left|y'\right|^2}{\left(\left|y'\right|^2+1\right)^{\frac{3}{2}}}\leq 1.
\end{equation*}
Also, since~$x\in B_{\hat{r}}$,
\begin{equation*}
\sup_{y_n\in A}\left|\frac{d^2}{dy_n^2}\left|y-\frac{x}{\left|x\right|}\right|\right|\leq\frac{1}{\left(\left|y'\right|^2+y_n^2-2y_n+1\right)^{\frac{1}{2}}}\leq \frac{2\left|x\right|}{r}\leq \frac{2\hat{r}}{r}\leq 2.
\end{equation*}
By Taylor's Theorem we thereby deduce that
\begin{equation}\label{labdtcbbbbre}
\left|\frac{x}{\left|x\right|}+y\right|-1=y_n+\frac{y_n^2}{2}\frac{d^2}{d\xi^2}\left|\frac{x}{\left|x\right|}-(y',\xi)\right|\leq y_n+\frac{1}{2}y_n^2 
\end{equation}
and that
\begin{equation}\label{clotebdgt}
1-\left|y-\frac{x}{\left|x\right|}\right|=y_n-\frac{y_n^2}{2} \frac{d^2}{d\eta^2}\left|(y',\eta)-\frac{x}{\left|x\right|}\right| \geq y_n-y_n^2,
\end{equation}for some~$\eta$, $\xi \in (0,y_n)$.

Furthermore, by~\eqref{papapsdfre}, we have that, for every~$y\in B_{\frac{w(x)}{\left|x\right|}}^{+}\cap B_{1}(x/\left|x\right|)\cap B_{\frac{r}{2\left|x\right|}}^c(x/\left|x\right|)$,
\begin{equation}\label{cfvd-nbvrt5vcf}
(1-y_n)\geq (1-\left|y\right|)\geq \left(1-\frac{w(x)}{\left|x\right|}\right)\geq \left(1-\frac{\min \left\lbrace \tilde{x},\rho\right\rbrace}{\left|x\right|}\right)\geq \frac{1}{2}. 
\end{equation}
Therefore, thanks to~\eqref{labdtcbbbbre}, \eqref{clotebdgt} and~\eqref{cfvd-nbvrt5vcf}, and applying Taylor's Theorem once again, we infer that, for every~$y\in B_{\frac{w(x)}{\left|x\right|}}^{+}\cap B_{1}(x/\left|x\right|)\cap B_{\frac{r}{2\left|x\right|}}^c(x/\left|x\right|)$, there exist some~$\xi$, $\eta\in (0,y_n)$ such that 
\begin{equation*}
\begin{split}
&\left(\left|\frac{x}{|x|}+y\right|-1\right)^{p-1}-\left(1-\left|\frac{x}{|x|}-y\right|\right)^{p-1}\\
=\;& \left(\sqrt{\left|y'\right|^2+y_n^2+2y_n+1}-1\right)^{p-1}-\left(1-\sqrt{\left|y'\right|^2+y_n^2-2y_n+1}\right)^{p-1}\\
\leq\; & y_n^{p-1}\left(\left(1+\frac{1}{2} y_n\right)^{p-1}-\left(1-y_n\right)^{p-1}\right)\\
\leq\; & y_n^{p-1}\left( 1+\frac{p-1}{2}\left(1+\frac{1}{2}\xi\right)^{p-2}y_n-1+(p-1)\left(1-\eta\right)^{p-2}y_n\right)\\
\leq \;& y_n^{p-1}\left(\frac{p-1}{2}y_n +(p-1)\left(1-y_n\right)^{p-2}y_n\right)\\ 
\leq \;& y_n^{p-1}\left(\frac{p-1}{2}y_n +(p-1)2^{2-p}y_n\right)\\ 
\leq\; &C_p y_n^p\\
\leq \;&C_{p}\left|y\right|^p
\end{split}
\end{equation*}
where we have defined~$C_p:=(p-1)\left(2^{-1}+2^{2-p}\right)$. This concludes the proof of~\eqref{Krema}.

As a consequence of a change of variable and~\eqref{Krema}, if~$x\in B_{\frac{2}{3}r}\setminus \overline{B}_{\frac{r}{2}}$,
\begin{eqnarray*}
&&\int_{B_{w(x)}^{+}\cap B_{\left|x\right|}(x)\cap B_{\frac{r}{2}}^c(x)} \frac{\left(\left(\left|x+y\right|-\left|x\right|\right)^{p-1}- \left(\left|x\right|-\left|x-y\right|\right)^{p-1}\right)\left|f'(\left|x\right|)\right|^{p-1}}{\left|y\right|^{n+sp}}\,dy\\
&=& |x|^{p-1-sp}\int_{B_{\frac{w(x)}{|x|}}^{+}\cap B_{1}(x/|x|)\cap B_{\frac{r}{2|x|}}^c(x/|x|)} 
\frac{((|x/|x|+y|-1)^{p-1}- (1-|x/|x|-y|)^{p-1})|f'(|x|)|^{p-1}}{|y|^{n+sp}}\,dy\\
&\leq &C_p |x|^{p-1-sp}\int_{B_{\frac{w(x)}{|x|}}^{+}\cap B_{1}(x/|x|)\cap B_{\frac{r}{2|x|}}^c(x/|x|)}
\frac{|f'(|x|)|^{p-1}}{|y|^{n+sp-p}}\,dy.
\end{eqnarray*}
{F}rom this and~\eqref{yllwdcfere} we thus find that
\begin{eqnarray*}
&&\int_{B_{w(x)}^{+}\cap B_{\left|x\right|}(x)\cap B_{\frac{r}{2}}^c(x)} \frac{\left(\left(\left|x+y\right|-\left|x\right|\right)^{p-1}- \left(\left|x\right|-\left|x-y\right|\right)^{p-1}\right)\left|f'(\left|x\right|)\right|^{p-1}}{\left|y\right|^{n+sp}}\,dy\\
&\leq &C_p C_{p,m}^{p-1}|x|^{p-1-sp}\left(\frac{r}{6}\right)^{-(p-1)(qs+k+1)}
\left(|x|-\frac{r}{2}\right)^{k(p-1)}
\int_{B_{\frac{w(x)}{|x|}}^{+}\cap B_{1}(x/|x|)\cap B_{\frac{r}{2|x|}}^c(x/|x|)}\frac{dy}{|y|^{n+sp-p}}\\
&\lesssim &
|x|^{p-1-sp}r^{-(p-1)(qs+k+1)}
{\tilde{x}}^{k(p-1)}
\int_{B_{\frac{w(x)}{|x|}}^{+}\cap B_{1}(x/|x|)\cap B_{\frac{r}{2|x|}}^c(x/|x|)}\frac{dy}{|y|^{n+sp-p}}.
\end{eqnarray*}

Now, we observe that~$w(x)\le \tilde{x}\le |x|$, thanks to~\eqref{fyuedwiqryeuwiyuwdtilde} and~\eqref{papapsdfre}. Therefore,
\begin{eqnarray*}
&&\int_{B_{w(x)}^{+}\cap B_{\left|x\right|}(x)\cap B_{\frac{r}{2}}^c(x)} \frac{\left(\left(\left|x+y\right|-\left|x\right|\right)^{p-1}- \left(\left|x\right|-\left|x-y\right|\right)^{p-1}\right)\left|f'(\left|x\right|)\right|^{p-1}}{\left|y\right|^{n+sp}}\,dy\\
&\lesssim &
|x|^{p-1-sp}r^{-(p-1)(qs+k+1)}
{\tilde{x}}^{k(p-1)}
\int_{B_{1}}\frac{dy}{|y|^{n+sp-p}}\\
&\lesssim & \frac{|x|^{p-1-sp}r^{-(p-1)(qs+k+1)}
{\tilde{x}}^{k(p-1)} }{p-sp}
.\end{eqnarray*}
Accordingly,
recalling also~\eqref{rho-tildex.kju},
\begin{equation*}
\begin{split}
&\int_{B_{w(x)}^{+}\cap B_{\left|x\right|}(x)\cap B_{\frac{r}{2}}^c(x)} \frac{\left(\left(\left|x+y\right|-\left|x\right|\right)^{p-1}- \left(\left|x\right|-\left|x-y\right|\right)^{p-1}\right)\left|f'(\left|x\right|)\right|^{p-1}}{\left|y\right|^{n+sp}}\,dy\\
&\qquad\qquad \lesssim \frac{r^{-sp-(p-1)(qs+k)}\rho^{k(p-1)} }{p-sp}
\lesssim \frac{\rho^{-qs(p-1)-sp}}{1-s}
,\end{split}
\end{equation*}
which proves~\eqref{add-76yh4}
when~$x\in B_{\frac{2}{3}r}\setminus \overline{B}_{\frac{r}{2}}$.

If~$x\in B_{\hat{r}}\setminus B_{\frac{2}{3}r}$, we argue
in a similar way, exploiting a change of variable and~\eqref{Krema}, but
using now~\eqref{p124353} to estimate~$|f'(|x|)|$ and~\eqref{jnurefv540as}
to see that~$w(x)\le\rho$. In this way, we conclude that 
\begin{equation*}
\begin{split}
&\int_{B_{w(x)}^{+}\cap B_{\left|x\right|}(x)\cap B_{\frac{r}{2}}^c(x)} \frac{\left(\left(\left|x+y\right|-\left|x\right|\right)^{p-1}- \left(\left|x\right|-\left|x-y\right|\right)^{p-1}\right)\left|f'(\left|x\right|)\right|^{p-1}}{\left|y\right|^{n+sp}}\,dy\\
\lesssim \;& \left|x\right|^{p-1-sp} \int_{B_{\frac{w(x)}{\left|x\right|}}^{+}\cap B_{1}(x/\left|x\right|)\cap B_{\frac{r}{2\left|x\right|}}^c(x/\left|x\right|)} \frac{\left|f'(\left|x\right|)\right|^{p-1}}{\left|y\right|^{n+sp-p}}\,dy\\
\leq\;  &C_{p,m}^{p-1}\left|x\right|^{p-1-sp}\rho^{-(qs+1)(p-1)} \int_{B_{\frac{w(x)}{\left|x\right|}}} \frac{dy}{\left|y\right|^{n+sp-p}}\\
\lesssim \;&  \left|x\right|^{p-1-sp}\rho^{-(qs+1)(p-1)}\frac{\rho^{p-sp}}{\left|x\right|^{p-sp}}\frac{1}{p-sp}\\
\lesssim \;& \frac{\rho^{-qs(p-1)-sp}}{1-s}.
\end{split}
\end{equation*}
Thus, the claim in~\eqref{add-76yh4} holds true for every~$x\in B_{\hat{r}}\setminus \overline{B}_{\frac{r}{2}}$.

In order to complete the proof of~\eqref{CL1} it is only left to show~\eqref{rhy654vf}. To do so, we observe that, for every~$t>0$ and~$y\in \R^n\setminus  B_{t}$, 
\begin{equation}\label{bruio54c7.kju}
(1+\left|y\right|)^{k-1}-1\leq (1+\left|y\right|)^{k+1}\leq \left(t^{-1}+1\right)^{k+1}\left|y\right|^{k+1}.
\end{equation}
Moreover, we see that 
\begin{equation}\label{9806ygvcsar2wsdxc54rfgb7uj9ol}
\begin{split}
&\int_{\left(B_{\rho}^{+}\setminus B_{w(x)}^{+}\right)\cap B_{\left|x\right|}(x)\cap B_{\frac{r}{2}}^c(x)} \frac{\left|h(\left|y+x\right|)-h(\left|x\right|)\right|^{p-1}-\left|h(\left|x\right|)-h(\left|x-y\right|)\right|^{p-1}}{\left|y\right|^{n+sp}}\,dy\\
\leq & \int_{\left(B_{\rho}^{+}\setminus B_{w(x)}^{+}\right)\cap B_{\left|x\right|}(x)\cap B_{\frac{r}{2}}^c(x)} \frac{\left|h(\left|y+x\right|)-h(\left|x\right|)\right|^{p-1}}{\left|y\right|^{n+sp}}\,dy\\
\leq  & \int_{B_{\rho}^{+}\setminus B_{w(x)}^{+}} \frac{\left|h(\left|y+x\right|)-h(\left|x\right|)\right|^{p-1}}{\left|y\right|^{n+sp}}\,dy\\
\leq  & \int_{B_{\rho}^{+}\setminus B_{w(x)}^{+}} \frac{\left|f(\left|y+x\right|)-f(\left|x\right|)\right|^{p-1}}{\left|y\right|^{n+sp}}\,dy.
\end{split}\end{equation}

Now, if~$x\in B_{\frac{2}{3}r} \setminus  \overline{B}_{\frac{r}{2}}$
and~$y\in B_\rho$,
we have that~$\left(\tilde{x}+\frac{r}2, \tilde{x}+\frac{r}2+|y|\right) \subset \left(\frac{r}2, \frac{5r}6 \right)$. Therefore,
we employ~\eqref{yllwdcfere} with~$\rho_0:=r/6$ and we obtain that
\begin{equation}\label{8954tfuegdsvhfakfey}\begin{split}
& \left|f(\left|y+x\right|)-f(\left|x\right|)\right|
\le \left|f\left(|y|+\tilde{x} +\frac{r}2\right)-f\left(\tilde{x}+\frac{r}2\right)\right|
=\int_{\tilde{x}}^{\tilde{x}+|y|}f'\left(t+\frac{r}2\right)\,dt 
\\&\qquad
\leq  C_{p,m} \left(\frac{r}6\right)^{-qs-k-1} \int_{\tilde{x}}^{\tilde{x}+|y|} t^k\,dt    \lesssim r^{-qs-k-1} \big( (\tilde{x}+|y|)^{k+1} - \tilde{x}^{k+1}\big) .
\end{split}\end{equation}
Thus, if~$x\in B_{\frac{2}{3}r} \setminus  \overline{B}_{\frac{r}{2}}$,
we employ a change of variable to deduce that
\begin{equation*}
\begin{split}
&\int_{\left(B_{\rho}^{+}\setminus B_{w(x)}^{+}\right)\cap B_{\left|x\right|}(x)\cap B_{\frac{r}{2}}^c(x)} \frac{\left|h(\left|y+x\right|)-h(\left|x\right|)\right|^{p-1}-\left|h(\left|x\right|)-h(\left|x-y\right|)\right|^{p-1}}{\left|y\right|^{n+sp}}\,dy\\
\lesssim \;&
r^{-(p-1)(qs+k+1)}\int_{B_{\rho} \setminus B_{w(x)}} \frac{\left((\tilde{x}+\left|y\right|)^{k+1}-\tilde{x}^{k+1}\right)^{p-1}}{\left|y\right|^{n+sp}}\,dy\\
\lesssim\; &  {r}^{-(p-1)(qs+k+1)} \tilde{x}^{(k+1)(p-1)-sp} \int_{B_{{\rho}/{\tilde{x}}} \setminus B_{\alpha_{p,m}}} \frac{\left((1+\left|y\right|)^{k+1}-1\right)^{p-1}}{\left|y\right|^{n+sp}}\,dy.
\end{split}
\end{equation*}
Now, formula~\eqref{bruio54c7.kju} (used here with~$t:=\alpha_{p,m}$)
gives that
\begin{equation*}
\begin{split}
&\int_{\left(B_{\rho}^{+}\setminus B_{w(x)}^{+}\right)\cap B_{\left|x\right|}(x)\cap B_{\frac{r}{2}}^c(x)} \frac{\left|h(\left|y+x\right|)-h(\left|x\right|)\right|^{p-1}-\left|h(\left|x\right|)-h(\left|x-y\right|)\right|^{p-1}}{\left|y\right|^{n+sp}}\,dy\\
\lesssim \;& r^{-(p-1)(qs+k+1)}\tilde{x}^{(k+1)(p-1)-sp} \int_{B_{{\rho}/{\tilde{x}}} \setminus B_{\alpha_{p,m}}} \frac{dy}{\left|y\right|^{n+sp-(k+1)(p-1)}}\\
\leq\; & r^{-(p-1)(qs+k+1)}\tilde{x}^{(k+1)(p-1)-sp} \int_{B_{{\rho}/{\tilde{x}}}} \frac{dy}{\left|y\right|^{n+sp-(k+1)(p-1)}}\\
\lesssim\; & \frac{r^{-(p-1)(qs+k+1)}}{(k+1)(p-1)-sp}\rho^{(k+1)(p-1)-sp}.
\end{split}
\end{equation*}
Thus, recalling~\eqref{oler-0988776},
\begin{equation*}
\begin{split}
&\int_{\left(B_{\rho}^{+}\setminus B_{w(x)}^{+}\right)\cap B_{\left|x\right|}(x)\cap B_{\frac{r}{2}}^c(x)} \frac{\left|h(\left|y+x\right|)-h(\left|x\right|)\right|^{p-1}-\left|h(\left|x\right|)-h(\left|x-y\right|)\right|^{p-1}}{\left|y\right|^{n+sp}}\,dy\\
&\qquad\qquad\lesssim  \frac{\rho^{-qs(p-1)-sp}}{(k+1)(p-1)-sp}
\leq \frac{\rho^{-qs(p-1)-sp}}{1-s},
\end{split}
\end{equation*} which is~\eqref{rhy654vf}
when~$x\in B_{\frac{2}{3}r} \setminus  \overline{B}_{\frac{r}{2}}$.

Analogously, if~$x\in B_{\hat{r}}\setminus B_{\frac{2}{3}r}$
we use~\eqref{p124353}
into~\eqref{9806ygvcsar2wsdxc54rfgb7uj9ol}
to obtain that
\begin{equation*}
\begin{split}
&\int_{\left(B_{\rho}^{+}\setminus B_{w(x)}^{+}\right)\cap B_{\left|x\right|}(x)\cap B_{\frac{r}{2}}^c(x)} \frac{\left|h(\left|y+x\right|)-h(\left|x\right|)\right|^{p-1}-\left|h(\left|x\right|)-h(\left|x-y\right|)\right|^{p-1}}{\left|y\right|^{n+sp}}\,dy\\
&\qquad\qquad\leq C_{p,m}^{p-1}\int_{B_{\rho}^{+}\setminus B_{w(x)}^{+}} \frac{\rho^{-(qs+1)(p-1)}}{\left|y\right|^{n+sp-p+1}}\,dy.
\end{split}
\end{equation*}
Hence, changing variable and recalling~\eqref{gallgbffer430},
\begin{equation*}
\begin{split}
&\int_{\left(B_{\rho}^{+}\setminus B_{w(x)}^{+}\right)\cap B_{\left|x\right|}(x)\cap B_{\frac{r}{2}}^c(x)} \frac{\left|h(\left|y+x\right|)-h(\left|x\right|)\right|^{p-1}-\left|h(\left|x\right|)-h(\left|x-y\right|)\right|^{p-1}}{\left|y\right|^{n+sp}}\,dy\\
=\;& C_{p,m}^{p-1}\rho^{-(qs+1)(p-1)-sp+p-1}
\int_{B_{1}^{+}\setminus B_{\alpha_{p,m}}^{+}} \frac{dy}{\left|y\right|^{n+sp-p+1}}\\
\lesssim \;&  \rho^{-qs(p-1)-sp}\begin{dcases}
\frac{{\alpha_{p,m}}^{p-sp-1}-1}{1+sp-p}\quad &\mbox{if}\quad s\in \left(\frac{p-1}{p},1\right),\\
\left|\ln\left(\alpha_{p,m}\right)\right|\quad &\mbox{if}\quad s=\frac{p-1}{p},\\
\frac{1-{\alpha_{p,m}}^{p-sp-1}}{p-1-sp} \quad &\mbox{if}\quad s\in \left(\frac{p-1}{2p},\frac{p-1}{p}\right)
\end{dcases}\\
\lesssim \;&\frac{\rho^{-qs(p-1)-sp}}{1-s}.
\end{split}
\end{equation*}
This concludes the proof of~\eqref{rhy654vf} and thus of the claim in~\eqref{CL1}.
\medskip

\underline{Proof of~\eqref{CL2}}:
Toward this objective, 
we use the notation for~$\tilde{x}$ in~\eqref{fyuedwiqryeuwiyuwdtilde} and
we observe that,
for every~$y\in B_{\frac{r}{2\tilde{x}}}\left(\frac{x}{\tilde{x}}\right)$,
\begin{equation}\label{reflop754r}
(\left|y\right|+1)^{k+1}-1\leq (\left|y\right|+1)^{k+1}\leq 2^{k+1}\left|y\right|^{k+1}.
\end{equation}
Indeed, the first inequality is obvious.
For the second one, we use the triangular inequality to obtain that,
if~$y\in B_{\frac{r}{2\tilde{x}}}\left(\frac{x}{\tilde{x}}\right)$,
\begin{equation*}
\frac{|x|}{\tilde{x}}-\left|y\right|\leq \left|y-\frac{x}{\tilde{x}}\right|\leq \frac{r}{2\tilde{x}},
\end{equation*}
from which it follows that 
\begin{equation*}
\left|y\right|\geq \left(\left|x\right|-\frac{r}{2}\right)\frac{1}{\tilde{x}}=1,
\end{equation*}
and therefore the proof of~\eqref{reflop754r} is complete.

Moreover, if~$x\in B_{\frac{2}{3}r} \setminus \overline{B}_{\frac{r}{2}}$ and~$y\in B_\rho$, we recall~\eqref{8954tfuegdsvhfakfey} and we see that
\begin{eqnarray*}&&
\left|h(\left|y+x\right|)-h(\left|x\right|)\right|^{p-1}-\left|h(\left|x\right|)-h(\left|x-y\right|)\right|^{p-1}\le \left|h(\left|y+x\right|)-h(\left|x\right|)\right|^{p-1}\\
&&\qquad \le \left|f(\left|y+x\right|)-f(\left|x\right|)\right|^{p-1}
\lesssim r^{-(qs+k+1)(p-1)} \big( (\tilde{x}+|y|)^{k+1} - \tilde{x}^{k+1}\big)^{p-1},
\end{eqnarray*}
and thus a change of variable gives that
\begin{equation*}
\begin{split}
&\int_{B_{\rho}^{+}\cap B_{\frac{r}{2}}(x)} \frac{\left|h(\left|y+x\right|)-h(\left|x\right|)\right|^{p-1}-\left|h(\left|x\right|)-h(\left|x-y\right|)\right|^{p-1}}{\left|y\right|^{n+sp}}\,dy\\
\lesssim \;& r^{-(p-1)(qs+k+1)} \int_{B_{\rho}^{+}\cap B_{\frac{r}{2}}(x)}\frac{\big((|y|+\tilde{x})^{k+1}-\tilde{x}^{k+1}\big)^{p-1}}{|y|^{n+sp}}\,dy\\
= \;&  r^{-(p-1)(qs+k+1)}\tilde{x}^{(k+1)(p-1)-sp} \int_{B_{\frac{\rho}{\tilde{x}}}^{+}\cap B_{\frac{r}{2\tilde{x}}}\left(\frac{x}{\tilde{x}}\right)} \frac{\big((|y|+1)^{k+1}-1\big)^{p-1}}{|y|^{n+sp}}\,dy.\end{split}\end{equation*}
{F}rom this, and using also~\eqref{oler-0988776}, \eqref{gallgbffer430} and~\eqref{reflop754r} we obtain that
\begin{equation*}
\begin{split}
&\int_{B_{\rho}^{+}\cap B_{\frac{r}{2}}(x)} \frac{\left|h(\left|y+x\right|)-h(\left|x\right|)\right|^{p-1}-\left|h(\left|x\right|)-h(\left|x-y\right|)\right|^{p-1}}{\left|y\right|^{n+sp}}\,dy\\
\lesssim \;& r^{-(p-1)(qs+k+1)}\tilde{x}^{(k+1)(p-1)-sp} \int_{B_{\frac{\rho}{\tilde{x}}}}\frac{dy}{\left|y\right|^{n+sp-(k+1)(p-1)}}\,dy\\
\lesssim\; &\frac{r^{-(p-1)(qs+k+1)}}{(k+1)(p-1)-sp}\rho^{(k+1)(p-1)-sp} \\
\lesssim \;&\frac{\rho^{-qs(p-1)-sp}}{(k+1)(p-1)-sp}  \\
\leq\; & \frac{\rho^{-qs(p-1)-sp}}{1-s}. 
\end{split}
\end{equation*}
This establishes~\eqref{CL2} if~$x\in B_{\frac{2}{3}r} \setminus \overline{B}_{\frac{r}{2}}$.

If instead~$x\in B_{\hat r}\setminus B_{\frac{2}{3}r}$ we argue as follows.
We notice that
\begin{equation}\label{jvexzy} 
B_{\frac{r}{2}}(x)\subset \R^n\setminus B_{\tilde{x}},
\end{equation}
since, if~$y\in B_{\frac{r}{2}}(x)$ then, recalling~\eqref{fyuedwiqryeuwiyuwdtilde},
$$ |y|\ge |x|-|y-x|=\tilde{x}+\frac{r}2-|y-x|\ge \tilde{x}.$$

Thus, according to~\eqref{jnurefv540as} and~\eqref{jvexzy}, we have that
\begin{equation*}
B_{\rho}^{+}\cap B_{\frac{r}{2}}(x)\subset B_{\rho}^{+}\cap  B_{\tilde{x}}^c=\varnothing.
\end{equation*}
As a consequence,
\begin{equation*}
\int_{B_{\rho}^{+}\cap B_{|x|}(x)\cap B_{\frac{r}{2}}(x)} \frac{\left|h(\left|y+x\right|)-h(\left|x\right|)\right|^{p-1}-\left|h(\left|x\right|)-h(\left|x-y\right|)\right|^{p-1}}{\left|y\right|^{n+sp}}\,dy=0.
\end{equation*}
This concludes the proof of claim~\eqref{CL2}.
\medskip

\underline{Proof of~\eqref{CL3}}:
Making use of~\eqref{ewlop545454} we see that, for every~$y\in B_{\rho}$, 
\begin{equation*}
\begin{split}& |h(|y\pm x|)-h(|x|)|\le 
|f(|x\pm y|)-f(|x|)|\leq\int_{|x|}^{|x|+|y|}f'(t)\,dt
\leq f'(|x|+|y|)|y| \\
& \qquad \lesssim \left(r-|x|-|y|\right)^{-qs-1}|y|
=\left(2\rho-|y|\right)^{-qs-1}|y|
\leq \rho^{-qs-1}|y|.
\end{split}
\end{equation*}
Thanks to this, and also changing variable,
we find that
\begin{equation*}
\begin{split}
&\int_{B_{\rho}^{+}\cap B_{\left|x\right|}^c(x)} \frac{\left|h(\left|y+x\right|)-h(\left|x\right|)\right|^{p-1}+\left|h(\left|x-y\right|)-h(\left|x\right|)\right|^{p-1}}{\left|y\right|^{n+sp}}\,dy\\
\lesssim \;&  \rho^{-(p-1)(qs+1)} \int_{B_{\rho}^{+}\cap B_{\left|x\right|}^c(x)} \frac{dy}{\left|y\right|^{n+sp-p+1}}\\
= \;&  \rho^{-(p-1)(qs+1)}\left|x\right|^{p-1-sp} \int_{B_{\frac{\rho}{\left|x\right|}}^{+}\cap B_{1}^c\left(\frac{x}{\left|x\right|}\right)} \frac{dy}{\left|y\right|^{n+sp-p+1}}\\
\leq\; &  \rho^{-(p-1)(qs+1)}\left|x\right|^{p-1-sp} \int_{B_{\frac{\rho}{\left|x\right|}}^+\cap
B_{1}^c\left(\frac{x}{\left|x\right|}\right)
} \frac{dy}{\left|y'\right|^{n+sp-p+1}}.
\end{split}
\end{equation*}
We suppose, without loss of generality, that~$x=|x| e_n$, and so
\begin{eqnarray*}
&&\int_{B_{\rho}^{+}\cap B_{\left|x\right|}^c(x)} \frac{\left|h(\left|y+x\right|)-h(\left|x\right|)\right|^{p-1}+\left|h(\left|x-y\right|)-h(\left|x\right|)\right|^{p-1}}{\left|y\right|^{n+sp}}\,dy\\
&\lesssim &
\rho^{-(p-1)(qs+1)}\left|x\right|^{p-1-sp} \int_{B_{\frac{\rho}{\left|x\right|}}^+\cap B_{1}^c\left(e_n\right)
} \frac{dy}{\left|y'\right|^{n+sp-p+1}}
\end{eqnarray*}

Now, we remark that, if~$x\in B_{\hat r}\setminus \overline{B}_{\frac{r}{2}}$,
\begin{equation*}
\frac{\rho}{\left|x\right|}=\frac{r}{2\left|x\right|}-\frac{1}{2}\leq 1-\frac{1}{2}=\frac{1}{2}.
\end{equation*}   
As a consequence,  denoting by~$y=(y',y_n)\in\R^{n-1}\times\R$, we have that
$$ B_{\frac{\rho}{\left|x\right|}}^+\cap B_1^c(e_n)\subset
\left\lbrace y\in\R^n\;{\mbox{ s.t. }}\;
|y'|\leq \frac{\rho}{\left|x\right|}\;{\mbox{ and }}\;
y_n\in\left(0, 1-\sqrt{1-\left|y'\right|^2}\right)\right\rbrace
,$$
and therefore
\begin{equation*}
\begin{split}
&\int_{B_{\rho}^{+}\cap B_{\left|x\right|}^c(x)} \frac{\left|h(\left|y+x\right|)-h(\left|x\right|)\right|^{p-1}+\left|h(\left|x-y\right|)-h(\left|x\right|)\right|^{p-1}}{\left|y\right|^{n+sp}}\,dy\\
\lesssim \;&  \rho^{-(p-1)(qs+1)}\left|x\right|^{p-1-sp} \int_{
\left\lbrace\left|y'\right|\leq \frac{\rho}{\left|x\right|}\right\rbrace}
\frac{1}{\left|y'\right|^{n+sp-p+1}}
\left(\int_{0}^{1-\sqrt{1-\left|y'\right|^2}} dy_n\right)
\,dy'\\
= \;& \rho^{-(p-1)(qs+1)}\left|x\right|^{p-1-sp} \int_{\left\lbrace\left|y'\right|\leq \frac{\rho}{\left|x\right|}\right\rbrace}\frac{1-\sqrt{1-\left|y'\right|^2}}{\left|y'\right|^{n+sp-p+1}}\,dy'\\
\leq \;& \rho^{-(p-1)(qs+1)}\left|x\right|^{p-1-sp} \int_{\left\lbrace\left|y'\right|\leq \frac{\rho}{\left|x\right|}\right\rbrace}\frac{2}{\sqrt{3}\left|y'\right|^{n+sp-p-1}}\,dy'\\
\lesssim \;& \rho^{-(p-1)(qs+1)}\left|x\right|^{p-1-sp}\frac{\rho^{p-sp}}{\left|x\right|^{p-sp}}\frac{1}{p-sp}\\
\lesssim \;& \frac{\rho^{-qs(p-1)-sp}}{1-s}. 
\end{split}
\end{equation*}
This concludes the proof of claim~\eqref{CL3}. 
\medskip

\noindent{\em{ {\bf{(iii) Case $ x\in B_{r}\setminus B_{\hat{r}}$.}}}}
In this case, for every~$y\in \R^n$ it holds that 
\begin{equation*}
v(y)-v(x)=v(y)-1\leq 0.
\end{equation*}
{F}rom this, we deduce that
\begin{equation*}
\int_{B_{\rho}(x)}\frac{\left(v(y)-v(x)\right)\left|v(y)-v(x)\right|^{p-2}}{\left|x-y\right|^{n+sp}}\,dy\leq 0,
\end{equation*}
and~\eqref{frl54t4t409} trivially follows.
\end{proof}

We now provide an estimate on~$r-\hat{r}$, where~$\hat r$
has been defined in~\eqref{g874rfegt56}.

\begin{prop}\label{telviv}
Let~$p\in(1,+\infty)$, $m\in \left[p,+\infty\right)$, $s\in(0,1)$ and~$r\in(1,+\infty)$. Let~$f$ and~$\hat{r}$ be defined respectively as in~\eqref{vdefinition} and~\eqref{g874rfegt56}. 

Then, there exists~$\widehat{C}_{p,m} \in (0,+\infty)$ such that 
\begin{equation*}
r-\hat{r}>\widehat{C}_{p,m}. 
\end{equation*}
\end{prop}

\begin{proof}
{F}rom the computations in~\eqref{fr-fe-87}, we see that,
for every~$t\in \left(\frac{r}{2},r\right)$,
\begin{equation*}
f(t)\ge  (r-t)^{-qs}-c_{k,q,s}\left(\frac{r}{2}\right)^{-qs},
\end{equation*} where~$c_{k,q,s}$ is given in~\eqref{contevbdkop}.

Thus, if we define 
\begin{equation}
r_0:=r-\left(c_{k,q,s}\left(\frac{r}{2}\right)^{-qs}+1\right)^{-\frac{1}{qs}},
\end{equation}
we obtain that~$f(t)\geq 1$ for every~$t\in (r_0,r)$. 

Therefore, owing to~\eqref{g874rfegt56},  we obtain that~$\hat{r}<r_0$. 
This gives that 
\begin{equation*}
r-\hat{r}\geq \left(c_{k,q,s}\left(\frac{r}{2}\right)^{-qs}+1\right)^{-\frac{1}{qs}}\geq \big(c_{k,q,s}2^{qs}+1\big)^{-\frac{1}{qs}}=:C_{q,s,k}.
\end{equation*}
Hence, in order to complete the proof of Lemma~\ref{telviv}, it is only left to show that there exists~$\widehat{C}_{p,m}\in (0,+\infty)$ such that 
\begin{equation}\label{fbercjured}
C_{q,s,k}\geq \widehat{C}_{p,m}.
\end{equation}
To prove~\eqref{fbercjured}, we make use of~\eqref{soimate}
and we obtain that  
\begin{equation}\label{befdgcttt543}
C_{q,s,k}\geq  \left(qs c_{p,m}^{(1)} 2^{q}+1\right)^{-\frac{1}{qs}}
=\left(1-\frac{qsc_{p,m}^{(1)}2^{q}}{qsc_{p,m}^{(1)}2^{q}+1}\right)^{\frac{1}{qs}}. 
\end{equation} 

Now we consider the continuous function~$l:(0,1]\to (0,+\infty)$ defined by 
\begin{equation*}
l(s):=\left(1-\frac{qsc_{p,m}^{(1)}2^{q}}{qsc_{p,m}^{(1)}2^{q}+1}\right)^{\frac{1}{qs}}.
\end{equation*}
We set~$s_0:=\min \left\lbrace 1, \frac{1}{q 2^{q+1}c_{p,m}^{(1)}} \right\rbrace$
and we observe that, for every~$s\in (0,s_0)$,
\begin{equation*}
1-\frac{qsc_{p,m}^{(1)}2^{q}}{qsc_{p,m}^{(1)}2^{q}+1}\geq \frac{1}{2}
\end{equation*}
and therefore
\begin{equation*}
l(s)\geq \big(1-qsc_{p,m}^{(1)}2^{q}\big)^{\frac{1}{qs}}.
\end{equation*}
Taking the liminf as~$s\to 0$, we obtain that 
\begin{equation*}
\liminf_{s\to 0} l(s)\geq e^{-c_{p,m}^{(1)}2^{q}}.
\end{equation*}
{F}rom this, the fact that~$l\in C((0,1],(0,+\infty))$ and~\eqref{befdgcttt543}, we obtain~\eqref{fbercjured}, as desired.
\end{proof}

\begin{cor}\label{bfdclaop74r-09}
Let~$p\in (1,2)$, $s\in \left[\frac{p-1}{2p},1\right)$, $r\in (1,+\infty)$ and~$\mu \in (0,+\infty)$. Let~$v:\R^n\to [0,1]$ be defined as in~\eqref{vdefinition}. 

Then, there exists~$C_{\mu,n,p,m}\in (0,+\infty)$
such that, for every~$x\in B_r$,
\begin{equation}\label{marklettieri}
\int_{B_{\mu}(x)}\frac{\left(v(y)-v(x)\right)\left|v(y)-v(x)\right|^{p-2}}{\left|x-y\right|^{n+sp}}\,dy\leq \frac{C_{\mu,n,p,m}}{1-s}.
\end{equation}
\end{cor}

\begin{proof}
We recall that, if~$\hat{r}$ is given as in~\eqref{g874rfegt56}, then, for every~$x\in B_{r}\setminus B_{\hat{r}}$ and every~$y\in \R^n$,
\begin{equation*}
v(y)-v(x)=v(y)-1\leq 0.
\end{equation*}
Consequently, for every~$x\in B_{r}\setminus B_{\hat{r}}$ and~$\mu \in (0,+\infty)$,
\begin{equation*}
\int_{B_{\mu}(x)}\frac{\left(v(y)-v(x)\right)\left|v(y)-v(x)\right|^{p-2}}{\left|x-y\right|^{n+sp}}\,dy\leq 0,
\end{equation*}
and~\eqref{marklettieri} plainly follows.

Furthermore, in virtue of~\eqref{frl54t4t409} and Lemma~\ref{telviv},
if~$\rho:=\frac{r-\left|x\right|}{2}$ we have that,
for every~$x\in B_{\hat{r}}\setminus \overline{B}_{\frac{r}{2}}$,
\begin{equation*}
\begin{split}&
\int_{B_{\rho}(x)}\frac{\left(v(y)-v(x)\right)\left|v(y)-v(x)\right|^{p-2}}{\left|x-y\right|^{n+sp}}\,dy \leq C_{n,p,m}\frac{\rho^{-qs(p-1)-sp}}{p-sp}\\
&\qquad\qquad \leq  \frac{C_{n,p,m}}{p-sp}\left(\frac{r-\hat{r}}{2}\right)^{-qs(p-1)-sp}
\leq \frac{C_{n,p,m}}{p-sp}\left(\frac{\widehat{C}_{p,m}}{2}\right)^{-qs(p-1)-sp}.
\end{split}
\end{equation*}
Moreover, if~$\mu\in [\rho,+\infty)$,
\begin{eqnarray*}&&
\int_{B_{\mu}(x)\setminus B_{\rho}(x)}\frac{\left(v(y)-v(x)\right)\left|v(y)-v(x)\right|^{p-2}}{\left|x-y\right|^{n+sp}}\,dy 
\leq  \int_{B_{\mu}(x)\setminus B_{\rho}(x)}\frac{dy}{\left|x-y\right|^{n+sp}} \\ &&\qquad\qquad
\leq  \frac{\omega_{n-1}}{sp}\rho^{-sp}
\leq  \frac{1}{p-1}\frac{\omega_{n-1}}{p-sp}\rho^{-sp}
\leq  \frac{1}{p-1}\frac{\omega_{n-1}}{p-sp}\left(\frac{\widehat{C}_{p,m}}{2}\right)^{-sp}.
\end{eqnarray*}
{F}rom the last two displays we thus obtain that, if~$\mu\in [\rho,+\infty)$,
\begin{equation*}
\begin{split}
&\int_{B_{\mu}(x)}\frac{\left(v(y)-v(x)\right)\left|v(y)-v(x)\right|^{p-2}}{\left|x-y\right|^{n+sp}}\,dy\\
\leq\;  & \frac{C_{n,m,p}}{p-sp}\left(\frac{\widehat{C}_{p,m}}{2}\right)^{-qs(p-1)-sp} +\frac{1}{p-1}\frac{\omega_{n-1}}{p-sp}\left(\frac{\widehat{C}_{p,m}}{2}\right)^{-sp}
\end{split}
\end{equation*}
which gives the desired estimate.
  
Similarly, if~$\mu\in (0,\rho)$, 
\begin{equation*}
\begin{split}
&\int_{B_{\mu}(x)}\frac{\left(v(y)-v(x)\right)\left|v(y)-v(x)\right|^{p-2}}{\left|x-y\right|^{n+sp}}\,dy\\
=\;& \int_{B_{\rho}(x)}\frac{\left(v(y)-v(x)\right)\left|v(y)-v(x)\right|^{p-2}}{\left|x-y\right|^{n+sp}}\,dy - \int_{B_{\rho}(x)\setminus B_{\mu}(x)}\frac{\left(v(y)-v(x)\right)\left|v(y)-v(x)\right|^{p-2}}{\left|x-y\right|^{n+sp}}\,dy \\
\leq\; & \frac{C_{n,m,p}}{p-sp}\left(\frac{\widehat{C}_{p,m}}{2}\right)^{-qs(p-1)-sp}+2^{p-1}\int_{B_{\rho}(x)\setminus B_{\mu}(x)}\frac{dy}{\left|x-y\right|^{n+sp}} \\
\leq \; & \frac{C_{n,m,p}}{p-sp}\left(\frac{\widehat{C}_{p,m}}{2}\right)^{-qs(p-1)-sp} +\frac{2^{p-1}}{p-1}\frac{\omega_{n-1}}{p-sp}\mu^{-sp},
\end{split}
\end{equation*} as desired.
\end{proof}

\end{appendix}

\end{document}